\documentclass[reqno]{amsart}

\usepackage{xypic,hyperref,todonotes,comment}
\usepackage{graphics,amssymb}
\usepackage{pagecolor,lipsum}
\usepackage{tikz}
\usepackage{latexsym}
\usepackage{mathrsfs}
\usepackage[normalem]{ulem}
\usepackage{hyperref}
\hypersetup{ 
colorlinks,
citecolor=cyan,
filecolor=cyan,
linkcolor=cyan,
urlcolor=cyan
}

\xyoption{curve}

\newtheorem{lemma}{Lemma}[section]
\newtheorem{proposition}[lemma]{Proposition}
\newtheorem{theorem}[lemma]{Theorem}
\newtheorem{corollary}[lemma]{Corollary}

\newtheorem*{theoremA}{Theorem}

\theoremstyle{definition}

\newtheorem{definition}[lemma]{Definition}
\newtheorem{remark}[lemma]{Remark}

\newcommand{\mfk}[1]{\mathfrak{#1}}
\newcommand{\mbb}[1]{\mathbb{#1}}
\newcommand{\mcl}[1]{\mathcal{#1}}
\newcommand{\mrm}[1]{\mathrm{#1}}
\newcommand{\msc}[1]{\mathscr{#1}}
\newcommand{\mbf}[1]{\mathbf{#1}}

\newcommand{\nocontentsline}[3]{}
\newcommand{\tocless}[2]{\bgroup\let\addcontentsline=\nocontentsline#1{#2}\egroup}

\DeclareMathOperator{\Hom}{Hom}
\DeclareMathOperator{\End}{End}
\DeclareMathOperator{\Ext}{Ext}
\DeclareMathOperator{\Aut}{Aut}

\DeclareMathOperator{\rep}{rep}
\DeclareMathOperator{\Rep}{Rep}

\DeclareMathOperator{\Ind}{Ind}

\DeclareMathOperator{\Fr}{Fr}
\DeclareMathOperator{\SL}{SL}
\DeclareMathOperator{\PSL}{PSL}
\DeclareMathOperator{\ord}{ord}
\DeclareMathOperator{\res}{res}
\DeclareMathOperator{\FPdim}{FPdim}
\DeclareMathOperator{\TL}{TL}
\DeclareMathOperator{\K}{K}

\newcommand{\0}{{\small|0\rangle}}

\newcommand{\ot}{\otimes}

\newcommand{\dotu}{\dot{\mathbf{u}}}

\newcommand{\aff}{\mathrm{aff}}

\renewcommand{\1}{\mathbf{1}}
\renewcommand{\tilde}{\widetilde}
\renewcommand{\O}{\mathscr{O}}
\renewcommand{\binom}[2]{{\Small\left[\begin{matrix}\ #1\ \\ #2 \end{matrix} \right]}}

\title[]{Quantum $\operatorname{SL}(2)$ and logarithmic vertex operator algebras at $(p,1)$-central charge}
\date{\today}

\author{Terry Gannon}
\email{tgannon@math.ualberta.ca}
\address{Department of Mathematics, University of Alberta,
Edmonton, Alberta, Canada T6G 2G1}

\author{Cris Negron}
\email{cnegron@email.unc.edu}
\address{Department of Mathematics, University of North Carolina, Chapel Hill, NC 27599}

\begin{document}
\maketitle

\begin{abstract}
We provide a ribbon tensor equivalence between the representation category of small quantum $\SL(2)$, at parameter $q=e^{\pi i/p}$, and the representation category of the triplet vertex operator algebra at integral parameter $p>1$.  We provide similar quantum group equivalences for representation categories associated to the Virasoro, and singlet vertex operator algebras at central charge $c=1-6(p-1)^2/p$.  These results resolve a number of fundamental conjectures coming from studies of logarithmic CFTs in type $A_1$.
\end{abstract}

\tocless\section{}

We prove a number of fundamental conjectures which relate quantum group representations for $\SL(2)$, and modules for non-rational vertex operator algebras VOAs at central charge $c=1-6(1-p)^2/p$.  We consider specifically the triplet $\mcl{W}_p$, singlet $\mcl{M}_p$, and Virasoro $\mcl{V}ir_c$ vertex operator algebras.
\par

For the Virasoro VOA, we consider a category of modules which is the logarithmic analog of a rational minimal model.  By \emph{logarithmic} we mean, in the simplest sense, that the theories under consideration are non-rational, i.e. non-semisimple. The rational minimal models occur at central charge $c_{p,q}=1-6(p-q)^2/pq$, for coprime $p,q>1$. Their logarithmic analogs, which occur at central charge $c=c_{p,1}$, were first considered in works of Pearce, Rasmussen, and Zuber \cite{pearceetal06,rasmussenpearce07}, though from a physical perspective.
\par

The triplet and singlet algebras have been studied extensively from both physical, and representation theoretic perspectives (many of these papers are included in our references below), with both algebras appearing first in work of Kausch from the early 90's \cite{kausch91}.  Taken together, these three classes of vertex operator algebras $\mcl{V}ir_c$, $\mcl{M}_p$, and $\mcl{W}_p$, provide the most well-studied non-rational VOAs, or logarithmic chiral CFTs, available to us at the present moment.
\par

In recent works of Creutzig, McRae, and Yang \cite{creutzigmcraeyang,mcraeyang}, and earlier work of Tsuchiya and Wood \cite{tsuchiyawood13}, it was shown that each of the VOAs mentioned above admits a corresponding ribbon tensor category of ``affine" representations.  We denote these representation categories by
\begin{equation}\label{eq:cats}
\rep(\mcl{V}ir_c)_{\aff},\ \ \rep(\mcl{M}_p)_{\aff},\ \ \text{and}\ \rep(\mcl{W}_p)
\end{equation}
respectively.  For modules over the triplet $\rep(\mcl{W}_p)$, we simply consider finite length $\mcl{W}_p$-modules.  The constructions of the categories $\rep(\mcl{V}ir_c)_{\aff}$ and $\rep(\mcl{M}_p)_{\aff}$ are slightly more involved, and are recalled in Sections \ref{sect:Vaff} and \ref{sect:Mp_rep} below.  Let us say here that the simple objects in $\rep(\mcl{V}ir_c)_{\aff}$ and $\rep(\mcl{M}_p)_{\aff}$ are those simple modules of integral lowest (conformal) weight $h_{n,s}=\frac{1}{4p}((np-s)^2-(p-1)^2)$, $n,s\in \mbb{Z}$.  An important point is that the categories \eqref{eq:cats} are all \emph{affine}, in the sense that they each admit a distinguished tensor generator, or distinguished faithful representation if one prefers (cf.\ \cite[\S II.5 Corollaire 5.2]{demazuregabriel70}).

We compare the tensor categories of \eqref{eq:cats} to categories of representations for quantum $\SL(2)$ at parameter $q=\exp(\pi i/p)$.  We consider specifically the categories
\[
\rep\SL(2)_q,\ \ \rep(\dotu_q(\mfk{sl}_2)),\ \text{and}\ \rep(u_q(\mfk{sl}_2))
\]
of character graded representations of Lusztig's divided power algebra $U^{Lus}_q(\mfk{sl}_2)$, character graded representations of small quantum $\SL(2)$, and usual representations of small quantum $\SL(2)$ respectively (see Sections \ref{sect:SL2q} and \ref{sect:udot}).  We establish the following collection of equivalences, which were conjectured across the works \cite{gainutdinovetal06,bfgt09,creutzigmilas14,costantinoetal15,creutzigetal}.

\begin{theoremA}[{\ref{thm:triplet}/\ref{thm:LMM}/\ref{thm:singlet}}]
There are equivalences of ribbon tensor categories
\[
\left\{\begin{array}{l}
\K:\rep\SL(2)_q\overset{\sim}\to \rep(\mcl{V}ir_c)_{\aff}\vspace{2mm}\\
\Psi:\rep(\dotu_q(\mfk{sl}_2))\overset{\sim}\to \rep(\mcl{M}_p)_{\aff}\vspace{2mm}\hspace{3cm}\\
\Theta:\rep(u_q(\mfk{sl}_2))\overset{\sim}\to \rep(\mcl{W}_p).
\end{array}\right.
\]
\end{theoremA}

The particular ribbon structures employed on the quantum group sides of Theorems \ref{thm:triplet}, \ref{thm:LMM}, and \ref{thm:singlet} are those ``inverse" to the standard choices of \cite{lusztig93,jantzen95} (see Section \ref{sect:SL2q}).  We note that, in the process of proving the above result we establish modularity of the category $\rep(\mcl{W}_p)$ of triplet modules (see Theorem \ref{thm:modular}).  This point may be of independent interest to readers.
\par

For a clearer historical account, the equivalences $\K$ and $\Theta$ were originally conjectured to exist in works of Bushlanov, Feigin, Gainutdinov, and Tipunin \cite{bfgt09}, and Gainutdinov, Semikhatov, Tipunin, and Feigin \cite{gainutdinovetal06} respectively.  The basis for these conjectures was that a number of invariants for quantum groups and their corresponding vertex operator algebras were (essentially) observed to agree.  Indeed, it was argued in \cite{feigingainutdinovsemikhatovtipunin06} that the modular group representations for $u_q(\mfk{sl}_2)$ and $\mcl{W}_p$ agree, and also that their fusion rings agree \cite{fuchshwangsemikhatovtipunin04}.  Furthermore, it was shown in work of Nagatomo and Tsuchiya \cite{nagatomotsuchiya11}, and subsequently McRae and Yang \cite{mcraeyang}, that there is an \emph{abelian}, non-tensor, equivalence between $\rep(u_q(\mfk{sl}_2))$ and $\rep(\mcl{W}_p)$.  At the particular parameter $p=2$, Creutzig, Lentner, and Rupert verified that this equivalence can in fact be enhanced with the desired tensor structure \cite{clr}.  The possibility of the equivalence $\Psi$ was alluded to in the works of Creutzig and Milas, and Costantino, Geer, and Patureau-Mirand \cite{creutzigmilas14,costantinoetal15}, then was conjectured explicitly in work of Creutzig, Gainutdinov, and Runkel \cite{creutzigetal} (see also Remark \ref{rem:CM}).
\par

As one might expect, analogs of Theorems \ref{thm:triplet}, \ref{thm:LMM}, and \ref{thm:singlet} have been conjectured in arbitrary Dynkin type \cite{feigintipunin,sugimoto,lentner21}.  The analog of the equivalence $\Theta$ at a given almost-simple algebraic group $G$, for example, proposes an equivalence of modular tensor categories between representations of the small quantum group for $G$ at $q=e^{\pi i/p}$, and modules for the ``logarithmic $W$-algebra" $\mcl{W}_p(G)$ of \cite{adamovicmilas14,sugimoto}.  Here, for the small quantum group, one should specifically take the cocycle corrected variant of \cite{negron,gainutdinovlentnerohrmann}.
\par

While additional infrastructure is needed to address these conjectures outside of type $A_1$--in particular the VOAs $\mcl{W}_p(G)$ need to be studied further--the results herein give credence to the claim that representations of quantum groups and CFTs should be strongly intertwined, even in the logarithmic context.  If we consider type $A_{>1}$ for example, and suppose a sufficently strong understanding of the algebras $\mcl{W}_p(\SL(n))$, one could presumably employ the Hecke presentation for $\rep\SL(n)_q$ of \cite[Proposition 4.7]{davydovmolev11}, and follow the arguments of the present text, using \cite[Proposition 7.3]{negron} and \cite[Proposition 7.4.2]{coulembieretingofostrikpauwels}, to provide the desired equivalences between $\rep u_q(\mfk{sl}_n)$ and $\rep\mcl{W}_p(\SL(n))$ at arbitrary $n$.

\subsection{Methods}

Let us focus on the equivalence $\Theta:\rep(u_q(\mfk{sl}_2))\overset{\sim}\to \rep(\mcl{W}_p)$ of Theorem \ref{thm:triplet}, which is the primary target of this work.  We first note that our small quantum group $u_q(\mfk{sl}_2)$ is the cocycle corrected variant of \cite{gainutdinovrunkel17}, which we accept at this point as the ``correct" version of the small quantum group at an even order root of unity.  A point which is essential to this work is the observation \cite{negron} that the category of representations $\rep(u_q(\mfk{sl}_2))$ can be understood as the de-equivariantization of $\rep\SL(2)_q$ along the embedding $\Fr:\rep\PSL(2)\to \rep\SL(2)_q$ provided by Lusztig's quantum Frobenius functor \cite{lusztig93}.  Of course, this is a rather technical statement, but the point is the following (Proposition \ref{prop:basechange}): One can identify tensor functors $\rep(u_q(\mfk{sl}_2))\to \msc{A}$ to a given tensor category $\msc{A}$ with a particular class of tensor functors $\rep\SL(2)_q\to \msc{A}$ out of big quantum $\SL(2)$.
\par

With this general framework in mind, we observe furthermore that tensor maps out of $\rep\SL(2)_q$ are classified in work of Ostrik \cite{ostrik08}.  (See Theorems \ref{thm:ostrik} and \ref{thm:braidedostrik}.)  In particular, a tensor map $\rep\SL(2)_q\to \msc{A}$ to a some category $\msc{A}$ is specified by a choice of self-dual object $W$ in the target $\msc{A}$, which satisfies certain non-degeneracy properties.  So we approach the equivalence $\Theta$ by leveraging the works \cite{ostrik08} and \cite{negron} in tandem.  
\par

Of course, in order to produce the equivalence $\Theta$ in the suggested manner, we must have a clear understanding of the category of modules $\rep(\mcl{W}_p)$, and in particular of its self-dual tensor generator $X^+_2$.  One obtains such a concrete understanding of $\rep(\mcl{W}_p)$ by exploiting relationships between matrix entries of compositions of intertwining operators and differential equations on the sphere.  Such relationships go back to the beginnings of CFT, and are also central to the philosophies of, say, Huang \cite{huang08} and Tsuchiya-Wood \cite{tsuchiyawood13}.  (See Section \ref{sect:conf_blocks}.)  This approach is also present in the recent works \cite{creutzigmcraeyang,mcraeyang}.  The equivalence $\Theta$ is therefore deduced via a propitious interplay between category theoretic and analytic techniques.
\par

The equivalence $\operatorname{K}$ for the Virasoro is essentially a corollary of our arguments for the triplet, which we employ in conjunction with works of Creutzig, Kanade, McRae, and Yang \cite{creutzigkanademcrae,mcraeyang}.  The equivalence $\Psi$ is deduced from a nontrivial analysis of the representation category $\rep(\mcl{M}_p)_{\rm aff}$, and an additional analysis of certain rational actions of the torus $\mbb{C}^\times$ on the categories $\rep(u_q(\mfk{sl}_2))$ and $\rep(\mcl{W}_p)$.

\begin{remark}
In work of Creutzig, Lentner, and Rupert \cite{clr}, the authors suggest an alternate construction of the equivalence $\Theta$ for the triplet, and they realize their construction explicitly when $p=2$.  The methods employed in \cite{clr} differ significantly from the ones employed here, and we invite the curious reader to consult the aforementioned text.
\end{remark}

\subsection{Outline}

Sections \ref{sect:C}--\ref{sect:Wp} cover background material.  Section \ref{sect:ostrik} recalls Ostrik's classification of tensor functors out of $\rep\SL(2)_q$, and also recalls basic facts about the Temperley-Lieb category.  In Section \ref{sect:calc} we perform a straightforward calculation of all braidings for the Temperley-Lieb category.  Sections \ref{sect:X2} and \ref{sect:conf_blocks} are dedicated to an analysis of the self-dual generator $X^+_2$ of $\rep(\mcl{W}_p)$.  In Sections \ref{sect:triplet}--\ref{sect:singlet} we establish the equivalences $\Theta$, $\operatorname{K}$, and $\Psi$ for the the triplet, Virasoro, and singlet vertex operator algebras, respectively.  In the appendices we cover some technical information regarding induction for VOA extensions, and the calculus of (de-)equivariantization for tensor categories equipped with algebraic group actions.

\subsection{Acknowledgements}

Thanks to Pavel Etingof and Victor Ostrik for their elaborations on the work \cite{ostrik08}, and also on various aspects of Tannakian reconstruction.  Thanks to Robert McRae for clarification on the literature, and to Thomas Creutzig and Dmitri Nikshych for helpful conversations.  Thanks to Simon Wood, who suggested the use of the ribbon element for the proof of Theorem \ref{thm:modular}.  The first author is supported by NSERC. The second author is supported by NSF grant DMS-2001608.  This material is based upon work supported by the National Science Foundation under grant DMS-1440140, while the authors were in residence at the Mathematical Sciences Research Institute in Berkeley, California, during the Spring 2020 semester.

\tableofcontents

\section{(Finite) tensor categories}
\label{sect:C}

We cover some basic information about finite tensor categories.  Our presentation is based on the texts \cite{bakalovkirillov01,egno15}, as well as the paper \cite{etingofostrik04}.  We work over the base field $k=\mbb{C}$.

\subsection{(Finite) tensor categories and fusion categories}

A tensor category (over $\mbb{C}$) is a $\mbb{C}$-linear, $\Hom$-finite, abelian monoidal category $\msc{C}$ which is rigid, has all objects of finite length, and has a simple unit object $\1$.  Rigidity means that all objects $X$ in $\msc{C}$ have left and right duals $X^\ast$ and ${^\ast X}$ \cite[\S 2.10]{egno15}.  A tensor functor between tensor categories $F:\msc{C}\to \msc{D}$ is, by definition, an \emph{exact}, $\mbb{C}$-linear, monoidal functor.  By a natural isomorphism between tensor functors we mean a natural isomorphism which respects the monoidal structures in the expected ways.  The following basic observation will be used throughout the text.

\begin{proposition}[{\cite[Proposition 1.19]{delignemilne82}}]\label{prop:faithful}
Any tensor functor $F:\msc{C}\to \msc{D}$ between tensor categories is faithful.
\end{proposition}

A tensor category $\msc{C}$ is called \emph{finite} if it has finitely many simple objects, up to isomorphism, and enough projectives.  A tensor category is called a \emph{fusion} category if it is finite and semisimple.
\par

Abstractly, any finite tensor category $\msc{C}$ admits an abelian equivalence $\msc{C}\cong \rep(B)$ to the representation category of a finite-dimensional algebra.  (One can take specifically $B$ to be the endomorphism ring of a projective generator.)  For some examples, one can consider the category $\msc{C}=\rep(\mcl{W}_p)$ of finite length modules over the triplet vertex operator algebra.  It was shown in \cite{nagatomotsuchiya11,tsuchiyawood13} that $\rep(\mcl{W}_p)$ admits a natural finite tensor category structure.  Also, for any finite-dimensional quasi-Hopf algebra $u$ \cite[\S 5.13]{egno15}, the category $\rep(u)$ of finite-dimensional $u$-representations has the natural structure of a finite tensor category, with the product $\ot$ on $\rep(u)$ induced by the coproduct on $u$.

\subsection{Frobenius-Perron dimension}
\label{sect:fpdim}

For $\msc{C}$ a tensor category, the Grothendieck ring $K(\msc{C})$ is a $\mbb{Z}_+$-ring \cite[Definition 3.1.1]{egno15}, in the sense that it is a free $\mbb{Z}$-algebra with specified basis $\{x_i\}_i\subset K(\msc{C})$ and non-negative structure coefficients $c_{i,j}^k$, $x_i\cdot x_j=\sum_kc_{i,j}^kx_k$.  The basis $\{x_i\}_i$ is provided by the isoclasses of the simples $\{[X_i]:X_i\ \text{simple in }\msc{C}\}$, and the unit $[\1]$ in $K(\msc{C})$ is provided by the unit object in $\msc{C}$.  When $\msc{C}$ is a finite tensor category, the Grothendieck ring $K(\msc{C})$ is of finite rank over $\mbb{Z}$.
\par

For any finite-rank $\mbb{Z}_+$-ring $A$ one has a canonically associated dimension function $\FPdim:A\to \mbb{R}$ called the Frobenius-Perron dimension.  For any $x$ in the specified basis for $A$, the Frobenius-Perron dimension $\FPdim(x)$ is defined as the maximal non-negative real eigenvalue of the linear map $x\cdot-:\mbb{R}\ot_\mbb{Z}A\to \mbb{R}\ot_\mbb{Z}A$.  Since multiplication by $x$ is represented by a matrix with non-negative entries, the Frobenius-Perron theorem ensures the existence of such an eigenvalue.  Furthermore, for any such $x$, we have that $\FPdim(x)\geq 1$ \cite[Propositions 3.3.4]{egno15}.
\par

The function $\FPdim$ is a ring homomorphism, and it is in fact the unique character of $A$ which takes positive values on the given basis \cite[Proposition 3.3.6]{egno15}.  We apply the above general construction to deduce a dimension function $\FPdim:K(\msc{C})\to \mbb{R}$ for the Grothendieck ring of any finite tensor category $\msc{C}$.  The uniqueness properties of the Frobenius-Perron dimension imply the following (standard) result.

\begin{lemma}
If $B$ is a finite-dimensional quasi-Hopf algebra, then for any $V$ in $\rep(B)$, $\FPdim(V)=\dim_\mbb{C}(V)$.
\end{lemma}

\begin{proof}
The vector space dimension defines an algebra map $\dim_\mbb{C}:K(\rep(B))\to \mbb{R}$ which takes positive values on each class $[V]$ of a non-zero representation $V$.  Since $\FPdim$ is the unique character of $K(\rep(B))$ with this property, it follows that $\FPdim=\dim_\mbb{C}$.
\end{proof}

We also have a general notion of Frobenius-Perron dimension for $\mbb{Z}_+$-rings \emph{themselves} \cite[Definition 3.3.12]{egno15}, which reduces to the following in our setting.

\begin{definition}[{\cite[Definition 6.1.7]{egno15}}]
For any finite tensor category $\msc{C}$, the Frobenius-Perron dimension of $\msc{C}$ is defined as
\[
\FPdim(\msc{C}):=\sum_i \FPdim(P_i)\FPdim(X_i),
\]
where the sum runs over the isoclasses of simples $X_i$, and each $P_i$ is the projective cover of $X_i$.
\end{definition}

\subsection{Surjective tensor functors}

A tensor functor $F:\msc{C}\to \msc{D}$ is called \emph{surjective} if any object in $\msc{D}$ is a subquotient of $F(X)$, for some $X$ in $\msc{C}$.  We have the following two essential results.

\begin{theorem}[{\cite[Theorem 2.5]{etingofostrik04}}]
If $F:\msc{C}\to \msc{D}$ is a surjective tensor functor between finite tensor categories, then the image $F(P)$ of any projective object $P$ in $\msc{C}$ is projective in $\msc{D}$.
\end{theorem}

\begin{theorem}[{\cite[Proposition 2.20]{etingofostrik04}}]
If $F:\msc{C}\to \msc{D}$ is a surjective tensor functor between finite tensor categories, then $\FPdim(\msc{D})\leq \FPdim(\msc{C})$. Furthermore, if $\FPdim(\msc{D})= \FPdim(\msc{C})$, then any surjective tensor functor $F:\msc{C}\to \msc{D}$ is an equivalence.
\end{theorem}

\subsection{Tensor generators}
\label{sect:tgen}

For a collection of objects $\{Y_j\}_j$ in a tensor category $\msc{C}$, let $\langle Y_j\rangle_j$ denote the smallest (full) tensor subcategory in $\msc{C}$ which contains the $Y_j$ and is closed under taking subquotients.  We call this subcategory the tensor subcategory in $\msc{C}$ \emph{generated by} the $Y_j$.  We say $\msc{C}$ is tensor generated by the collection $\{Y_j\}_j$ if $\msc{C}=\langle Y_j\rangle_j$.

\subsection{Braided tensor categories, ribbon tensor categories, etc.}

A braiding on a tensor category $\msc{C}$ is a chosen collection of natural isomorphisms $c_{X,Y}:X\ot Y\to Y\ot X$, for each $X$ and $Y$ in $\msc{C}$, which satisfy the equation
\begin{equation}\label{eq:191}
c_{X,Y\ot Z}=(id\ot c_{X,Z})(c_{X,Y}\ot id)\ \ \text{and}\ \ c_{X\ot Y,Z}=(c_{X,Z}\ot id)(id\ot c_{Y,Z})
\end{equation}
for each triple of objects in $\msc{C}$, and $c_{X',Y'}(f\ot g)=(g\ot f)c_{X,Y}$ for each pair of morphisms $f:X\to X'$ and $g:Y\to Y'$ in $\msc{C}$.  We also require that the braidings $c_{\1,X}$ and $c_{X,\1}$ composed with the unit isomorphisms are the identity.  A tensor category equipped with a particular choice of braiding is called a \emph{braided tensor category}.

\begin{remark}
We have suppressed the associator in the equations \eqref{eq:191}.
\end{remark}

For a braiding $c$ on a tensor category $\msc{C}$, we let $c^2$ denote the square operation $c^2_{X,Y}:=c_{X,Y}c_{X,Y}$.  The \emph{M\"uger center} $Z_{\text{\rm M\"ug}}(\msc{C})$ of a braided tensor category $\msc{C}$ is the full subcategory consisting of all objects $X$ in $\msc{C}$ for which $c^2_{X,-}=id_{X\ot-}$.  We call a finite braided tensor category \emph{non-degenerate} if the M\"uger center $Z_{\text{\rm M\"ug}}(\msc{C})$ is just $Vect$, i.e.\ if any M\"uger central object is isomorphic to some additive power of the unit.

A \emph{twist} for a braided tensor category $\msc{C}$ is a choice of natural automorphism $\theta$ of the identity functor, i.e.\ a collection of natural isomorphisms $\theta_X:X\to X$ for each $X$, such that
\[
\theta_{X\ot Y}=(\theta_X\ot\theta_Y)c_{X,Y}^2
\]
at all $X$ and $Y$ in $\msc{C}$.  A \emph{ribbon tensor category} is a braided tensor category with a choice of twist $\theta$ which is stable under duality, in the sense that $\theta_X^\ast=\theta_{X^\ast}$ at all $X$ in $\msc{C}$.  A \emph{modular tensor category} is a finite, non-degenerate, ribbon tensor category. Note that, unlike some authors, we do not require a modular tensor category to be semisimple. When it is, we call it a modular fusion category.

\subsection{Tensor categories coming from vertex operator algebras}\label{sect:voacat}

Consider $\mcl{V}$ a vertex operator algebra (VOA). VOAs and their modules carry an action of the Virasoro algebra, coming from the conformal vector $\omega\in\mcl{V}$. By a generalized $\mcl{V}$-module $W$ is meant a module in the obvious algebraic sense, together with a grading $W=\coprod_{r\in\mbb{C}}W_{(r)}$ into generalized eigenspaces of the Virasoro operator $L(0)$.  The presence of the adjective `generalized' here is unfortunate and historical, and merely refers to the subspaces $W_{(r)}$ being generalized eigenspaces. We are primarily interested in grading-restricted $W$, which means each subspace $W_{(r)}$ is finite-dimensional, and for each $r\in\mbb{C}$ we have $W_{(r+k)}=0$ for all sufficiently small $k\in\mbb{Z}$. For example, $\mcl{V}$ is a grading-restricted generalized module over itself, where all $V_{(r)}$ are actually eigenspaces for $L(0)$, and where all weights $r$ are integers. A common requirement on generalized modules is $C_1$-cofiniteness (defined e.g. in \cite{huang09}). 

\par In the series of papers \cite{huanglepowskyzhang14}-\cite{huanglepowskyzhang}, Huang, Lepowsky and Zhang give technical conditions under which a full subcategory of the category of grading-restricted generalized $\mcl{V}$-modules can be a braided tensor category. In particular, see \cite[Theorem 12.15, Corollary 12.16]{huanglepowskyzhang}. In these cases, the VOA $\mcl{V}$ itself serves as the unit, and all structure maps are deduced via a certain analysis of multivalued functions on the punctured complex plane.  We should be clear that, when we speak of a class $\mcl{C}$ of $\mcl{V}$-modules admitting \emph{a} braided tensor structure, the braided tensor structure is specified uniquely.  In general, the problem is that, for a given class of $\mcl{V}$-modules, no such structure may exist.

\par Establishing rigidity is more subtle. The dual of a grading-restricted generalized module $W$ should be the contragredient  $W^*$, which is the natural $\mcl{V}$-module structure on the restricted dual $\coprod_rW_{(r)}^*$. In a natural sense,  $(W^*)^*$ can be identified with $W$ -- this is clearly true as a vector space (since the $W_{(r)}$ are finite-dimensional), but as well the formula for the vertex operator of $(W^*)^*$ collapses to that of $W$. In establishing rigidity, the (co-)evaluation maps are generally clear up to scaling; the challenge is to verify the rescalings are finite.

\par Such a tensor category of VOA modules comes equipped with a twist provided by the exponential $\theta=e^{2\pi i L(0)}$, which one verifies as in \cite[Theorem 4.1]{huang08}.  (A concise recounting of the situation can be found in \cite[Section 3]{creutzigkanademcrae}.)

\par We follow the standard VOA practice of distinguishing between modules and representations. The notion of a representation of a VOA $\mcl{V}$ (i.e. a homomorphism from $\mcl{V}$ to some sort of VOA  canonically associated to vector space $W$), and its relation to $\mcl{V}$-modules,  is much more subtle than it is for say associative algebras. This is discussed in more detail in the book \cite{lepowskyli}, which also serves as a standard introduction to VOA theory.

\section{Quantum $\SL(2)$}

We recall basic information for the category of quantum $\SL(2)$-representations, and give two interpretations of the small quantum group for $\SL(2)$ at a root of unity $\zeta\in \mbb{C}^\times$.  Much of our presentation is general, in the sense that the quantum parameter $\zeta$ can be an arbitrary (nonzero) complex number.  However, when specificity is needed, we focus on the even order case $\ord(\zeta)=2p$.

\subsection{Big quantum $\SL(2)$}
\label{sect:SL2q}

Consider an arbitrary parameter $\zeta\in \mbb{C}^\times$, and let $\Lambda$ denote the character lattice for $\SL(2)$, so that $\Lambda=\frac{1}{2}\mbb{Z}\alpha$ with $\alpha$ simple and positive.  We consider the tensor category
\[
\rep \SL(2)_\zeta=\left\{
\begin{array}{c}
\text{The category of $\Lambda$-graded representations}\\
\text{of Lusztig's divided power algebra }U^{Lus}_\zeta(\mfk{sl}_2)
\end{array}\right\}=\rep\dot{U}_\zeta(\mfk{sl}_2)
\]
The final algebra $\dot{U}_\zeta(\mfk{sl}_2)$ is the modified quantum enveloping algebra of \cite[Chapter 31]{lusztig93}, and for the $\Lambda$-grading we require that toral elements in $U^{Lus}_\zeta(\mfk{sl}_2)$ act on the $\lambda$-space $V_\lambda$ via the appropriate eigenfunctions.  We refer to objects in $\rep\SL(2)_\zeta$ as $\SL(2)_\zeta$-representations.
\par

For a finite order parameter $\ord(\zeta^2)=p$, for example, we require that the toral elements in the quantum group act as
\[
K_\alpha\cdot v=\zeta^{(\alpha,\lambda)}v\ \text{and}\ \binom{K_\alpha; 0}{p}\cdot v=\binom{\langle \alpha,\lambda\rangle}{p}v,\ \text{where $\binom{a}{b}$ is the $\zeta$-binomial.
}
\]
At such $\zeta$ the category $\rep\SL(2)_\zeta$ can then be described explicitly as the category of finite-dimensional $\Lambda$-graded vector spaces $V$ equipped with linear operators $E, F, E^{(p)}, F^{(p)}:V\to V$ which shift the grading as
\[
E\cdot V_\lambda\subset V_{\lambda+\alpha},\ \ E^{(p)}\cdot V_\lambda\subset V_{\lambda+p\alpha},\ \ F\cdot V_\lambda\subset V_{\lambda-\alpha},\ \ F^{(p)}\cdot V_\lambda\subset V_{\lambda-p\alpha}
\]
and satisfy the standard quantum group relations of \cite{lusztig90,lusztig93}.  At an infinite order parameter one can deduce a similar (but slightly easier) presentation of $\rep\SL(2)_\zeta$.
\par

Let $\Omega:V\ot W\to V\ot W$ denote the diagonal endomorphism associated to the normalized Killing form,
\[
\Omega(v\ot w)=\zeta^{(\deg(v),\deg(w))}v\ot w,\ \ (\alpha,\alpha)=2.
\]
When $\zeta$ is of finite order $\ord(\zeta^2)=p$, we have the formal element
\begin{equation}\label{eq:R}
R=(\sum_{n=0}^{p-1}\zeta^{-n(n-1)/2}\frac{(\zeta^{-1}-\zeta)^n}{[n]!}E^n\ot F^n)\Omega^{-1}=(1-(\zeta-\zeta^{-1})E\ot F+\dots)\Omega^{-1}
\end{equation}
which acts as a well-defined linear endomorphism on products $V\ot W$ of quantum group representations.  Similarly, at an infinite order parameter we have the element $R$ defined by replacing the finite sum \eqref{eq:R} with the evident power series.  The element $R$ provides an $R$-matrix for the category $\rep\SL(2)_\zeta$, so that we have the associated braiding
\[
\begin{array}{l}
c_{V,W}:V\ot W\to W\ot V,\\
 c_{V,W}(v\ot w):=R_{21}\cdot(w\ot v)=\zeta^{-(\deg v,\deg w)}\sum_n\zeta^{-n(n-1)/2}\frac{(\zeta^{-1}-\zeta)^n}{[n]!}F^nw\ot E^n v
\end{array}
\]
\cite{jantzen95,lusztig93}.

\begin{remark}
To be clear, the normalized Killing form takes half-integer values on the character lattice $\Lambda$, so that the definition of $\Omega=\sum_{\lambda,\mu\in \Lambda}\zeta^{(\lambda,\mu)}1_\lambda\ot 1_\mu$ involves a choice of square root for $\zeta$.  We simply halve the argument and take the positive square root of the magnitude to define $\zeta^{1/2}:=\sqrt{|\zeta|}\exp(i\operatorname{arg}/2)$, where $0\leq \operatorname{arg}<2\pi$ is such that $\zeta=|\zeta|e^{i\operatorname{arg}}$.  There is, however, another $R$-matrix for $\rep\SL(2)_\zeta$ defined by taking the negative square root of $|\zeta|$ and, accounting for reverse braidings, we observe a total of \emph{four} possible braidings on the category of $\SL(2)_\zeta$-representations.
\par

We refer to the above braiding on $\rep\SL(2)_\zeta$, defined by taking the positive square root of the magnitude $|\zeta|$ and corresponding $R$-matrix \eqref{eq:R}, as the ``standard braiding" on $\rep\SL(2)_\zeta$.
\end{remark}

Suppose now, for the sake of specificity, that $\zeta$ is of even order $\ord(\zeta)=2p$.  The category $\rep\SL(2)_\zeta$ is pivotal \cite[Definition 4.7.7]{egno15}, with slightly unorthodox pivotal element provided by the grouplike $K^{p-1}$.  We consider the Drinfeld morphism
\[
u_W:W\to W,\ \ u_W(w)=\zeta^{(\deg w,\deg w)}\sum_{n=0}^{p-1} \zeta^{n(n+1)/2}\zeta^{n(\deg w,\alpha)}\frac{(\zeta^2-1)^n}{[n]!}F^nE^n w
\]
defined by $R$ to find the corresponding twist
\[
\theta_W^{-1}=u_WK^{-p+1}:W\to W,
\]
\[
\theta_W^{-1}(w)=(-1)^{(\deg w,\alpha)}\zeta^{(\deg w,\deg w)}\sum_{n=0}^{p-1} \zeta^{n(n+1)/2}\zeta^{(n+1)(\deg w,\alpha)}\frac{(\zeta^2-1)^n}{[n]!}F^nE^n w.
\]
This twist provides the category of quantum group representations $\rep\SL(2)_\zeta$ with a ribbon structure.

\subsection{Simple objects in $\rep\SL(2)_\zeta$ and the standard representation}

For any dominant weight $\lambda\in \Lambda$, i.e.\ any weight of the form $\lambda=n\frac{\alpha}{2}$ with $n$ positive, we have a uniquely associated simple module $L(\lambda)$ in $\rep\SL(2)_\zeta$ of highest weight $\lambda$.  Furthermore, these representations exhaust all of the simple objects in $\rep\SL(2)_\zeta$ \cite[Proposition 6.4]{lusztig89}.  Each $L(\lambda)$ appears as
\[
L(\lambda)=\mbb{C}v_\lambda\oplus \mbb{C}v_{\lambda-m_1\alpha}\oplus \dots \oplus \mbb{C}v_{\lambda-m_{n(\lambda)}\alpha},
\]
where the $m_i$ are positive integers which depend on $\lambda$.  At $\lambda=\frac{\alpha}{2}$ we obtain a $2$-dimensional simple representation
\[
\mbb{V}:=L(\frac{1}{2}\alpha)=\mbb{C}v_1\oplus \mbb{C}v_{-1},\ \ \text{with}\ \ v_{\pm 1}=v_{\pm \frac{\alpha}{2}}.
\]
We call $\mbb{V}$ the \emph{standard representation} for $\SL(2)_\zeta$.
\par

When $\zeta$ is of infinite order, each simple $L(n\alpha/2)$ is of dimension $n+1$ and the category $\rep\SL(2)_\zeta$ is semisimple.  Indeed, at such $\zeta$ we have the obvious abelian equivalence between classical representations $\rep\SL(2)$ and $\rep\SL(2)_\zeta$ which induces an isomorphism on Grothendieck rings.
\par

Consider now $\zeta$ of finite order, and take $p=\ord(\zeta^2)$.  We adopt special notations for the first $p$ simples:
\begin{equation}\label{eq:Vs}
V_1=\1=L(0),\ V_2=\mbb{V}=L(\frac{1}{2}\alpha),\ V_3=L(\alpha),\ \dots,\ V_p=L(\frac{p-1}{2}\alpha).
\end{equation}
Each simple $V_s$ is of dimension $s$, and has non-vanishing weight spaces in degrees $(s-1-2j)\alpha/2$, for all $j=0,\dots, s-1$.
\par

By considering the behaviors of weight spaces under the tensor product, it is easy to see that each simple $L(\lambda)$ is a quotient of some power $\mbb{V}^{\ot n}$ of the standard representation.  The following is well-known, and is a consequence of the fact that the simple representation $V_p$ is also projective \cite[Lemma 3.2.1]{bfgt09}.

\begin{lemma}
The standard representation $\mbb{V}$ generates $\rep\SL(2)_\zeta$, as a tensor category.
\end{lemma}

As for the simples of highest weight $>\frac{p-1}{2}\alpha$, each simple $L(r\frac{p}{2}\alpha)$ is $(r+1)$-dimensional, is supported in weight spaces $(r-2j)\frac{p}{2}\alpha$, for $0\leq j\leq r$, and is annihilated by $E$ and $F$.  We have the following basic result.

\begin{lemma}[{\cite[Theorem 1.10]{andersenwen92}}]\label{lem:L_decomp}
Each simple $L(\lambda)$ admits a unique decomposition $L(\lambda)=L(\mu)\ot V_s$, where $\mu\in p\mbb{Z}\frac{\alpha}{2}$ and $\lambda=\mu+\frac{(s-1)}{2}\alpha$.
\end{lemma}

\begin{proof}
The general version of this theorem is covered in \cite{andersenwen92}.  In our case, one can examine the product $L(\mu)\ot V_s$ directly to see that it admits no submodules, and so is simple.  Since $L(\mu)\ot V_s$ is of highest weight $\lambda=\mu+\frac{(s-1)}{2}\alpha$, we conclude $L(\mu)\ot V_s=L(\lambda)$.
\end{proof}

\subsection{Lusztig's quantum Frobenius}
\label{sect:Fr}

\emph{For the remainder of the section we suppose $\zeta$ is of even order $\ord(\zeta)=2p$}, again for the sake of specificity.
\par

Let us consider the sub-lattice $\Lambda^M:=p\mbb{Z}\alpha$ in $\Lambda$.  The collection of objects $W$ in $\rep \SL(2)_\zeta$ whose $\Lambda$-grading is supported on the sublattice $\Lambda^M$ forms a tensor subcategory in $\rep\SL(2)_\zeta$.  This subcategory is semisimple and is equivalent to the category of representations of classical $\PSL(2)$.  Indeed, Lusztig's quantum Frobenius provides a braided tensor embedding
\[
\Fr:\rep\PSL(2)\to \rep\SL(2)_\zeta
\]
\cite[Chapter 31]{lusztig93}.  Furthermore, one sees that the image of the functor $\Fr$ is precisely the M\"uger center of $\rep\SL(2)_\zeta$ \cite[Theorem 4.3]{negron}, so that we have an identification
\[
Z_{\text{\rm M\"ug}}(\rep\SL(2)_\zeta)=\rep\PSL(2).
\]

\subsection{Small quantum $\SL(2)$ as a quasi-Hopf algebra \cite{creutzigetal}}
\label{sect:smallsl2}

Recall our standing assumption $\ord(\zeta)=2p$.  As an algebra, the small quantum group $u_\zeta(\mfk{sl}_2):=u_\zeta(\SL(2))$ has a presentation
\[
u_q(\mfk{sl}_2)=\mbb{C}\langle E,F,K\rangle/(K^{2p}-1,\ E^p,\ F^p,\ [E,F]-\frac{K-K^{-1}}{\zeta-\zeta^{-1}},\ KE-\zeta^2E,\ KF-\zeta^{-2}FK).
\]
We let $G=\langle K\rangle$ denote the subgroup of units in $u_\zeta(\mfk{sl}_2)$ generated by $K$.
\par

The quasi-Hopf structure on $u_\zeta(\mfk{sl}_2)$ is a ``toral perturbation" of the usual Hopf structure induced by the algebra inclusion $u_\zeta(\mfk{sl}_2)\to U^{Lus}_\zeta(\mfk{sl}_2)$ into Lusztig's divided power algebra \cite{lusztig90,lusztig93}.  While we won't explicitly describe the quasi-Hopf structure here, let us say that the (co)associator $\phi\in u_\zeta(\mfk{sl}_2)\ot u_\zeta(\mfk{sl}_2)$ lies in the group of grouplikes $\phi\in \mbb{C}G\ot \mbb{C}G$, and the coproduct is of some form
\[
\Delta(E)=E\ot L+M(K\ot E),\ \ \Delta(F)=E\ot L^{-1}+M'(K^{-1}\ot F),\ \ \Delta(K)=K\ot K,
\]
with $L\in \mbb{C}G$ and $M,M'\in \mbb{C}G\ot \mbb{C}G$.  The counit $u_\zeta(\mfk{sl}_2)\to \mbb{C}$ is the usual one.  (More details can be found in any of \cite{gainutdinovrunkel17,gainutdinovlentnerohrmann,negron}.)
\par

The category of $u_\zeta(\mfk{sl}_2)$-representations admits a ribbon structure for which the restriction functor
\[
\res^\omega=(\res,T^\omega):\rep\SL(2)_\zeta\to \rep u_\zeta(\mfk{sl}_2)
\]
becomes a map of ribbon tensor categories, after we introduce a nontrivial tensor compatibility $T^\omega_{V,W}:\res(V)\ot \res(W)\to \res(V\ot W)$ \cite[Proposition 6.3]{negron}.
\par

The particular form of the quasi-Hopf structure on $u_\zeta(\mfk{sl}_2)$ will not be important for us.  Indeed, when addressing the tensor category $\rep u_\zeta(\mfk{sl}_2)$ we prefer the presentation of Section \ref{sect:dE} below.  There are, however, some advantages to considering the quasi-Hopf interpretation of representations for the small quantum group.  Namely, simple facts about the representation theory of $u_\zeta(\mfk{sl}_2)$ are clear from this concrete perspective.
\par

We observe for $u_\zeta(\mfk{sl}_2)$ the anomalous $1$-dimensional representation
\[
\chi=\mbb{C}v_p,\ \ E\cdot v_p=F\cdot v_p=0,\ K\cdot v_p=-v_p.
\]
Since $\chi$ is $1$-dimensional, one sees abstractly that it must be invertible in the tensor category $\rep u_\zeta(\mfk{sl}_2)$, and since $\chi$ is in fact the \emph{unique} non-trivial $1$-dimensional representation for $u_\zeta(\mfk{sl}_2)$ we see that it is self-dual.  So $\chi\ot \chi\cong \mbf{1}$.
\par

There are furthermore a total of $2p$ simple $u_\zeta(\mfk{sl}_2)$-representations which correspond to the $2p$ characters of the group of grouplikes $G$.  The first $p$ simples are provided by the images of the simples $V_s$ in $\rep\SL(2)_\zeta$, under restriction.  The final $p$ simples are given as the products $\chi\ot V_s$, so that we have a complete list of simple representations $\{V_s,\ \chi\ot V_s:1\leq s\leq p\}$, up to isomorphism.
\par

The fact that the restriction functor from $\rep\SL(2)_\zeta\to \rep u_\zeta(\mfk{sl}_2)$ is a tensor functor, or rather admits the structure of a tensor functor, tells us that the fusion rules for $\rep u_\zeta(\mfk{sl}_2)$ are the expected ones.  Rather, the fusion rules are those one would calculate from the na\"ive Hopf structure on $u_\zeta(\mfk{sl}_2)$ induced by that of $U^{Lus}_\zeta(\mfk{sl}_2)$.

\subsection{Small quantum $\SL(2)$ as Frobenius de-equivariantizaton \cite{arkhipovgaitsgory03,negron}}
\label{sect:dE}

We have the quantum Frobenius $\Fr:\rep\PSL(2)\to \rep\SL(2)_\zeta$ of Section \ref{sect:Fr}, and can consider the central algebra object $\O:=\Fr\O(\PSL(2))$ in the category $\Rep\SL(2)_\zeta$ of infinite-dimensional quantum group representations.  (Specifically, $\Rep\SL(2)_\zeta$ is the category of arbitrary $U^{Lus}_\zeta(\mfk{sl}_2)$-modules $M$ which are the unions $M=\cup_\mu M_\mu$ of finite-dimensional subobjects $M_\mu$ in $\rep\SL(2)_\zeta$.)  We consider the de-equivariantization
\[
(\rep\SL(2)_\zeta)_{\PSL(2)}:=\left\{
\begin{array}{c}
\text{The monoidal category of finitely presented}\\
\O\text{-modules in }\Rep\SL(2)_\zeta
\end{array}
\right\}
\]
\cite{arkhipovgaitsgory03,davydovetingofnikshych18}.  The category $(\rep\SL(2)_\zeta)_{\PSL(2)}$ is monoidal, with product $\ot_\O$, and we have the monoidal functor
\[
dE:\rep\SL(2)_\zeta\to (\rep\SL(2)_\zeta)_{\PSL(2)},\ \ dE(V):=\O\ot V
\]
\cite[Theorem 1.6]{kirillovostrik02}.  There is a unique ribbon structure on the category $(\rep\SL(2)_\zeta)_{\PSL(2)}$ so that the de-equivariantization map $dE$ is a map of ribbon monoidal categories, and we consider $(\rep\SL(2)_\zeta)_{\PSL(2)}$ as a ribbon category with its induced ribbon structure.

\begin{theorem}[{\cite{negron}}]\label{thm:N}
The category $(\rep\SL(2)_\zeta)_{\PSL(2)}$ is finite, rigid, and non-degenerate, and hence modular.  Furthermore, there is a ribbon tensor equivalence $\mbb{C}\ot^\omega_\O-:(\rep\SL(2)_\zeta)_{\PSL(2)}\overset{\sim}\to \rep u_\zeta(\mfk{sl}_2)$ which fits into a diagram
\[
\xymatrix{
\rep\SL(2)_\zeta\ar[rr]^{\res^\omega}\ar[dr]_{dE} & & \rep u_\zeta(\mfk{sl}_2)\\
 & (\rep\SL(2)_\zeta)_{\PSL(2)}\ar[ur]_{\mbb{C}\ot^\omega_\O-} & .
}
\]
\end{theorem}

The above result is comprised specifically of \cite[Corollary 5.6, Proposition 6.3, \& Corollary 7.2]{negron}.  We argue in the present paper that one should observe an equivalence
\[
F:(\rep\SL(2)_\zeta)_{\PSL(2)}\overset{\sim}\to \rep(\mcl{W}_p)
\]
directly \emph{from the de-equivariantization} of $\rep\SL(2)_\zeta$, rather than from the representation category of the quasi-Hopf algebra $u_\zeta(\mfk{sl}_2)$. This is very natural from the VOA perspective, as we will see starting next section. Indeed, we will interpret the de-equivariantization $dE$ as the induction of $\mcl{V}ir_c$-modules to $\mcl{W}_p$-modules.

\section{Modules over the triplet algebra}
\label{sect:Wp}

We provide basic information regarding the tensor category of modules for the triplet vertex operator algebra $\mcl{W}_p$. In this section, by module we mean finite length generalized module.  Most of our presentation is deducible from works of Adamovic-Milas \cite{adamovicmilas08}, Nagatomo-Tsuchiya \cite{nagatomotsuchiya11}, and Tsuchiya-Wood \cite{tsuchiyawood13}, although some of our conclusions are not explicitly in the literature.
\par

As explained in the introduction, we compare modules over the triplet algebra to representations of quantum $\SL(2)$ at a particular complex parameter $q$.
\begin{center}
\emph{Throughout this work we fix $q=e^{\pi i/p}$}.
\end{center}
The parameter $q$ appears in a number of formulas related to modules for the triplet algebra.
\begin{center}
\emph{We similarly fix the central charge $c=1-6(p-1)^2/p$.}
\end{center}

\subsection{Tensor categories for $C_2$-cofinite VOAs}
\label{sect:wpcat}

To keep this subsection relatively short, we will use some standard technical terminology without defining it -- see e.g. the series of papers \cite{huanglepowskyzhang14}-\cite{huanglepowskyzhang} for details in a very general context.

 By a \textit{strongly-finite VOA} we mean  a simple $C_2$-cofinite VOA of positive energy with $\mcl{V}$ isomorphic to $\mcl{V}^*$ as $\mcl{V}$-modules. By $\rep(\mcl{V})$ we mean the category whose objects are finite length generalized $\mcl{V}$-modules, and whose morphisms are homomorphisms. For simplicity we will just call these modules.  In this context, `finite length' is equivalent to `grading-restricted' or `quasi-finite-dimensional' (see \cite[Proposition 4.3]{huang09}). 
  
Recall from Section \ref{sect:voacat} that VOA modules carry actions of the Virasoro algebra. When the VOA is strongly-finite, the eigenvalues of $L(0)$ are rational. Any module $W$ has a minimum such eigenvalue, called  the \textit{conformal weight} $h(W)$.

  The category $\rep(\mcl{V})$ for strongly-finite $\mcl{V}$ has finitely many simples \cite[Proposition 3.6]{DongLiMason2000} and enough projectives \cite[Theorem 3.23]{huang09}.
 In fact, $\rep(\mcl{V})$ is a braided  monoidal category \cite[Theorem 4.11]{huang09} (we use monoidal here rather than tensor because the latter usually requires rigidity), with the tensor unit being $\mcl{V}$ itself. It is a  finite tensor category provided the simple modules (hence all modules) are in addition rigid (see \cite[Theorem 4.4.1]{creutzigmcraeyang}).

\subsection{The triplet algebra and some modules}
\label{sect:irrep_Wp}

We consider the triplet vertex operator algebra $\mcl{W}_p$ at an arbitrary integer parameter $p>1$. Let $V_{\sqrt{p}Q}$  be the lattice VOA associated to the $\sqrt{p}$-scaling of the root lattice $Q\subset \mfk{h}^\ast$ for $\SL(2)$, where $Q$ is given its normalized Killing form $(\alpha,\alpha)=2$, and we employ the non-standard conformal vector
\[
\omega=\frac{1}{4}\alpha_{(-1)}^2\0+\frac{p-1}{2\sqrt{p}}\alpha_{(-2)}\0\in V_{\sqrt{p}Q}.
\]
Then $\mcl{W}_p$ is a vertex operator subalgebra $\mcl{W}_p\subset V_{\sqrt{p}Q}$, and both algebras are of central charge $c=1-6(p-1)^2/p$.  We therefore have a sequence of inclusions
\[
\mcl{V}ir_c\to \mcl{W}_p\to V_{\sqrt{p}Q},
\]
where $\mcl{V}ir_c$ is the simple Virasoro VOA at the prescribed central charge.  (Details can be found in \cite{adamovicmilas08}.)
\par

\par $\rep(\mcl{W}_p)$ is a braided finite tensor category: that $\mcl{W}_p$ is strongly-finite is proved in \cite{carquevilleflohr06,adamovicmilas08}, while rigidity is proved in \cite{tsuchiyawood13} (with some details fleshed out in \cite{creutzigmcraeyang}).  The categorical duals ${}^\ast X$ and $X^\ast$ of $X$ are both identified with its contragredient. In Theorem \ref{thm:modular} we prove that $\rep(\mcl{W}_p)$ is a modular tensor category, as was expected.  We denote the tensor product on $\rep(\mcl{W}_p)$ via the generic notation $\ot$.

\par The VOA $\mcl{W}_p$ has precisely $2p$ simple modules $X^{\pm}_s$, $1\leq s\leq p$, and the triplet algebra itself appears as the simple $\mcl{W}_p=X^+_1$.
As modules for the Virasoro algebra we have
\[
X_s^+=\bigoplus_{m\geq 1}(2m-1)\mcl{L}_{2m-1,s},\ \ X_s^-=\bigoplus_{m\geq 1}2m\!\ \mcl{L}_{2m,s},
\]
where $\mcl{L}_{n,s}$ is the simple Virasoro module of  lowest conformal weight
\begin{equation}\label{eq:h_ns}
h_{n,s}=\frac{1}{4p}((np-s)^2-(p-1)^2)
\end{equation}
\cite{adamovicmilas08} and central charge $c$.  Therefore $X^+_s$ and $X^-_s$ have respective conformal weights
\[
h^+_s:=h_{1,s}
%=\frac{1}{4p}(p^2-2sp+s^2-p^2+2p-1)=\frac{1}{4p}(2(p-s)+s^2-1)
=\frac{1}{4p}(s^2-1)-\frac{1}{2}(s-1)\ \ \text{and}\ \ 
h^-_s:=h_{2,s}
%=\frac{1}{4p}((2p-s)^2-(p-1)^2)\\
=\frac{3}{4}p+\frac{1}{4p}(s^2-1)-s+\frac{1}{2}.
\]
In the case $s=2$ the above expression reduces to $h^+_{2}=\frac{3}{4p}-\frac{1}{2}$.

The triplet algebra $\mcl{W}_p=X_1^+$ is generated as a VOA by the conformal vector $\omega$ together with three vectors of conformal weight $2p-1$.  The conformal vector generates the subalgebra $\mcl{V}ir_c$. Each of those three vectors generates over $\mcl{V}ir_c$ a copy of $\mcl{L}_{3,1}$. We establish in this paper the ribbon tensor equivalence of rep\,SL(2)$_q$ with a subcategory of rep$(\mcl{V}ir_c)$; under this equivalence, $\mcl{L}_{3,1}$ is identified with the image by Lusztig's quantum Frobenius of the adjoint representation of $\PSL(2)$, and $\mcl{W}_p$ (as a $\mcl{V}ir_c$-module) with the regular representation of $\PSL(2)$. From the point of view of VOAs, this happens because Aut$(\mcl{W}_p)\cong \PSL({2},\mbb{C})$ \cite[Theorem 2.3]{adamoviclinmilas12},  and the orbifold (fixed-point subalgebra) of this $\PSL({2},\mbb{C})$ action on $\mcl{W}_p$ is $\mcl{V}ir_c$.

\begin{remark}
With the two notations $\PSL(2)$ and $\PSL(2,\mbb{C})$ we are simply distinguishing between $\PSL(2)$ considered as an algebraic group, or group scheme, and its corresponding discrete group, or Lie group, of $\mbb{C}$-points.
\end{remark}

\subsection{The fusion ring for $\rep(\mcl{W}_p)$ and Frobenius-Perron dimensions}

Let us recall some basic information from \cite{tsuchiyawood13}.

\begin{proposition}[{\cite[Theorem 4.4]{adamovicmilas08} \cite[Theorem 5.1]{nagatomotsuchiya11}}]\label{prop:319}
The object $X^+_p$ is projective (and simple) in $\rep(\mcl{W}_p)$.
\end{proposition}

An independent verification of the projectivity of $X^+_p$ can also be found in \cite[Lemma 7.7]{mcraeyang}.

\begin{theorem}[\cite{tsuchiyawood13}]\label{thm:323}
In the fusion ring $K(\rep(\mcl{W}_p))$, the class $[X_1^+]$ is the unit, the class $[X_1^-]$ acts as $[X_1^-]\cdot[X_s^\pm]=[X_s^{\mp}]$, and the class $[X_2^+]$ acts as
\[
\left\{\begin{array}{ll}
{[X_2^+]}\cdot [X_s^{\pm}]=[X^{\pm}_{s-1}]+[X^{\pm}_{s+1}], & \text{when }1<s<p,\vspace{3mm}\\
{[X_2^+]}\cdot [X^{\pm}_p]=2[X^{\mp}_1]+2[X^\pm_{p-1}].
\end{array}\right.
\]
\end{theorem}

The above two results imply the following.

\begin{corollary}\label{cor:wp_gen}
The tensor category $\rep(\mcl{W}_p)$ is generated by $X^+_2$.
\end{corollary}

\begin{proof}
By Theorem \ref{thm:323} each simple object in $\rep(\mcl{W}_p)$ appears as a subquotient of some power $(X^+_2)^{\ot n}$, so that each simple object lies in the subcategory $\langle X^+_2\rangle$ of $\rep(\mcl{W}_p)$ generated by $X^+_2$.  By Proposition \ref{prop:319}, the projective cover $\msc{P}$ of some simple object $X$ also appears in $\langle X^+_2\rangle$.  Via the composite
\[
X^\ast\ot \msc{P}\to X^\ast\ot X\overset{ev}\to \1, 
\]
exactness of the tensor product, and projectivity of the product $X^\ast\ot \msc{P}$ in $\rep(\mcl{W}_p)$ \cite[Proposition 2.1]{etingofostrik04}, we find that the unit object $\1$ admits a surjection from a projective object which lies in $\langle X^+_2\rangle$.  It follows that the projective cover of $\1$ lies in $\langle X^+_2\rangle$, and subsequently that all projectives in $\rep(\mcl{W}_p)$ lie in $\langle X^+_2\rangle$.  Hence the subcategory $\langle X^+_2\rangle$ is all of $\rep \mcl{W}_p$.
\end{proof}

Via associativity, the formulas of Theorem \ref{thm:323} determine the fusion rules for $K(\rep(\mcl{W}_p))$ completely.  Furthermore, one can explicitly calculate the fusion ring of $u_q(\mfk{sl}_2)$-representations to find that there is a unique isomorphism of $\mbb{Z}^+$-rings
\begin{equation}\label{eq:353}
T:K(\rep u_q(\mfk{sl}_2))\overset{\sim}\to K(\rep \mcl{W}_p)
\end{equation}
which sends the class of the generator $[\mbb{V}]$ to $[X^+_2]$.  We consider the Frobenius-Perron dimension of the category $\rep(\mcl{W}_p)$ (see Section \ref{sect:fpdim}).

\begin{lemma}\label{lem:fpdim}
$\operatorname{FPdim}(\rep(\mcl{W}_p))=\operatorname{FPdim}(\rep u_q(\mfk{sl}_2))=2p^3$.
\end{lemma}

\begin{proof}
The isomorphism \eqref{eq:353} sends each simple $[V_s]$ over $u_q(\mfk{sl}_2)$ to $[X^+_s]$, and each simple $[\chi\ot V_s]$ to $[X^-_s]$.  Since the Frobenius-Perron dimension of an object is determined by its action on the Grothendieck ring, we find now
\[
\operatorname{FPdim}(X^\pm_s)=\operatorname{FPdim}(\chi\ot V_s)=\operatorname{FPdim}(V_s)=s.
\]
Furthermore, by a direct comparison of composition factors for the projective covers $P_s\to V_s$ and $\msc{P}^{\pm}_s\to X^{\pm}_s$ \cite[Proposition 4.5]{nagatomotsuchiya11} \cite[Proposition 2.3.5]{kondosaito11}, one sees also that for the projective covers of the simples we have
\[
T[P_s]=[\msc{P}^+_s],\ \ T[\chi\ot P_s]=[\msc{P}^-_s],
\]
so that the Frobenius-Perron dimensions of the projective covers agree as well.  It follows that
\[
\begin{array}{rl}
\operatorname{FPdim}(\rep(\mcl{W}_p))&=\operatorname{FPdim}(\oplus_{s,\pm}\msc{P}^\pm_s\ot X^\pm_s)\\
&=2\cdot \operatorname{FPdim}(\oplus_{s}P_s\ot V_s)=\operatorname{FPdim}(\rep u_q(\mfk{sl}_2))=2p^3.
\end{array}
\]
\end{proof}

\subsection{The ribbon structure and modularity of $\rep(\mcl{W}_p)$}

We recall that the category of $\mcl{W}_p$-modules is ribbon, with twist provided by the exponential of the zero mode of the conformal vector
\[
\theta=e^{2\pi iL(0)}.
\]
Applying this to the simples described in Section \ref{sect:irrep_Wp} gives the following.

\begin{lemma}\label{lem:wp_tw}
The twist $\theta$ acts on the simple modules $X_s^{\pm}$ as the scalars
\[
\theta_{X^+_s}=(-1)^{(s-1)}q^{(s^2-1)/2}\ \ \text{and}\ \ \theta_{X^-_s}=-e^{3\pi i p/2}q^{(s^2-1)/2}.
\]
In particular, $\theta_{X^+_2}=-q^{3/2}$.
\end{lemma}

\begin{proof}
One simply checks the value of $e^{2\pi i L(0)}$ on the vector $v^\pm_s$ of lowest conformal weight in $X^\pm_s$,
\[
e^{2\pi i L(0)}v^\pm_s=e^{2\pi i h^\pm_s}v^\pm_s.
\]
We plug in the weights $h^\pm_s$ to obtain
\[
e^{2\pi i h^+_s}=\exp(\frac{\pi i(s^2-1)}{2p})\exp(-\pi i(s-1))=(-1)^{(s-1)}q^{(s^2-1)/2}
\]
and
\[
e^{2\pi i h^-_s}=\exp(\pi i\frac{3}{2}p+\pi i\frac{1}{2p}(s^2-1)-2\pi is+\pi i)=-e^{3\pi i p/2}q^{(s^2-1)/2},
\]
as claimed.
\end{proof}

One can use the above formula for the twist to establish modularity of the category $\rep(\mcl{W}_p)$ of triplet modules.

\begin{theorem}\label{thm:modular}
The category $\rep(\mcl{W}_p)$ is a non-degenerate, and hence modular, tensor category.
\end{theorem}

\begin{proof}
By naturality of the braiding, if a given object $X$ is M\"uger central, then all of its simple composition factors are also M\"uger central.  Additionally, the unit object in a finite tensor category admits no extensions \cite[Theorem 2.17]{etingofostrik04}.  So $Z_{\text{\rm M\"ug}}(\rep(\mcl{W}_p))=Vect$ if and only if the only M\"uger central simple is the unit object.  So we can simply check the braidings of $X^+_2$ against the simples, using the formula for the twist given in Lemma \ref{lem:wp_tw}, to observe the claimed triviality of the M\"uger center.  Let us compute.
\par

The fusion
\[
X^+_2\ot X^{\pm}_s\cong X^{\pm}_{s-1}\oplus X^{\pm}_{s+1},
\]
at $1< s<p$, implies that the twist $\theta_{X^+_2\ot X^{+}_s}$ has eigenvalues $\theta_{X^+_{s-1}}=(-1)^{s}q^{(s^2-2s)/2}$ and $\theta_{X^+_{s+1}}=(-1)^{s}q^{(s^2+2s)/2}$.  Similarly $\theta_{X^+_2\ot X^{-}_s}$ has eigenvalues $\theta_{X^-_{s-1}}=-e^{3\pi i p/2}q^{(s^2-2s)/2}$ and $\theta_{X^+_{s+1}}=-e^{3\pi i p/2}q^{(s^2+2s)/2}$.  Since $q^s=q^{-s}$ if and only if $s=p$, we see that each twist $\theta_{X^+_2\ot X^{\pm}_s}\in \End_{W_p}(X^+_2\ot X^{\pm}_s)$ has two distinct eigenvalues when $1<s<p$.  The competing endomorphism $\theta_{X^+_2}\ot \theta_{X^{\pm}_s}$ acts as a scalar, so that the square braiding
\[
\theta_{X^+_2\ot X^{\pm}_s}\circ(\theta^{-1}_{X^+_2}\ot \theta^{-1}_{X^{\pm}_s})=c^2_{X^+_2,X^{\pm}_s}
\]
is \emph{not} the identity at $1<s<p$.  It follows that no such simple $X^{\pm}_s$ is M\"uger central.
\par

At $s=1$, $\pm =-$, we have $X^+_2\ot X^{-}_1\cong X^-_2$ and $\theta_{X^-_2}=-e^{3\pi i p/2}q^{3/2}$ while $\theta_{X^+_2}\ot \theta_{X^-_1}=e^{3\pi ip/2}q^{3/2}$.  Hence
\[
-1=\theta_{X^+_2\ot X^{-}_1}\circ(\theta^{-1}_{X^+_2}\ot \theta^{-1}_{X^{-}_1})=c^2_{X^+_2,X^{-}_1},
\]
and we conclude that $X^-_1$ is not M\"uger central.
\par

Finally, for $X^{\pm}_p$ we note that $\theta_{X^+_2\ot X^{\pm}_p}$ has an eigenvalue $-e^{3\pi ip/2}$ or $1$, while $\theta_{X^+_2}\ot \theta_{X^{\pm}_p}$ is scaling by
\[
(-1)^pq^{(p^2-1)/2}q^{3/2}%=q^{3p^2/2}q
=e^{3\pi ip/2}q\ \ \text{or}\ \ e^{3\pi i p/2}q^{(p^2-1)/2}q^{3/2}%=e^{3\pi i p/2}e^{\pi i p/2}q
=q,
\]
respectively.  From this information we conclude, just as above, that the square braiding $c^2_{X^+_2,X^{\pm}_p}$ has an eigenvalue $\pm q^{-1}$, so that the objects $X^{\pm}_p$ are not M\"uger central.  This exhausts all non-trivial simples in $\rep(\mcl{W}_p)$, and we conclude that the M\"uger center of $\rep(\mcl{W}_p)$ is precisely $Vect$.
\end{proof}

\section{Ostrik's theorem}
\label{sect:ostrik}

We recall a theorem of Ostrik, which classifies tensor functors $\rep\SL(2)_\zeta\to \msc{A}$ from quantum $\SL(2)$ into an arbitrary tensor category $\msc{A}$.
\par

Our presentation is relatively robust, and we discuss various mechanisms employed in the proof of Ostrik's theorem in detail.  In particular, we recall the construction of the Temperley-Lieb category $\TL(d)$ and its universal properties.  The Temperley-Lieb category will play a prominent role in many of the arguments in the sections that follow.

\subsection{Ostrik's theorem}
\label{sect:thm}

By a self-dual object $W$ in a tensor category $\msc{A}$ we mean a triple $(W,coev,ev)$ of an object $W$, and maps $coev:\1\to W\ot W$ and $ev:W\ot W\to \1$ which identify $W$ as its own left and right dual (in the sense of \cite[\S 2.10]{egno15}).

\begin{definition}
For a self-dual object $W$ in a tensor category $\msc{A}$, the intrinsic dimension of $W$ is the composite endomorphism
\[
d(W):=(\1\overset{coev}\longrightarrow W\ot W\overset{ev}\longrightarrow \1).
\]
\end{definition}

Since $\End_\msc{A}(\1)=\mbb{C}id_\1$, we can identify the endomorphism $d(W)$ with a scalar.  Also, when $W$ is simple, this scalar $d(W)$ is independent of choice of structure maps for $W$, and so is an invariant of $W$ as an object in $\msc{A}$.  When $W$ is \emph{not} simple however, this dimension does depend on the choices of coevaluation and evaluation.  So calling such a dimension intrinsic is a slight abuse of language in this case.
\par

\begin{remark}
More generally, when $W$ is simple and not self-dual, all that can be defined is an intrinsic squared dimension. When $\msc{A}$ has a pivotal structure, the square of the intrinsic dimension of $W$ equals the product of the left and right categorical dimensions of $W$.
\end{remark}

For a self-dual object $W$ of intrinsic dimension $d(W)$, there is a unique-up-to-inversion scalar $\zeta\in \mbb{C}^\times$ so that $d(W)=-(\zeta+\zeta^{-1})$.  Hence the order $\operatorname{ord}(\zeta^2)$ is an invariant of $W$ as a self-dual object, or simply as an object in $\msc{A}$ when $W$ is simple.  Let us call this number the quantum order of $\msc{A}$ at $W$, or just the quantum order of $W$.
\par

As we explain below, when $W$ is of (finite) quantum order $p$, the coevaluation and evaluation maps provide a collection of self-dual objects $W_1=\1, W_2=W,\ W_3,\dots, W_p$ in $\msc{A}$ and splittings
\[
W\ot W_r\cong W_{r-1}\oplus W_{r+1},\ \ \text{for all }r<p.
\]
Each object $W_r$ is of intrinsic dimension $\pm (\zeta^{r-1}+\zeta^{r-3}+\dots+\zeta^{-r+1})$ for appropriate choice of signs.  In particular, $d(W_p)=0$.  We say $W$ is \emph{non-reduced} (resp.\ \emph{reduced}) if the object $W_p$ is non-vanishing (resp.\ vanishes) in $\msc{A}$.

\begin{theorem}[{Ostrik's Theorem \cite[\S 2.6]{ostrik08}}]\label{thm:ostrik}
Suppose $\msc{A}$ is a tensor category, and $W$ is a self-dual object in $\msc{A}$ which is of intrinsic dimension $-(\zeta+\zeta^{-1})$.  Suppose also that $\ord(\zeta^2)<\infty$, and that $W$ is non-reduced.  Then there is a uniquely associated (exact) tensor functor
\[
F_W:\rep\SL(2)_\zeta\to \msc{A}\ \ \text{with }F_W(\mbb{V})=W.
\]
Furthermore, any tensor functor $F:\rep\SL(2)_\zeta\to \msc{A}$ is isomorphic to $F_W$ for some non-reduced, self-dual object $W$.
\end{theorem}

As we recall below, the standard representation $\mbb{V}$ in $\rep\SL(2)_\zeta$ is self-dual, and the equality $F_W(\mbb{V})=W$ in the statement of Theorem \ref{thm:ostrik} should be interpreted as an identification of self-dual objects.  So, $F_W$ is required to send the (co)evaluation maps for $\mbb{V}$ to the (co)evaluation maps for $W$.  Theorem \ref{thm:ostrik} appears precisely in \cite[Remark 2.10]{ostrik08}.  We sketch the proof of Ostrik's theorem below.
\par

For the sake of consistency, we call a self-dual object non-reduced whenever $\ord(\zeta^2)=\infty$ as well.  In the infinite order (aka generic) case, the appropriate analog of the above theorem holds.  Namely, a self-dual object $W$ in $\msc{A}$ which is of infinite quantum order uniquely specifies a map $F_W:\rep\SL(2)_\zeta\to \msc{A}$.  This generic version of Theorem \ref{thm:ostrik} is well-known, and is easily deduced from our presentation.

\begin{remark}\label{rem:739}
Consider the situation presented in Theorem \ref{thm:ostrik}.  In the case in which $W$ is self-dual and \emph{reduced}, with $d(W)=-(\zeta+\zeta^{-1})$, we obtain a uniquely associated tensor functor $\overline{F}_W:\msc{C}(\SL(2),\zeta)\to \msc{A}$ from the semisimplified representation category of $\SL(2)$ at $\zeta$.   The category $\msc{C}(\SL(2),\zeta)$ is explicitly the quotient of the category of tilting modules in $\rep\SL(2)_\zeta$ by the ideal of negligible morphisms \cite[\S 3.3]{bakalovkirillov01}.  An example of a tensor category $\msc{A}$ realizing this case is the modular fusion category rep$(L_{\widehat{\mfk{sl}_2}}(k,0))$ with $\zeta=e^{\pi i/(k+2)}$, where $L_{\widehat{\mfk{sl}_2}}(k,0)$ is the rational $C_2$-cofinite VOA associated to affine $\mfk{sl}_2$ at level $k\in\mbb{Z}_{>0}$ \cite[Chapter 6]{lepowskyli}. So, a more complete recounting of Ostrik's theorem would include the reduced case as well.  This point, however, is inessential for our study.
\end{remark}

\subsection{The Temperley-Lieb category}

Consider $\TL(d)$, the Temperley-Lieb category at parameter $d\in \mbb{C}$ \cite{kauffmanlins94} \cite[Appendix]{freedman03}.  This category has objects $[n]$, for each non-negative integer $n$, and morphisms $\Hom_{\TL}([m],[n])$ given as the $\mbb{C}$-linear span of non-crossing pairings (non-crossing planar string diagrams) between $m$ points and $n$ points, up to isotopy.  So, a general morphism from $[5]$ to $[7]$ is a sum of diagrams which look like
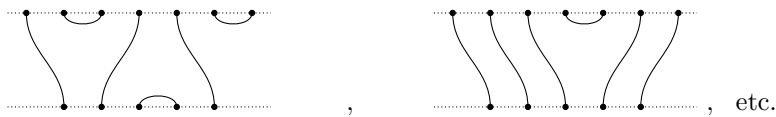
\[
\begin{tikzpicture}[scale=0.5]
  \draw[densely dotted] (-3.5, 0.0) -- (3.5,0.0);
  \draw[densely dotted] (-3.5, 2.5) -- (3.5,2.5);
  \foreach \i in {-2,...,2} {
    \draw[fill] (\i, 0) circle (0.07);
  }
  \foreach \i in {-3,...,3} {
    \draw[fill] (\i, 2.5) circle (0.07);
  }
  \draw (-2, 0) to[out=90, in=270] (-3,2.5);
  \draw (-1, 0) to[out=90, in=270] (0,2.5);
  \draw (0, 0) to[out=90, in=90] (1,0);
  \draw (2, 0) to[out=90, in=270] (1,2.5);

  \draw (-2, 2.5) to[out=270, in=270] (-1,2.5);
  \draw (2, 2.5) to[out=270, in=270] (3,2.5);
\end{tikzpicture}
\hspace{1cm},\hspace{1cm}
\begin{tikzpicture}[scale=0.5]
  \draw[densely dotted] (-3.5, 0.0) -- (3.5,0.0);
  \draw[densely dotted] (-3.5, 2.5) -- (3.5,2.5);
  \foreach \i in {-2,...,2} {
    \draw[fill] (\i, 0) circle (0.07);
  }
  \foreach \i in {-3,...,3} {
    \draw[fill] (\i, 2.5) circle (0.07);
  }
  \draw (-2, 0) to[out=90, in=270] (-3,2.5);
  \draw (-1, 0) to[out=90, in=270] (-2,2.5);
  \draw (0, 0) to[out=90, in=270] (-1,2.5);
  \draw (1, 0) to[out=90, in=270] (2,2.5);
  \draw (2, 0) to[out=90, in=270] (3,2.5);

  \draw (0, 2.5) to[out=270, in=270] (1,2.5);
\end{tikzpicture}\ ,\ \ \text{etc}.
\]
Composition is given by stacking diagrams vertically, where we reduce any diagram with loops to a diagram without loops by replacing each loop with the scalar $d$, $\bigcirc=d\cdot id\in \Hom_{\TL}([0],[0])$.
\par

The category $\TL(d)$ is a $\mbb{C}$-linear, rigid monoidal category with products given by placing diagrams alongside each other horizontally, so that $[n]\ot [m]=[n+m]$.  The generating object $[1]$ is self-dual with coevaluation and evaluation given by the cup and cap morphisms
\[
coev_{[1]}=\cup\in \Hom_{\TL}([0],[2]),\ \ ev_{[1]}=\cap\in \Hom_{\TL}([2],[0]).
\]
Hence all products $[m]=[1]^{\ot m}$ are also self-dual, and the object $[1]$ is of intrinsic dimension
\[
d([1])=ev_{[1]}\circ coev_{[1]}=\bigcirc=d.
\]
The category $\TL(d)$ has a universal property among monoidal categories equipped with a chosen self-dual object.  We recall this universal property momentarily, after introducing a preferred additive completion of $TL(d)$.
\par

Let $\mcl{TL}(d)$ denote the Karoubi, followed by additive, completion of the Temperley-Lieb category
\[
\mcl{TL}(d)=Kar(\TL(d))^\oplus.
\]
So, we construct $\mcl{TL}(d)$ by splitting idempotents in $\TL(d)$, then adding in finite sums of objects.

\begin{theorem}[{\cite[Theorem 4.1.1]{freydyetter89}, \cite[Lemma 6.1]{yamagami}}]
For $\msc{A}$ a $\mbb{C}$-linear, pre-additive monoidal category, and $W$ a self-dual object in $\msc{A}$, there is a unique linear monoidal functor $F_W:\TL(d)\to \msc{A}$ with $F_W([1])=W$, $F_W(coev_{[1]})=coev_W$, and $F_W(ev_{[1]})=ev_W$.
\par

When $\msc{A}$ is a tensor category, the aforementioned universal map extends uniquely to an additive monoidal functor $F_W:\mcl{TL}(d)\to \msc{A}$.
\end{theorem}

By ``unique" we mean unique up to natural isomorphism of monoidal functors.

\begin{proof}
As an alternative to the cited texts, one can note that $\TL(d)$ admits a monoidal presentation by generators and relations \cite[\S 2.2]{turaev90}, with generating object $[1]$ and generating morphisms given by evaluation and coevaluation.  One then applies \cite[Lemma 4.3.1]{turaev90} to obtain the claimed universal property.
\end{proof}

\subsection{The self-dual generator for $\SL(2)_\zeta$}

The object $\mbb{V}$ in $\SL(2)_\zeta$ is self-dual, with (co)evaluation maps
\[
\begin{array}{rl}
coev:\1\to \mbb{V}\ot \mbb{V},& coev(1)=\zeta^{-1/2}v_1\ot v_{-1}-\zeta^{1/2}v_{-1}\ot v_1\\
ev:\mbb{V}\ot \mbb{V}\to \1,& ev(v_1\ot v_1)=ev(v_{-1}\ot v_{-1})=0\\
 & ev(v_1\ot v_{-1})=-\zeta^{-1/2}\\
 & ev(v_{-1}\ot v_1)=\zeta^{1/2}.
\end{array}
\]
One deduces these morphisms via the (co)evaluation maps for the usual dual $\mbb{V}^\ast$ and the explicit $\SL(2)_\zeta$-isomorphism
\[
\phi:\mbb{V}\overset{\cong}\to\mbb{V}^\ast,\ \ \phi(v_1)=-\zeta^{-1/2} v^{-1},\ \phi(v_{-1})=\zeta^{1/2}v^1,
\]
where $v^i$ denotes the dual function to $v_i$ in the above expression.  We then calculate the dimension
\[
d(\mbb{V})=ev\circ coev=-(\zeta+\zeta^{-1}).
\]
The following result is fundamental.

\begin{theorem}[{\cite[Lemma A.7]{eliaslibedinsky17} \cite{ostrik08}}]\label{thm:Tq}
For $d=-(\zeta+\zeta^{-1})$, the map $F_\mbb{V}:\mcl{TL}(d)\to \rep\SL(2)_\zeta$, determined by the self-dual object $\mbb{V}$, is fully faithful, and restricts to an equivalence
\[
F_\mbb{V}:\mcl{TL}(d)\overset{\sim}\to \mcl{T}_\zeta
\]
onto the subcategory of tilting objects in $\rep\SL(2)_\zeta$.
\end{theorem}

While the tilting category $\mcl{T}_\zeta$ in $\rep\SL(2)_\zeta$ admits a rich representation theoretic structure (see e.g.\ \cite{andersen92,andersenparadowski95,sawin06,andersenstroppeltubbenhauer17}), for us it is simply the full, additive, Karoubian, monoidal subcategory in $\rep\SL(2)_\zeta$ generated by $\mbb{V}$.  This category is all of $\rep\SL(2)_\zeta$ precisely when $\zeta$ has infinite order.

\begin{proof}[Sketch proof]
Fully faithfulness of the initial functor $F_\mbb{V}:\TL(d)\to \rep\SL(2)_\zeta$ follows from \cite[Lemma A.7]{eliaslibedinsky17} via base change, and implies that the corresponding map from the completion $\mcl{TL}(d)$ is also fully faithful.  Since the indecomposable objects in $\mcl{T}_\zeta$ are summands of the powers $F_\mbb{V}([n])= \mbb{V}^{\ot n}$ of the standard representation \cite[Proposition 4]{sawin06}, the image of $\mcl{TL}(d)$ is precisely the subcategory $\mcl{T}_\zeta$ of tilting objects.
\end{proof}

\begin{comment}
\begin{lemma}
Consider any point $\mbb{Q}[v,v^{-1}]\to K$, and let $q\in K$ be the image of $v$.  Let $\rep \SL(2)_v$ denote representations of the quantum group $\mbb{Q}[v,v^{-1}]$, and $\tilde{\mbb{V}}$ denote the $2$-dimensional simple in this category \cite{lusztig93}.  The map
\[
K\ot_{\mbb{Q}[v,v^{-1}]}\Hom_{\SL(2)_v}(\tilde{\mbb{V}}^{\ot n},\tilde{\mbb{V}}^{\ot m})\to \Hom_{\SL(2)_q}(\mbb{V}^{\ot n},\mbb{V}^{\ot m})
\]
is an isomorphism.
\end{lemma}

Here by $\tilde{\mbb{V}}$ (resp.\ $\tilde{V}_j$) we mean the $2$-dimensional (resp.\ $j$-dimensional) representation which is cyclically generated by a highest weight vector $v_1$ (resp.\ $v_{j-1}$) of weight $f\in P$ (resp.\ $(j-1)f\in P$).

\begin{proof}
Via duality it suffices to show that the base change of the map
\[
K\ot_{\mbb{Q}[v,v^{-1}]}\Hom_{\SL(2)_v}(\1_v,\tilde{\mbb{V}}^{\ot n})\to \Hom_{\SL(2)_q}(\1,\mbb{V}^{\ot n})
\]
is an isomorphism.  As we have a decomposition
\[
\tilde{\mbb{V}}^{\ot n}\cong \oplus_{j\leq n} \tilde{V}_j^{n_j}
\]
into generic simples [citation], it suffices to show further that the map
\[
\dim\Hom_{\SL(2)_q}(\1,V_j)=\delta_{1,j},
\]
i.e.\ that $V_j$ has no invariant vectors which $j>1$.  But this is clear by a direct analysis of this module 
\end{proof}
\end{comment}

From the above identification $\mcl{TL}(d)\cong \mcl{T}_\zeta$ we deduce a universal property for the subcategory of tilting modules in $\rep\SL(2)_\zeta$.

\begin{corollary}[{\cite[Theorem 2.4]{ostrik08}}]\label{cor:809}
For any tensor category $\msc{A}$, and self-dual object $W$ in $\msc{A}$ of dimension $d(W)=-(\zeta+\zeta^{-1})$, there is a unique additive monoidal functor $F_W:\mcl{T}_\zeta\to \msc{A}$ with $F_W(\mbb{V})=W$, $F_W(coev_{\mbb{V}})=coev_W$, and $F_W(ev_{\mbb{V}})=ev_W$.
\end{corollary}

As stated above, $\mcl{T}_\zeta=\rep\SL(2)_\zeta$ when $\ord(\zeta)=\infty$.  So Corollary \ref{cor:809} already implies the generic version of Ostrik's theorem.

 When $\zeta$ of finite order $\ord(\zeta^2)=p$, recall that we have the $p$ simple $\SL(2)_\zeta$-representations $V_s$ of respective highest weights $(s-1)\alpha/2$, for $s\leq p$, as in \eqref{eq:Vs}.  These simples $V_s$ all lie in the subcategory of tilting modules in $\rep\SL(2)_\zeta$.  It follows that for any self-dual object $W$ in $\msc{A}$ which is of dimension $d(W)=-(\zeta+\zeta^{-1})$, the functor $F_W$ defines distinguished objects $W_s$ in $\msc{A}$ by taking
\begin{equation}\label{eq:Ws}
W_1=\1,\ W_2=W,\ \text{and}\ W_s=F_W(V_s)\ \ \text{for all }s\leq p.
\end{equation}
These $W_s$ are alternatively defined via certain idempotent endomorphisms in $\TL(d)$ \cite{wenzl87,jones97}.

\subsection{The proof of Ostrik's theorem}
\label{sect:Oproof}

We paraphrase the proof from \cite{ostrik08}.  Consider $W$ a self-dual object in $\msc{A}$ of dimension $d(W)=-(\zeta+\zeta^{-1})$, and $\ord(\zeta^2)=p<\infty$.  Suppose also that $W$ is non-reduced.  In the arguments below, by a tensor triangulated category we simply mean a triangulated category with a compatible rigid monoidal structure.
\par

Fix $F=F_W:\mcl{T}_\zeta\to \msc{A}$.  We consider the bounded homotopy category $K^b(\mcl{T}_\zeta)$ of tilting complexes and have maps
\begin{equation}\label{eq:611}
A:K^b(\mcl{T}_\zeta)\to D^b(\SL(2)_\zeta)\ \ \text{and}\ \ RF:K^b(\mcl{T}_\zeta)\to D^b(\msc{A})
\end{equation}
induced by the inclusion $\mcl{T}_\zeta\to \rep\SL(2)_\zeta$ and the additive map $F:\mcl{T}_\zeta\to \msc{A}$.  Here $D^b(\msc{C})$ denotes the bounded derived category of the given abelian category $\msc{C}$, and $RF$ specifically denotes the composite
\[
K^b(\mcl{T}_\zeta)\overset{K^bF}\longrightarrow K^b(\msc{A})\overset{\rm localize}\longrightarrow D^b(\msc{A}).
\]
All of the functors of \eqref{eq:611} are maps of tensor triangulated categories.
\par

By a general result of Beilinson, Bezrukavnikov, and Mirkovi\'{c} \cite{beilinsonbezrukavnikovmirkovic04} \cite[Proposition 2.4]{andersenstroppeltubbenhauer17} \cite[Proposition 2.7]{ostrik08}, the functor $A:K^b(\mcl{T}_\zeta)\to D^b(\SL(2)_\zeta)$ is an equivalence of tensor triangulated categories.  We consider the candidate extension of $F$ to all of $\rep\SL(2)_\zeta$ defined via the composite
\begin{equation}\label{eq:722}
\msc{F}:=\left(\rep\SL(2)_\zeta\to D^b(\SL(2))\overset{A^{-1}}\longrightarrow K^b(\mcl{T}_\zeta)\overset{RF}\longrightarrow D^b(\msc{A})\overset{H^0}\longrightarrow \msc{A}\right).
\end{equation}
The functor $\msc{F}$ is a perfectly well-defined map of additive categories, and satisfies $\msc{F}|_{\mcl{T}_\zeta}=F$.  We must argue now that $\msc{F}$ is exact, and preserves tensor products, provided $F(V_p)\neq 0$.
\par

We claim that, for any $V$ in $\rep\SL(2)_\zeta$,
\begin{equation}\label{eq:629}
\text{the cohomology $H^\ast(RF\circ A^{-1}(V))$ vanishes in all nonzero degrees}.
\end{equation}
If we can verify this claim, then the tensor map $\rep\SL(2)_\zeta\to D^b(\msc{A})$ appearing in \eqref{eq:722} has image in the abelian subcategory $D^b(\msc{A})^\heartsuit$ of objects with cohomology concentrated in degree $0$, and it follows that $\msc{F}$ is an exact, $\mbb{C}$-linear, monoidal functor as desired.  Indeed, in this case we have an exact tensor functor $\rep\SL(2)_\zeta\to D^b(\msc{A})^\heartsuit$ and composing with the tensor equivalence $H^0:D^b(\msc{A})^\heartsuit\overset{\cong}\to \msc{A}$ we see that $\msc{F}$ is exact and monoidal.
\par

As all objects in $\rep\SL(2)_\zeta$ are obtainable from the simples via extension, it suffices to prove the desired vanishing \eqref{eq:629} for all of the simples $L(\lambda)$ in $\rep\SL(2)_\zeta$.  As explained in \cite[\S 2.6]{ostrik08}, it furthermore suffices to prove the desired vanishing (only) at the $2$-dimensional simple $L(\frac{p}{2}\alpha)$ in $\rep\SL(2)_\zeta$.  For this final claim one writes out $L(\frac{p}{2}\alpha)$ as a $3$-term complex $A^{-1}(L(\frac{p}{2}\alpha))=0\to L\to \mbb{V}\ot L'\to L\to 0$ of tilting modules, and observes that after tensoring with the projective $V_p$, the complex $A^{-1}(L(\frac{p}{2}\alpha))\ot V_p$ is isomorphic to a tilting module $L''$ concentrated in degree $0$ in $K^b(\mcl{T}_\zeta)$ \cite[Sublemma 1 \S 2.5]{ostrik08}.  We therefore calculate
\[
RF\big(A^{-1}(L(\frac{p}{2}\alpha))\big)\ot F(V_p)\cong RF\big(A^{-1}(L(\frac{p}{2}\alpha))\ot V_p\big)\cong F(L''),
\]
and find that the leftmost object has cohomology concentrated in degree $0$.  Via exactness and faithfulness of the product on $\msc{A}$ \cite[Proposition 2.1.8]{bakalovkirillov01}, and the fact that $F(V_p)$ is nonzero in $\msc{A}$ by assumption, we conclude that $RF(A^{-1}(L(\frac{p}{2}\alpha)))$ has cohomology concentrated in degree $0$, as desired.  So we establish the desired vanishing \eqref{eq:629}.
\par

We conclude finally that, when $W_p=F(V_p)\neq 0$ in $\msc{A}$, the map $\msc{F}:\rep\SL(2)_\zeta\to \msc{A}$ of \eqref{eq:722} provides the desired extension of $F:\mcl{T}_\zeta\to \msc{A}$ to all of $\rep\SL(2)_\zeta$.  Uniqueness follows from the fact that, by the above information, any two such extensions $\msc{F}$ and $\msc{F}'$ of $F$ must have isomorphic derived functors, and hence must themselves be isomorphic.

\section{Braidings for $\TL(d)$ and quantum group representations}
\label{sect:calc}

We directly calculate all possible braidings on $\TL(d)$, and on $\rep\SL(2)_\zeta$.  The results of this section are used to determine when a given tensor functor $F_W:\rep\SL(2)_\zeta\to \msc{A}$ into a braided tensor category $\msc{A}$ is in fact a \emph{braided} tensor functor.

\subsection{Calculating braidings for $\TL(d)$}

Consider $d\in\mbb{C}$ and write $d=-(\zeta+\zeta^{-1})$ for some $\zeta\in \mbb{C}^\times$.  Recall that there is a unique solution to the equation $d=-(X+X^{-1})$ up to inversion, so that the choice here is between $\zeta$ and $\zeta^{-1}$.  We define the square root $\zeta^{1/2}$ unambiguously by halving the argument of $\zeta$ and taking the positive square root of the magnitude $|\zeta|$.

Consider the Temperley-Lieb category $\TL(d)$ and write, for any nonnegative integer $m$, $\TL_m(d)=\Hom_{\TL}([m],[m])$.  We have the map $f:[2]\to [2]$ given by
\[
f=\cup\circ \cap=coev\circ ev
\]
and the two elements $\{id_{[2]},f\}$ form a basis for the endomorphism ring $\TL_2(d)$.  When $d\neq 0$ we furthermore have the idempotent $d^{-1}f$ which realizes $[0]$ as a summand of $[2]$ in the Karoubi completion of $\TL(d)$.
\par

In $\TL_3(d)$ one can calculate the (standard) equations
\[
f\circ coev=d\cdot coev,\ \ (f\ot 1)(1\ot coev)=(coev\ot 1),\ \ (1\ot f)(coev\ot 1)=(1\ot coev).
\]
For an element $c\in \TL_2(d)=\End_{\TL}([1]\ot[1])$ let us define
\[
c_{[2],[1]}=(c\ot 1)(1\ot c)\ \ \text{and}\ \ c_{[0],[1]}=id.
\]
The following two lemmas are straightforward, and certainly well-known.

\begin{lemma}\label{lem:660}
There are precisely four solutions $c\in \TL_2(d)$ to the equation
%Suppose an endomorphism $c=af+b$ in $\TL_2(d)$ solves the equation
\begin{equation}\label{eq:786}
c_{[2],[1]}(coev_{[1]}\ot 1)=(coev_{[1]}\ot 1)c_{[0],[1]}
\end{equation}
in $\TL_3(d)$, namely
%in $\TL_3(d)$.  
%Then $a=\pm \zeta^{\pm 1/2}$ and $b=a^{-1}$.  Rather, there are precisely four solutions to the equation \eqref{eq:786} in $\TL_2(d)$,
\[
c=\pm (\zeta^{1/2}f+\zeta^{-1/2})\ \ \text{and}\ \ c=\pm(\zeta^{-1/2}f+\zeta^{1/2}).
\]
\end{lemma}

\begin{proof}
Write  $c=af+b$. We have directly
\[
\begin{array}{l}
c_{[2],[1]}(coev\ot 1)
=(c_{[1],[1]}\ot 1)(1\ot c_{[1],[1]})(coev\ot 1)\vspace{1mm}\\
\hspace{4mm}=a(c_{[1],[1]}\ot 1)(1\ot f)(coev\ot 1)+b(c_{[1],[1]}\ot 1)(coev\ot 1)\vspace{1mm}\\
\hspace{4mm}=a^2(f\ot 1)(1\ot f)(coev\ot 1)+ ab(1\ot f+f\ot 1)(coev\ot 1)+b^2(coev\ot 1)\vspace{1mm}\\
\hspace{4mm}=(a^2+b^2)(coev\ot 1)+ ab(1\ot f+f\ot 1)(coev\ot 1)\vspace{1mm}\\
\hspace{4mm}=(a^2+b^2+dab)(coev\ot 1)+ab(1\ot coev).
\end{array}
\]
Since the two representing diagrams for $(coev\ot 1)$ and $(1\ot coev)$ are non-isotopic, and hence these maps are linearly independent in $\TL_3(d)$, the identity \eqref{eq:786} implies the equations
\[
ab=1\ \Rightarrow\ b=a^{-1}\ \ \text{and subsequently}\ \ d=-(a^2+a^{-2}).
\]
The final equation gives $a^2=\zeta^{\pm 1}$.  Hence $a=\pm \zeta^{\pm 1/2}$.
\end{proof}

Note that for any braiding $c$ on $\TL(d)$ the endomorphism $c_{[1],[1]}\in \TL_2(d)$ solves the equation \eqref{eq:786}.  Note also that any braiding on $\TL(d)$ is specified uniquely by its value on $[1]\ot [1]$.  So the above lemma says that there are precisely four possible braidings on $\TL(d)$.  In the above formulas, inverting $\zeta$ corresponds to replacing the braiding by its opposite. Changing the sign of the braiding corresponds to choosing a different square root of $\zeta$.  The existence of the Kauffman bracket implies that all four possibilities are realized as braidings on $\TL(d)$.

\begin{proposition}\label{prop:braidingsTL}
The category $\TL(d)$ has precisely four braidings, and precisely two braidings modulo inversion.  These two braidings (up to inversion) are specified uniquely by their values on $[1]\ot[1]$, and differ only by a sign
\[
c_{[1],[1]}=\pm (\zeta^{1/2}f+\zeta^{-1/2}).
\]
\end{proposition}

\begin{proof}
The Kauffman bracket realizes all four of the possible braidings suggested in Lemma \ref{lem:660} as braidings on $\TL(d)$.  (These braidings are also realized as braidings on $\rep\SL(2)_\zeta$, which one recalls involves a choice of square root for $\zeta$.)  To pair inverses we check
\[
(\zeta^{-1/2}f+\zeta^{1/2})(\zeta^{1/2}f+\zeta^{-1/2})=(d+\zeta+\zeta^{-1})f+1=1.
\]  
\end{proof}

When $d\neq 0$, we have the two idempotents $e_1=d^{-1}f$ and $e_3=1-d^{-1}f$ which provide an alternate basis for the endomorphisms $\TL_2(d)$.  One translates directly to find $\zeta^{1/2}f+\zeta^{-1/2}=-\zeta^{3/2}e_1+\zeta^{-1/2}e_3$.  So, in terms of this basis of idempotents, Proposition \ref{prop:braidingsTL} yields the following.

\begin{proposition}\label{prop:braidingsTL2}
When $d\neq 0$, the category $\TL(d)$ has precisely two braidings, modulo inversion, which are specified by the values
\[
c_{[1],[1]}=\pm(-\zeta^{3/2}e_1+\zeta^{-1/2}e_3).
\]
\end{proposition}

\subsection{Calculations with the generator $\mbb{V}$ in $\rep\SL(2)_\zeta$}

Via the equivalence $K^b(\mcl{TL}(d))\cong D^b(\SL(2)_\zeta)$ we observe a bijection between braidings on $\rep\SL(2)_\zeta$ and braidings on the Temperley-Lieb category (see \S \ref{sect:Oproof}).  In particular, restricting along the fully faithful functor $F_\mbb{V}:\TL(d)\to \rep\SL(2)_\zeta$ induces a bijection between the respective collections of braidings.  So we may compare the standard braiding of Section \ref{sect:SL2q}, for the category of quantum group representations, with the possibilities of Proposition \ref{prop:braidingsTL}.  We perform some calculations with the generator $\mbb{V}$.
\par

Consider the braiding $c$ on $\rep\SL(2)_\zeta$ determined by the $R$-matrix \eqref{eq:R}.  Since $\mbb{V}$ is annihilated by the second powers of $E$ and $F$, we have
\[
c_{\mbb{V},\mbb{V}}(v\ot v') =\zeta^{-(\deg(v),\deg(v')}v'\ot v-\zeta^{-(\deg(v),\deg(v'))}(\zeta-\zeta^{-1})Fv'\ot Ev
\]
so that
\[
c_{\mbb{V},\mbb{V}}=\left\{
\begin{array}{l}
v_1\ot v_1\mapsto \zeta^{-1/2}v_1\ot v_1,\\
v_{-1}\ot v_{-1}\mapsto \zeta^{-1/2}v_{-1}\ot v_{-1},\\
v_1\ot v_{-1}\mapsto \zeta^{1/2}v_{-1}\ot v_1,\\
v_{-1}\ot v_1\mapsto \zeta^{1/2}v_1\ot v_{-1}-\zeta^{1/2}(\zeta-\zeta^{-1})v_{-1}\ot v_1.
\end{array}\right.
\]
Also, the element $f_\mbb{V}=coev\circ ev\in \End_{\rep\SL(2)_\zeta}(\mbb{V}\ot \mbb{V})$ is such that
\[
\begin{array}{c}
f_\mbb{V}(v_1\ot v_1)=f_\mbb{V}(v_{-1}\ot v_{-1})=0,\\
f_\mbb{V}(v_1\ot v_{-1})=-\zeta^{-1}v_1\ot v_{-1}+v_{-1}\ot v_1,\ \ f_{\mbb{V}}(v_{-1}\ot v_1)=v_1\ot v_{-1}-\zeta v_{-1}\ot v_1.
\end{array}
\]
One observes  an expression for the standard braiding directly from the above calculations.

\begin{lemma}
$c_{\mbb{V},\mbb{V}}=\zeta^{1/2}f_\mbb{V}+\zeta^{-1/2}$.
\end{lemma}

\section{An analysis of the generator $X^+_2$}
\label{sect:X2}

The object $X^+_2$ is self-dual \cite[Theorem 37]{tsuchiyawood13}, and so the materials of Section \ref{sect:thm} imply the existence of a tensor functor from some quantum group category to $\rep(\mcl{W}_p)$.  As explained in Theorem \ref{thm:ostrik} and Remark \ref{rem:739}, the precise domain of this functor depends on the behaviors of $X^+_2$ in $\rep(\mcl{W}_p)$.  In this section we prove the following

\begin{proposition}\label{prop:dimX2}
The object $X^+_2$ is of intrinsic dimension $d(X^+_2)=-(q+q^{-1})$, and is non-reduced.
\end{proposition}

Proposition \ref{prop:dimX2} implies, via Ostrik's theorem, the existence of a unique tensor functor
\[
F_{X^+_2}:\rep\SL(2)_q\to \rep(\mcl{W}_p)
\]
with $F_{X^+_2}(\mbb{V})=X^+_2$.  This functor is examined in more detail in Section \ref{sect:triplet}.
\par

We note that the dimension $d(X^+_2)$ can be calculated directly by following the analysis of Tsuchiya and Wood \cite[\S 4.2]{tsuchiyawood13}.  We will, however, determine the dimension via abstract arguments which simultaneously address the non-reduced property of $X^+_2$, and also provide information on the braiding for $\rep(\mcl{W}_p)$.  At the conclusion of the section we determine the braiding $c_{X^+_2,X^+_2}$, up to a sign.  We complete our analysis of the braiding in Section \ref{sect:conf_blocks}.

\subsection{Proof of Proposition \ref{prop:dimX2}}

Recall that at $p=2$ the product $X^+_2\ot X^+_2$ is the minimal projective cover of the unit $\1=X^+_1$ \cite[Theorem 37]{tsuchiyawood13}, and the evaluation map $X^+_2\ot X^+_2\to \1$ identifies $\1$ with the cosocle of this projective module.  In particular, evaluation is the unique nonzero map in $\Hom_{\mcl{W}_p}(X^+\ot X^+,\1)$, up to scaling.  This statement also holds at $p>2$, since we have a decomposition $X^+_2\ot X^+_2\cong \1\oplus X^+_3$.  Therefore at arbitrary $p$ there is a unique linear map
\[
{_\1|(-)}:\End_{\mcl{W}_p}(X^+_2\ot X^+_2)\to \End_{\mcl{W}_p}(\1)
\]
defined by composing functions with any nonzero morphism $X^+_2\ot X^+_2\to \1$.  In terms of evaluation for example, this maps is uniquely defined by the property $({_\1|f})\circ ev=ev\circ f$.

\begin{lemma}\label{lem:braid2Wp}
At $p>2$ the natural braiding on $\rep(\mcl{W}_p)$ is such that
\[
c^2_{X^+_2,X^+_2}=q^{-3}e_1+qe_3,
\]
and at all $p$ the corestriction of the braiding along evaluation is ${_\1|c^2_{X^+_2,X^+_2}}=q^{-3}$.
\end{lemma}

\begin{proof}
When $p>2$ we have, by naturality of the twist,
\[
\theta_{X^+_2\ot X^+_2}=\theta_{\1}e_1+\theta_{X^+_3}e_3=e_1+q^4e_3
\]
and $\theta_{X^+_2}^2=q^3$.  Hence
\[
c^2_{X^+_2,X^+_2}=\theta_{X^+_2\ot X^+_2}(\theta_{X^+_2}^{-1}\ot\theta_{X^+_2}^{-1})=q^{-3}e_1+qe_3.
\]
In general, naturality of the twist gives
\[
{_\1|\theta_{X^+_2\ot X^+_2}(\theta_{X^+_2}^{-1}\ot\theta_{X^+_2}^{-1})}=q^{-3}({_\1|\theta_{X^+_2\ot X^+_2}})=q^{-3}\theta_\1=q^{-3}.
\]
\end{proof}

\begin{lemma}\label{lem:qordX_2}
The category $\rep(\mcl{W}_p)$ is of finite quantum order at $X^+_2$ (see \S \ref{sect:thm}).  Furthermore, this order divides $p$.
\end{lemma}

\begin{proof}
Take $\zeta\in \mbb{C}^\times$ so that $d(X^+_2)=-(\zeta+\zeta^{-1})$.  Then by the fusion rule for $\rep(\mcl{W}_p)$, and induction, we observe
\[
d(X^+_r)=\pm(\zeta^{r-1}+\zeta^{r-3}+\dots +\zeta^{-r+1})
\]
for all $1\leq r\leq p$.  But now, the simple $X^+_{p}$ is projective, and so must be of dimension $0$.  To make this point more clear, simply note that any composite $\1\to X^+_p\ot X^+_p\to \1$ must be $0$, since otherwise $\1$ would appear as a summand of the projective object $X^+_p\ot X^+_p$, and would therefore be projective.  But, we know this is not the case.
\par

The vanishing $d(X^+_p)=0$ implies
\[
0=\zeta^{p-1}d(X^+_p)=\sum_{m=0}^{p-1}(\zeta^2)^m\ \ \Rightarrow\ \ 0=(\zeta^2-1)(\sum_{m=0}^{p-1}(\zeta^2)^m)=\zeta^{2p}-1.
\]
Hence $\ord(\zeta^2)$, which is the quantum order of $X^+_2$ in $\rep(\mcl{W}_p)$, is finite and divides $p$.
\end{proof}

\begin{lemma}\label{lem:X2_nr}
The object $X^+_2$ is non-reduced in $\rep(\mcl{W}_p)$ (see \S \ref{sect:thm}).
\end{lemma}

\begin{proof}
Take $\zeta$ such that $d=d(X^+_2)=-(\zeta+\zeta^{-1})$.  We already saw at Lemma \ref{lem:qordX_2} that $\ord(\zeta^2)\leq p$.  Take $p'=\ord(\zeta^2)$.  In the arguments that follow, the fact that $p'$ is less than or equal to $p$ is of significance.
\par

We have the universal map $F_{X_2^+}:\mcl{T}_\zeta\cong \mcl{TL}(d)\to \rep(\mcl{W}_p)$, $F(\mbb{V})=X^+_2$, and one sees by a recursive argument that $F_{X^+_2}(V_s)\cong X^+_s$ for all $s\leq p'$.  Namely, we have $F_{X^+_2}(\mbb{V})=F_{X^+_2}(V_2)=X^+_2$, and if we assume $F_{X^+_2}(V_r)\cong X^+_r$ for all $r<s$ then we have
\[
\begin{array}{rl}
X^+_{s-2}\oplus F_{X^+_2}(V_s)&\cong F_{X^+_2}(\mbb{V}\ot V_{s-1})\\
&\cong F_{X^+_2}(\mbb{V})\ot F_{X^+_2}(V_{s-1})\cong X^+_2\ot X^+_{s-1}\cong X^+_{s-2}\oplus X^+_s. 
\end{array}
\]
Uniqueness of Jordan-Holder series therefore forces $F_{X^+_2}(V_s)\cong X^+_s$, as proposed.  These identifications imply, in particular, that $F_{X^+_2}(V_{p'})$ is isomorphic to the (nonzero) simple $X^+_{p'}$.  So we see that $X^+_2$ is non-reduced.
\end{proof}

We can now prove the claimed result.

\begin{proof}[Proof of Proposition \ref{prop:dimX2}]
We have already seen that $X^+_2$ is non-reduced, in Lemma \ref{lem:X2_nr}.  So we only have to determine the dimension.  When $p=2$ the object $X_2^+\ot X^+_2$ is projective, and thus the intrinsic dimension of $X^+_2$ is $0=-(i-i^{-1})$.  Suppose now $p>2$.
\par

Fix $\zeta$ so that $d=d(X^+_2)=-(\zeta+\zeta^{-1})$.  We want to show that $\zeta=q^{\pm 1}$.  We have the universal map $F_{X^+_2}:\TL(d)\to \rep(\mcl{W}_p)$, $F([1])=X^+_2$.  Since $X^+_2$ is non-reduced, the map $F_{X^+_2}$ extends to a tensor functor from $\rep\SL(2)_q$ and is therefore faithful \cite[Proposition 1.19]{delignemilne82}.  Faithfulness, and the fact that
\[
\dim_\mbb{C}\End_{\mcl{W}_p}(X^+_2\ot X^+_2)=\dim_\mbb{C}\End_{\mcl{W}_p}(X^+_1\oplus X^+_3)=2,
\]
implies that the map on morphisms $F_{X^+_2}:\TL_2(d)\to\End_{\mcl{W}_p}(X^+_2\ot X^+_2)$ is an isomorphism.
\par

Now, by Lemma \ref{lem:660} and Proposition \ref{prop:braidingsTL}, the preimage $c\in \TL_2(d)$ of the element $c_{X^+_2,X^+_2}\in \End_{\mcl{W}_p}(X^+_2\ot X^+_2)$ specifies a unique braiding on $\TL(d)$ under which the functor $F_{X^+_2}$ is braided monoidal.  By the calculation of Lemma \ref{lem:braid2Wp}, this braiding satisfies
\[
c^2_{[1],[1]}=q^{-3}e_1+qe_3.
\]
By the constraints placed on braidings on $\TL(d)$ given in Proposition \ref{prop:braidingsTL2}, we see that $\zeta=q^{\pm 1}$ and $d=d(X^+_2)=-(q+q^{-1})$.
\end{proof}

\subsection{The braiding on $\rep(\mcl{W}_p)$, up to a sign}

In keeping with our general convention, we let $f_{X^+_2}:X^+_2\ot X^+_2\to X^+_2\ot X^+_2$ denote the composite of the evaluation and coevaluation maps
\[
X^+_2\ot X^+_2\to \1\to X^+_2\ot X^+_2.
\]
Since $X^+_2$ is simple, the endomorphism $f_{X^+_2}$ is independent of the specific choice of evaluation and coevaluation.

\begin{lemma}\label{lem:pre_braidWp}
The braiding on $\rep(\mcl{W}_p)$ is specified, up to a sign, as
\begin{equation}\label{eq:864}
c_{X^+_2,X^+_2}=\pm(q^{-1/2}f_{X_2^+}+q^{1/2})
\end{equation}
\end{lemma}

\begin{proof}
At $p>2$ the endomorphism \eqref{eq:864} can be rewritten as $\pm (-q^{-3/2}e_1+q^{1/2}e_3)$ and squares to the appropriate formula according to Lemma \ref{lem:braid2Wp}.  At $p=2$ the square of this endomorphism is $2f+q=2f+q^{-3}$, since $q^4=1$ and $f^2=0$, where we write $f=f_{X_2^+}$.  (At $p=2$ the square of $f$ is zero since $\1$ is not projective, and so does not appear as a summand of the projective object $X^+_2\ot X^+_2$.)
\par

As explained in the proof of Proposition \ref{prop:dimX2}, there is a unique braiding on $\TL(d)$ under which the (faithful) monoidal functor $F_{X^+_2}:\TL(d)\to \rep(\mcl{W}_p)$ is braided monoidal.  By the constraints of Propositions \ref{prop:braidingsTL}, and the fact that ${_\1|c^2_{X^+_2,X^+_2}}=q^{-3}$ by Lemma \ref{lem:braid2Wp}, we see that the only possible values for $c_{X^+_2,X^+_2}$ are those of the form \eqref{eq:864}.
\end{proof}

At general $p$, these two possibilities for the braiding can be distinguished by their compositions along any nonzero projection $X^+_2\ot X^+_2\to \1$.

\begin{corollary}\label{cor:1143}
We have ${_\1|c_{X^+_2,X^+_2}}=\pm(-q^{-3/2})$, with sign $\pm$ corresponding precisely to the sign at \eqref{eq:864}.
\end{corollary}

\begin{proof}
One calculates directly
\[
ev\circ (q^{-1/2}f_{X_2^+}+q^{1/2})=(q^{-1/2}d+q^{1/2})ev=-q^{-3/2}ev.
\]
\end{proof}

\section{The braiding for $\rep(\mcl{W}_p)$}
\label{sect:conf_blocks}

We calculate the braiding on $\rep(\mcl{W}_p)$ explicitly via its vertex tensor structure.

\subsection{Categorical data for strongly-finite VOAs}

Let $\mcl{V}$ be a strongly-finite VOA, and $W\in\rep(\mcl{V})$ be a module (recall the discussion in Section \ref{sect:wpcat}).
Then as in Section \ref{sect:voacat} we obtain the decompositions $\mcl{V}=\coprod_{n=0}^\infty\mcl{V}_{(n)}$ and $W=\coprod_rW_{(r)}$ by $L(0)$.
 For each $v\in \mcl{V}_{(n)}$  we have the vertex operator $Y_W(v,z)=\sum_{m\in\mbb{Z}}v_{m}z^{-m-1}$, where the linear map $v_m$ sends $W_{(r)}$ to $W_{(r+n-m-1)}$.

To understand the tensor product of $\mcl{V}$-modules, we need a generalization of vertex operators called (logarithmic) intertwining operators $\mcl{Y}$. If $W^1,W^2,W^3$ are simple, then an intertwining operator of type $\left({W^3\atop W^1W^2}\right)$ has shape
$\mcl{Y}(w^1,z)=z^{h(W^3)-h(W^1)-h(W^2)}\sum_{m\in\mbb{Z}}w^1_mz^{-m-1}$ for any $w^1\in W^1$, where $w^1_m$ sends $W^2_{(s)}$ to $W^3_{(r-h(W^1)+s-h(W^2)-m-1+h(W^3))}$ when $w^1\in W^1_{(r)}$. More generally, when $W^i$ are merely indecomposable,
each term in the sum will include polynomials in $\log(z)$. More precisely, choosing $w^1\in W^1_{(r)}$ and $w^2\in W^2_{(s)}$, let $k_i$ be the smallest values such that $(L(0)-r)^{k_1}w^1=0=(L(0)-s)^{k_2}w^2$, then the $m$th term in the sum is $z^{h(W^3)-h(W^1)-h(W^2)-m-1}$  times a polynomial in $\log(z)$ of degree  $k_1+k_2+k_3-3$  with coefficients in  $W^3_{(t)}$, where $t=h(W^3)+r-h(W^1)+s-h(W^2)-m-1$ and  $k_3$ is the smallest power such that $(L(0)-t)^{k_3}w^3=0$ for all $w^3\in W^3_{(t)}$.
\par

We denote by $\mcl{V}\left({W^3\atop W^1W^2}\right)$ the space of intertwining operators of that type. Its dimension equals that of Hom$_{\mcl{V}}(W^1\otimes W^2,W^3)$ (\cite[Proposition 4.17]{huanglepowskyzhang10}). Thus when $W^3$ is simple and $W^1\otimes W^2$ is semisimple, 
dim$\,\mcl{V}\left({W^3\atop W^1W^2}\right)$ will equal the multiplicity of $W^3$ in $W^1\ot W^2$. 
\par

The tensor unit is $\mcl{V}$. There is a unique nondegenerate invariant bilinear form on $\mcl{V}$ up to scaling; let $(u,v)$ denote the one normalised so that $(\mathbf{1},\mathbf{1})=1$, where $\mcl{V}_0=\mbb{C}\mathbf{1}$. 
We will need an explicit intertwining operator $\mcl{Y}$ of type $\left({\mcl{V}\atop W\!\ W^*}\right)$ (for $W$ simple, this space is always 1-dimensional). It is defined by
\begin{equation}\label{eq:wwv}( v,\mcl{Y}_{W,W*}^{\mcl{V}}(w,z)w')=\langle  z^{-2\mathrm{wt}\,w}e^{L(-1)/z}Y_W(v,-z^{-1})e^{-zL(1)}w,w'\rangle\end{equation}
for any $v\in\mcl{V}$, $w\in W$ and $w'\in W^*$. The Virasoro operator $L(1)$ lowers weights by 1 (so automatically kills lowest weight vectors), while $L(-1)$ raises them by 1.
\par

For VOAs, the  braiding isomorphism $c_{W^1,W^2}:W^1\otimes W^2\to W^2\otimes W^1$ is skew-symmetry. More precisely, suppose $\mcl{V}$ is strongly-finite and $W^1,W^2$ are simple. When  $W^1\otimes W^2$ is semisimple (i.e. isomorphic to a sum of simples), we can identify
the $\mcl{V}$-module $W^1\otimes W^2$ with $\oplus_{W^3}\mcl{V}\left({W^3\atop W^1W^2}\right)\otimes_{\mbb{C}} W^3$, where the (finite) sum is over inequivalent simple $\mcl{V}$-modules $W^3$. In this case we can interpret $c_{W^1,W^2}$ as a linear map between spaces of intertwining operators  $\mcl{V}\left({W^3\atop W^1W^2}\right)\to \mcl{V}\left({W^3\atop W^2W^1}\right)$:
\begin{equation}\label{eq:voabraiding}c_{W^1,W^2}(\mcl{Y})(w^2,z)w^1=e^{zL(-1)}\mcl{Y}(w^1,e^{\pi i} z)w^2\end{equation}
for any $\mcl{Y}\in\mcl{V}\left({W^3\atop W^1W^2}\right)$, where we're using the skew-symmetry operator $\Omega_r$ defined in \cite[equation (7.1)]{huanglepowskyii}
(for us, $r=0$). The notation $e^{\pi i}$ here indicates the appropriate choice of branch of logarithm, needed to evaluate fractional powers. 
We also use \eqref{eq:voabraiding} when $W^1\otimes W^2$ is not semisimple, though in this case the relation of the intertwining operator to the tensor product can be slightly more subtle.

\subsection{The sign of the braiding}

In Proposition \ref{prop:dimX2} we determined the dimension $d(X^+_2)$ of $X^+_2$ by ``abstract nonsense".  In Lemma \ref{lem:pre_braidWp} we similarly determined the braiding on $\rep(\mcl{W}_p)$ up to a sign, again by ``nonsense".  Here we specify the braiding on the category of triplet modules precisely, by dealing directly with the vertex tensor structure on $\rep(\mcl{W}_p)$.

\begin{proposition}\label{prop:c_XX}
The braiding on $\rep(\mcl{W}_p)$ satisfies $c_{X^+_2,X^+_2}=q^{-1/2}f_{X_2^+}+q^{1/2}$.
\end{proposition}

\begin{proof}
By Corollary \ref{cor:1143}, in order to specify uniquely the braiding of rep($\mcl{W}_p$)
 it suffices to apply  \eqref{eq:voabraiding} to the intertwining operator $\mcl{Y}\in\mcl{W}_p\left({\mcl{W}_p\atop X_2^+X_2^+}\right)$ defined in  \eqref{eq:wwv}.   For all $p$, this $\mcl{Y}$ corresponds to the (unique up to scaling) homomorphism $X_2^+\ot X_2^+\to \mcl{W}_p$.

Take $w^1=w^2=w$, a nonzero vector in the lowest weight space of $X_2^+$, which is 1-dimensional. Note that
$\mcl{Y}(w,z)w\in z^{1-3/2p}(x\mathbf{1}+z\mcl{W}_p[[z]])$ for some $x\in\mbb{C}$.  It is clear from \eqref{eq:wwv} that $x\ne 0$. We find 
 $$c_{X_2^+,X_2^+}(\mcl{Y})(w,z)w\in e^{zL(-1)}e^{\pi i(1-3/2p)}z^{1-3/2p}(x\mathbf{1}+ z\mcl{W}_p[[z]])$$ $$=-e^{-3\pi i/2p}z^{1-3/2p}(x\mathbf{1}+z\mcl{W}_p[[z]])$$   
 This must be a multiple of $\mcl{Y}$ (since dim$\,\mcl{W}_p\left({\mcl{V}\atop X_2^+X_2^+}\right)=1$), and because $x\ne 0$, that multiple is thus $-q^{-3/2}$.
 \end{proof}

For $p=2$ the product $X_2^+\ot X_2^+$ is not semisimple, but rather is the minimal projective cover $P_1^+$ of the unit. This plays no role in the proof, however. When $p=2$, for what it's worth, $\mcl{Y}$ is the part of the full logarithmic intertwining operator of type $\left({P_1^+\atop X_2^+X_2^+}\right)$ coming with the highest powers of $\log(z)$.

For $p>2$ we can determine the rest of the braiding in a similar way.  Take any nonzero $\mcl{Y}\in\mcl{W}_p\left({X_3^+\atop X_2^+X_2^+}\right)$. Then $\mcl{Y}(w,z)w\in z^{1/2p}(w_3+zX_3^+[[z]])$, so
 $$c_{X_2^+,X_2^+}(\mcl{Y})(w,z)w\in e^{zL(-1)}e^{\pi i/2p}z^{1/2p}(w_3+ zX_3^+[[z]])=e^{\pi i/2p}z^{1/2p}(w_3+zX_3^+[[z]])$$   
for some $w_3\in X_3^+$. This must be a multiple of $\mcl{Y}$, and provided $w_3\ne 0$, that multiple is thus $q^{1/2}$, as desired. It is elementary to show that $w_3\ne 0$ using the hypergeometric function calculations in \cite{tsuchiyawood13,creutzigmcraeyang}.  But we will skip the details as  we already knew $q^{1/2}$ is correct from Lemma \ref{lem:pre_braidWp} together with the calculation in the preceding proof.

We note that, by naturality of the braiding and the braid relation, the braiding on the category $\rep(\mcl{W}_p)$ is determined uniquely by its value $c_{X^+_2,X^+_2}$ on the generator, as implied by Lemma \ref{lem:pre_braidWp}.

\section{A modular tensor equivalence $\rep(u_q(\mfk{sl}_2))\cong \rep(\mcl{W}_p)$}
\label{sect:triplet}

At Theorem \ref{thm:triplet} below, we prove that there is an equivalence of (non-semisimple) modular tensor categories $\rep(u_q(\mfk{sl}_2))^{\rm rev}\cong \rep(\mcl{W}_p)$, at quantum parameter $q=e^{\pi i/p}$.  Our proof relies on a number of technical points.  First, we identify tensor functors out of the de-equivariantization $(\rep\SL(2)_q)_{\PSL(2)}$ with a certain class of tensor functors out of $\rep\SL(2)_q$.  We then prove a braided version of Ostrik's theorem, which is given in Theorem \ref{thm:braidedostrik} below.  These results, in conjunction with the analysis of $\rep(\mcl{W}_p)$ provided in Sections \ref{sect:X2} and \ref{sect:conf_blocks}, will imply the claimed equivalence.

\subsection{De-equivariantization as categorical base change}

Suppose $\zeta\in \mbb{C}^\times$ is of even order $2p$, for the sake of specificity.  We say a functor $\operatorname{M}:\rep\SL(2)_\zeta\to \msc{A}$ \emph{annihilates} the central subcategory $\rep\PSL(2)\subset \rep\SL(2)_\zeta$ if the image of $\rep\PSL(2)$ in $\msc{A}$ is precisely the subcategory $Vect\subset \msc{A}$ generated by the unit, and for any $W$ in $\rep\PSL(2)$ and $V$ in $\rep\SL(2)_\zeta$ the diagram
\begin{equation}\label{eq:829}
\xymatrix{
\operatorname{M}(W\ot V)\ar[rr]^{\operatorname{M}(c_{V,W})}\ar[d]_\cong & & \operatorname{M}(V\ot W)\ar[d]^{\cong}\\
\operatorname{M}(W)\ot \operatorname{M}(V)\ar[rr]^{\tau_{\operatorname{M}(W),\operatorname{M}(V)}} & & \operatorname{M}(V)\ot \operatorname{M}(W)
}
\end{equation}
commutes.  Here the vertical maps are the structure maps for $\operatorname{M}$, and $\tau_{\operatorname{M}(W),\operatorname{M}(V)}$ is the half-braiding provided by the trivial lift $Vect\to Z(\msc{A})$ of the unit $Vect\to \msc{A}$.  (Rather, $\tau$ is the half-braiding provided by the unit structure $-\ot\1\cong id\cong \1\ot-$.)  A very practical way to observe the above diagram \eqref{eq:829} is the following.

\begin{lemma}
Given a braided tensor functor $\operatorname{M}:\rep\SL(2)_\zeta\to \msc{A}$, into a braided tensor category $\msc{A}$, $\operatorname{M}$ annihilates $\rep\PSL(2)$ if and only if $\operatorname{M}(\rep\PSL(2))\subset Vect$.
\end{lemma}

\begin{proof}
Compatibility with the braiding implies that the diagram \eqref{eq:829} commutes, after we replace $\tau$ with the braiding $c$ for $\msc{A}$.  But now, the hypotheses on any braiding requires $c_{X,Y}=\tau_{X,Y}$ for all trivial $X$, so that compatibility with the braiding implies the diagram \eqref{eq:829}.
\end{proof}

The following is a general result about de-equivariantization, which we state only in the case of $(\rep\SL(2)_\zeta)_{\PSL(2)}$.  The general situation is discussed in Appendix \ref{sect:C_G}.

\begin{proposition}\label{prop:basechange}
Suppose a tensor functor $\operatorname{M}:\rep\SL(2)_\zeta\to \msc{A}$ annihilates the central subcategory $\rep\PSL(2)\subset \rep\SL(2)_\zeta$.  Then there exists a tensor functor from the de-equivariantization $\mu:(\rep\SL(2)_\zeta)_{\PSL(2)}\to \msc{A}$ which fits into a ($2$-)diagram
\begin{equation}\label{eq:849}
\xymatrix{
\rep\SL(2)_\zeta\ar[rr]^{\operatorname{M}}\ar[dr]_{dE} & & \msc{A}\\
 & (\rep\SL(2)_\zeta)_{\PSL(2)}\ar[ur]_{\mu} & .
}
\end{equation}
When $\msc{A}$ is braided, and $\operatorname{M}$ is a map of braided tensor categories, then $\mu$ can be taken to be a map of braided tensor categories as well.  Furthermore, the collection of all such factorizations $\{\mu\ \text{\rm admitting a diagram }\eqref{eq:849}\}$ admits a natural $\PSL(2)$-action.
\end{proposition}

As is made clear in the appendix, a more careful statement of Proposition \ref{prop:basechange} would specify not only the functor $\operatorname{M}$, but also a choice of trivialization of the restriction $\operatorname{M}|_{\rep\PSL(2)}$.

The factorization $\mu$ can be written explicitly as $\mu=\1\ot_{\operatorname{M}\O}\operatorname{M}(-)$, where $\O=\Fr\O(\PSL(2))$.  The ambiguity appears in the choice of the point (algebra map) $\operatorname{M}\O\to \1$ at which we take the fiber to produce $\mu$.  The collection of such points for $\operatorname{M}\O$ has the structure of a $\PSL(2)$-torsor.

\begin{proof}
Suppose we have such a map $\operatorname{M}:\rep\SL(2)_\zeta\to \msc{A}$ which annihilates $\rep\PSL(2)$.  Then the algebra object $\O=\operatorname{Fr}\O(\PSL(2))$ in $\rep\SL(2)_\zeta$ has image $R=\operatorname{M}\O$, a trivial algebra in $\Ind\msc{A}$.  Thus we may consider the monoidal category $\mrm{mod}_\msc{A}(R)$ of finitely presented $R$-modules in $\msc{A}$.  Given any augmentation $R\to \1$ we have the associated monoidal functor $\1\ot_R-:\mrm{mod}_\msc{A}(R)\to \msc{A}$ for which the composite
\[
\msc{A}\overset{Free}\longrightarrow \mrm{mod}_\msc{A}(R)\overset{\1\ot_R-}\longrightarrow \msc{A}
\]
is isomorphic to the identity.  (Note that we use the canonical central structure on the embedding $Vect\to \msc{A}$ in order to consider $R$-modules in $\msc{A}$ as $R$-bimodules.)
\par

The tensor functor $\operatorname{M}$ induced an exact monoidal functor
\[
\operatorname{M}_\O:(\rep\SL(2)_\zeta)_{\PSL(2)}\to \mrm{mod}_\msc{A}(R),
\]
and we compose with the fiber $\1\ot_R-$ at any (fixed) augmentation $R\to \1$ to obtain the desired factoring of $\operatorname{M}$ through the de-equivariantization $\mu=(\1\ot_R-)\circ \operatorname{M}_\O$.  (We note that, by uniqueness of the fiber functor for $\rep\PSL(2)$ \cite{delignemilne82}, there is an algebra isomorphism $R\cong forget\O(\PSL(2))$, and any point of $\PSL(2)$ subsequently specifies an algebra map $R\to \1$.)  The functor $\mu$ is monoidal, as it is a composition of monoidal functors, and we claim that it is exact.
\par

For exactness, note that the reduction $\1\ot_R-:\mrm{mod}_\msc{A}(R)\to \msc{A}$ admits a right adjoint $\msc{A}\to \mrm{mod}_\msc{A}(R)$ which sends an object $V$ in $\msc{A}$ to the $R$-module in $\msc{A}$ which is simply $V$, as an object in $\msc{A}$, with $R$ acting through the augmentation $R\to \1$.  As any functor which has a right adjoint is right exact, we see that the reduction functor is right exact, and hence that $\mu$ is right exact as it is a composition of right exact functors.  Now, compatibility with duality \cite[Exercise 2.10.6]{egno15} implies that $\mu$ is left exact as well.

In the braided context we have that $\mrm{mod}_\msc{A}(R)$ inherits a braiding from that of $\msc{A}$, and the braiding caries through all of the given arguments.  The $\PSL(2)$-action comes from the fact that the $\PSL(2)$-action on $\O(\PSL(2))$, by translation, induces a $\PSL(2)$-action on the algebra $R$.  This $\PSL(2)$-action permutes the points $R\to \1$ at which we define the functor $\mu=\1\ot_R\operatorname{M}_\O(-)$.
\end{proof}

\begin{remark} Note that restriction (or Ostrik's Theorem) gives us a functor from $\rep\SL(2)_\zeta$ to the representation category of the small quantum group (not cocycle corrected), which sends $\rep\PSL(2)$ to $Vect$. However it does not annihilate $\rep\PSL(2)$ in our sense (it does not satisfy \eqref{eq:829}), and won't factorize as in \eqref{eq:849}, so cannot be directly used to establish an equivalence with $(\rep\SL(2)_\zeta)_{\PSL(2)}$. Likewise, the forgetful functor $\rep\SL(2)_\zeta\to Vect$ also does not annihilate $\rep\PSL(2)$ in our sense, so we do not get from it an induced fiber functor for the de-equivariantization. It does admit the structure of a quasi-fiber functor, however. \end{remark}

\subsection{A braided version of Ostrik's theorem}

Consider a self-dual object $W$ in a tensor category $\msc{A}$.  For such $W$ we have an associated endomorphism 
\[
f_W:W\ot W\to \1\to W\ot W,\ \ f_W=coev\circ ev,
\] 
of the second tensor power of $W$.  When the endomorphism algebra $\End_\msc{A}(W)$ is $1$-dimensional the morphism $f_W$ is independent of the choice of structure maps $coev$ and $ev$.  This occurs, for example, when $W$ is simple.

\begin{theorem}\label{thm:braidedostrik}
Suppose $W$ is a self-dual object in a braided tensor category $\msc{A}$, with $d(W)=-(\zeta+\zeta^{-1})$.  Suppose also that $\ord(\zeta)<\infty$, that $W$ is non-reduced, and that the braiding endomorphism $c_{W,W}\in \End_\msc{A}(W\ot W)$ is in the subspace spanned by $f_W$ and the identity.
\par

Then, under one of the four braided structures on $\rep\SL(2)_\zeta$ provided in Proposition \ref{prop:braidingsTL}, the tensor functor $F_W:\rep\SL(2)_\zeta\to \msc{A}$ promised in Theorem \ref{thm:ostrik} is a \emph{braided} tensor functor.
\end{theorem}

\begin{proof}
Since the tensor functor $F_W$ is necessarily faithful, its restriction to the subcategory of tilting modules $F_W|_{\mcl{T}_\zeta}$ is faithful as well.  As in the proof of Proposition \ref{prop:dimX2}, faithfulness and the fact the the braiding endomorphism $c_{W,W}$ is in the image of the map $F_W:\End_{\mcl{T}_\zeta}(\mbb{V}\ot \mbb{V})\to \End_\msc{A}(W\ot W)$ implies that $F_W|_{\mcl{T}_\zeta}$ respects the braidings on $\mcl{T}_\zeta$ and $\msc{A}$, under one of the four braidings on $\mcl{T}_\zeta$ considered in Proposition \ref{prop:braidingsTL2}. Hence the corresponding functor of triangulated categories $R(F_W|_{\mcl{T}_\zeta}):K^b(\mcl{T}_\zeta)\to D^b(\msc{A})$ is also a map of braided monoidal categories.
\par

We recall that the inclusion $\mcl{T}_\zeta\to \rep\SL(2)_q$ induces an equivalence of tensor triangulated categories $K^b(\mcl{T}_\zeta)\overset{\sim}\to D^b(\SL(2)_\zeta)$ \cite{beilinsonbezrukavnikovmirkovic04}, and note the diagram
\[
\xymatrixrowsep{3mm}
\xymatrix{
& K^b(\mcl{T}_\zeta)\ar[dr]^{R(F_W|_{\mcl{T}_\zeta})}\ar[dl]_\sim & \\
D^b(\SL(2)_\zeta)\ar[rr]_{RF_W} & & D^b(\msc{A}).
}
\]
The above diagram implies that the functor $RF_W$ is braided monoidal, since both maps out of $K^b(\mcl{T}_\zeta)$ are braided.  It follows that the original map $F_W$ is a braided tensor functor.
\end{proof}

\subsection{A modular tensor equivalence}

\begin{theorem}\label{thm:triplet}
There is an equivalence of modular tensor categories
\[
\Theta:\rep(u_q(\mfk{sl}_2))^{\rm rev}\overset{\sim}\to \rep(\mcl{W}_p).
\]
The functor $\Theta$ is such that $\Theta(\mbb{V})\cong X^+_2$.
\end{theorem}

The superscript $(-)^{\rm rev}$ indicates that we consider the representation category for $u_q(\mfk{sl}_2)$ with inverse braiding, and inverse twist, relative to that of Section \ref{sect:SL2q}.
 
\begin{proof}
Throughout the proof we consider $\rep\SL(2)_q$ with its inverse ribbon structure.  Since we have the ribbon equivalence $(\rep\SL(2)_q)_{\PSL(2)}\cong \rep(u_q(\mfk{sl}_2))$ of Theorem \ref{thm:N}, it suffices to produce such an equivalence $\Theta$ from the de-equivariantization $(\rep\SL(2)_q)_{\PSL(2)}$.
\par

By Theorem \ref{thm:braidedostrik}, Proposition \ref{prop:dimX2}, and Proposition \ref{prop:c_XX} we have a braided tensor functor
\[
F_W:\rep\SL(2)_q\to \rep(\mcl{W}_p)
\]
with $F_W(\mbb{V})=X_2^+$.  Since $\rep(\mcl{W}_p)$ is generated by $X_2^+$, the functor $F_W$ is surjective.  Consequently, the M\"uger center in $\rep\SL(2)_q$ maps to a M\"uger central subcategory of $\rep(\mcl{W}_p)$.  But by Theorem \ref{thm:modular}, the category of triplet modules is non-degenerate, so that the only central subcategory in $\rep(\mcl{W}_p)$ is $Vect$.  It follows that $F_W(\rep\PSL(2))\subset Vect$.  By Proposition \ref{prop:basechange} the map $F_W$ therefore factors to provide a braided tensor functor
\[
\Theta:(\rep\SL(2)_q)_{\PSL(2)}\to \rep(\mcl{W}_p)
\]
from the de-equivariantization.
\par

Since $\Theta\circ dE\cong F_W$ we have $\Theta(\mbb{V})\cong X_2^+$, and since $\rep(\mcl{W}_p)$ is generated by $X^+_2$ we conclude that $\Theta$ is surjective.  Finally, since
\[
\FPdim\left((\rep\SL(2)_q)_{\PSL(2)}\right)=\FPdim(\rep(u_q(\mfk{sl}_2)))=\FPdim( \rep(\mcl{W}_p)),
\]
by Lemma \ref{lem:fpdim}, surjectivity of $\Theta$ implies that $\Theta$ is an equivalence \cite[Proposition 6.3.4]{egno15}.
\par

Finally, for the ribbon structure, let $\theta$ denote the twist of Section \ref{sect:SL2q}.  The category $(\rep\SL(2)_q)_{\PSL(2)}$, with its reversed braiding, admits precisely two twists, $\theta^{-1}$ and $K^p\theta^{-1}$.  This is because $u_q(\mfk{sl}_2)$ has precisely two central grouplike elements, $1$ and $K^p$.  These twists are such that $\theta^{-1}_{\mbb{V}}=-q^{3/2}$ and $K^p\theta^{-1}_{\mbb{V}}=q^{3/2}$.  Since the value of the twist for $\rep(\mcl{W}_p)$ at $X^+_2$ is $-q^{3/2}$, we see that $\Theta$ is an equivalence of modular tensor categories, after we provide $(\rep\SL(2)_q)_{\PSL(2)}$ with the inverse ribbon structure $\theta^{-1}$.
\end{proof}

\section{Results for logarithmic minimal models}
\label{sect:LMM}

Consider $\mcl{V}ir_c$, the Virasoro vertex operator algebra at our fixed central charge $c=1-6(p-1)^2/p$. This algebra is not $C_2$-cofinite (hence not strongly-finite).  In \cite{mcraeyang} McRae and Yang verify the existence of a vertex tensor structure on a distinguished ``affine" category $\rep(\mcl{V}ir_c)_{\rm aff}$ of Virasoro modules, which we recall below.  Some important features of the category $\rep(\mcl{V}ir_c)_{\rm aff}$ are that it (a) admits a tensor generator, and (b) contains all of the integral simples $\mcl{L}_{r,s}$.
\par

We prove the following theorem, which was first conjectured in \cite{bfgt09} (cf.\ \cite[\S 11.2]{negron}).

\begin{theorem}\label{thm:LMM}
There is a ribbon tensor equivalence $\K:\rep\SL(2)_q^{\rm rev}\overset{\sim}\to \rep(\mcl{V}ir_c)_{\aff}$ which fits into a ($2$-)diagram
\begin{equation}\label{eq:1192}
\xymatrix{
\rep\SL(2)^{\rm rev}_q\ar[d]_{\res^\omega}\ar[rr]^{\K}& & \rep(\mcl{V}ir_c)_{\aff}\ar[d]^{\mcl{W}_p\ot-}\\
\rep(u_q(\mfk{sl}_2))^{\rm rev}\ar[rr]^\Theta & & \rep(\mcl{W}_p).
}
\end{equation}
\end{theorem}

It was argued by Bushlanov, Feigin, Gainutdinov, and Tipunin \cite{bfgt09} that the fusion rings for $\rep\SL(2)_q$ and $\rep(\mcl{V}ir_c)_{\aff}$ are isomorphic.  The equivalence of Theorem \ref{thm:LMM} categorifies the isomorphism of fusion rings proposed in \cite{bfgt09}.
\par

Throughout the section we ignore (notationally) the inversion of the ribbon structure on $\rep\SL(2)_q$, and take for granted that we are considering the category of $\SL(2)_q$-representations along with its reverse braiding and inverted twist, relative to that of Section \ref{sect:SL2q}.

\subsection{Big categories of VOA modules}
\label{sect:ind_cat}

We begin with an aside on categories of ``big" modules for a given VOA.  Let $\mcl{V}$ be a vertex operator algebra with some category $\rep(\mcl{V})_{\rm dist}$ of preferred, finite length modules.  We'll require that the inclusion $\rep(\mcl{V})_{\rm dist}\subset \rep(\mcl{V})$ into the ambient category of finite length $\mcl{V}$-modules is fully faithful, and that the class of preferred modules is closed under taking subquotients.
\par

Given such $\rep(\mcl{V})_{\rm dist}$, we find ourselves in various situations which require a larger category $\Rep(\mcl{V})_{\rm dist}$ which contains $\rep(\mcl{V})_{\rm dist}$ and in which we can, say, take infinite direct sums of modules.  Formally speaking, we construct such a big category $\Rep(\mcl{V})_{\rm dist}$ by taking the Ind-category of $\rep(\mcl{V})_{\rm dist}$ \cite[\S 6.1]{kashiwaraschapira}.  In our setting, one can understand this category rather concretely \cite[Corollary 6.3.5]{kashiwaraschapira}: We first consider the category $\Rep(\mcl{V})$ of arbitrary $\mcl{V}$-modules, with no finiteness assumptions, then take
\[
\Rep(\mcl{V})_{\rm dist}:=\left\{\begin{array}{l}
\text{The full subcategory of $\mcl{V}$-modules $M$ in $\Rep(\mcl{V})$}\\
\text{for which $M$ is the union $M=\cup_\alpha M_\alpha$ of its}\\
\text{(finite length) submodules $M_\alpha\subset M$ in $\rep(\mcl{V})_{\rm dist}$.}
\end{array}\right.
\]
\par

The category $\Rep(\mcl{V})_{\rm dist}$ inherits an abelian structure from the inclusion $\Rep(\mcl{V})_{\rm dist}\subset \Rep(\mcl{V})$,  and it is closed under taking arbitrary filtered colimits.  We refer to the category $\Rep(\mcl{V})_{\rm dist}$ simply as the Ind-category of $\rep(\mcl{V})_{\rm dist}$, or informally as the ``big" category of preferred $\mcl{V}$-modules.
\par

We note that if $\rep(\mcl{V})_{\rm dist}$ admits a (braided) monoidal structure, for which the product is right exact and commutes with sums, then $\Rep(\mcl{V})_{\rm dist}$ inherits a unique (braided) monoidal structure which extends that of $\rep(\mcl{V})_{\rm dist}$ and commutes with colimits.  In particular, when $\rep(\mcl{V})_{\rm dist}$ admits a vertex tensor structure then the big category $\Rep(\mcl{V})_{\rm dist}$ is canonically braided monoidal.  One can see \cite{creutzigmcraeyangII} for an explicit, vertex algebraic, analysis of this extended monoidal structure.

\subsection{Categories of Virasoro modules}
\label{sect:Vaff}

We recall some results from \cite{creutzigjianghunzikerridoutyang,mcraeyang}.  We adopt a calligraphic notation for Virasoro modules $\mcl{L}$, to distinguish such objects from quantum group representations, which we generally denote via a Roman text $L$.
\par

We recall that there is a unique simple $\mcl{V}ir_c$-module of conformal weight $h$ for any $h\in\mbb{C}$.
Consider first the category $\rep(\mcl{V}ir_c)_{\rm int}$ of finite length $\mcl{V}ir_c$-modules with composition factors among the ``integral weighted" simples $\mcl{L}_{r,s}$.  The $\mcl{L}_{r,s}$ are specifically the simple $\mcl{V}ir_c$-modules of conformal weight $h_{r,s}=\frac{1}{4p}((rp-s)^2-(p-1)^2)$, where $r$ is a positive integer and $1\leq s\leq p$.  We consider all morphisms, so that $\rep(\mcl{V}ir_c)_{\rm int}$ is a full subcategory in the ambient category of all $\mcl{V}ir_c$-modules.
\par

It is shown in work of Creutzig, Jiang, Hunziker, Ridout, and Yang \cite{creutzigjianghunzikerridoutyang}, and also McRae and Yang \cite{mcraeyang}, that the category $\rep(\mcl{V}ir_c)_{\rm int}$ satisfies the necessary conditions to admit a vertex tensor structure, and that the corresponding braided monoidal category is rigid.  This category furthermore admits a ribbon structure provided by the (standard) exponential $\theta=e^{2\pi i L(0)}$.  The authors subsequently define the subcategory
\[
\begin{array}{l}
\rep(\mcl{V}ir_c)_{\rm aff}\vspace{1mm}\\
:=\{\text{the full subcategory in $\rep(\mcl{V}ir_c)_{\rm int}$ (tensor) generated by the simple }\mcl{L}_{1,2}\}.
\end{array}
\]
\par

The subcategory $\rep(\mcl{V}ir_c)_{\rm aff}$ in fact contains \emph{all} of the simples in $\rep(\mcl{V}ir_c)_{\rm int}$, and so is the subcategory in $\rep(\mcl{V}ir_c)_{\rm int}$ generated by all of the simple objects $\mcl{L}_{r,s}$.  One can alternatively define $\rep(\mcl{V}ir_c)_{\rm aff}$ as the M\"uger centralizer of the triplet algebra $\mcl{W}_p=\oplus_{n\geq 0}(2n-1)\mcl{L}_{2n+1,1}$ in $\rep(Vir)_{\rm int}$ \cite[Theorems 4.4, 5.2, \& 5.4]{mcraeyang}.  The category $\rep(\mcl{V}ir_c)_{\rm aff}$ is \emph{affine} in the sense that it admits a tensor generator (cf.\ \cite[\S II.5 Corollaire 5.2]{demazuregabriel70}).

\begin{remark}
The categories $\rep(\mcl{V}ir_c)_{\rm int}$ and $\rep(\mcl{V}ir_c)_{\aff}$ are denoted by $\mcl{O}_c$ and $\mcl{O}^0_c$ in \cite{mcraeyang}, respectively.
\end{remark}

By a general theory of VOA extensions \cite{huangkirillovlepowski15,creutzigkanademcrae}, which we've recalled in Appendix \ref{sect:VW}, we have an induction functor
\[
F:=I_{\mcl{V}ir_c}^{\mcl{W}_p}:\rep(\mcl{V}ir_c)_{\rm aff}\to \Rep(\mcl{W}_p)
\]
from the affine category of Virasoro modules to the Ind-category $\Rep(\mcl{W}_p)$ of $\mcl{W}_p$-modules.  McRae and Yang observe the following.

\begin{proposition}[{\cite[Proposition 7.4]{mcraeyang}}]
For the simple objects $\mcl{L}_{r,s}$ in $\rep(\mcl{V}ir_c)_{\aff}$ one has $F(\mcl{L}_{r,s})=rX^{\varepsilon(r)}_s$, with $\varepsilon(r)=(-1)^{r+1}$.
\end{proposition}

The above proposition implies, in particular, that the induction of each simple object in $\rep(\mcl{V}ir_c)_{\rm aff}$ lies in the usual category $\rep(\mcl{W}_p)$ of finite length $\mcl{W}_p$-modules.  We apply Lemma \ref{lem:ind} to find

\begin{lemma}\label{lem:Iprime}
Induction provides  a surjective ribbon tensor functor $F:\rep(\mcl{V}ir_c)_{\rm aff}\to \rep(\mcl{W}_p)$.  In particular, $F$ is faithful and exact.
\end{lemma}

\begin{proof}
Any such induction functor is braided monoidal \cite[Theorem 1.6]{kirillovostrik02}, and exactness of $F$ follows from exactness of the tensor product on $\rep(\mcl{V}ir_c)_{\aff}$ \cite[Theorem 3.65]{creutzigkanademcrae}, or more directly from exactness of the tensor product and Lemma \ref{lem:ind}.  Faithfulness holds because any tensor functor between tensor categories is faithful \cite[Proposition 1.19]{delignemilne82}.  Since we have $F(\mcl{L}_{1,2})=X^+_2$, and since $X^+_2$ generates the tensor category $\rep(\mcl{W}_p)$ (Corollary \ref{cor:wp_gen}), we also see that $F$ is surjective.  Compatibility with the ribbon structure follows from the fact that $\mcl{W}_p$ is central in $\Rep(\mcl{V}ir_c)_{\rm aff}$ and is a fixed point for the twist, so that
\[
e^{2\pi i L(0)}|_{F(\mcl{L})}=\theta_{\mcl{W}_p,\mcl{L}}=(\theta_{\mcl{W}_p}\ot\theta_{\mcl{L}})c^2_{\mcl{W}_p,\mcl{L}}=\theta_{\mcl{L}}=e^{2\pi i L(0)}|_{\mcl{L}}
\]
at arbitrary $\mcl{L}$ in $\rep(\mcl{V}ir_c)_{\rm aff}$.
\end{proof}

\subsection{At the level of Grothendieck rings}

\begin{lemma}
There is a ribbon tensor functor $\K:\rep\SL(2)_q\to \rep(\mcl{V}ir_c)_{\aff}$ which fits into the diagram \eqref{eq:1192}.  In terms of Ostrik's theorem (Theorems \ref{thm:ostrik} and \ref{thm:braidedostrik}), $\K$ is specified by the self-dual simple object $\mcl{L}_{1,2}$ in $\rep(\mcl{V}ir_c)_{\aff}$.
\end{lemma}

\begin{proof}
We have the self-dual object $\mcl{L}_{1,2}$ in $\rep(\mcl{V}ir_c)_{\aff}$ \cite[Theorem 4.1]{mcraeyang} which is of intrinsic dimension
\[
d=d(\mcl{L}_{1,2})=d\left(F(\mcl{L}_{1,2})\right)=d(X^+_2)=-(q+q^{-1}).
\]
So we obtain, via the universal property of the Temperley-Lieb category, a linear monoidal functor $F_{\mcl{L}_{1,2}}:\mcl{TL}(d)\to \rep(\mcl{V}ir_c)_{\aff}$.  The composition
\[
\mcl{TL}(d)\to \rep(\mcl{V}ir_c)_{\aff}\overset{F}\to \rep(\mcl{W}_p)
\]
sends $\mcl{L}_{1,2}$ to $X^+_2$, so that the composite is the universal map specified by the self-dual object $X^+_2$ in $\rep(\mcl{W}_p)$.  We have already seen that this map to $\rep(\mcl{W}_p)$ is faithful, via Proposition \ref{prop:dimX2}, Theorem \ref{thm:ostrik}, and Proposition \ref{prop:faithful}, so that the original map $F_{\mcl{L}_{1,2}}$ to $\rep(\mcl{V}ir_c)_{\aff}$ must be faithful as well.
\par

By Ostrik's theorem we now have a unique extension $\K:\rep\SL(2)_q\to \rep(\mcl{V}ir_c)_{\aff}$ of $F_{\mcl{L}_{1,2}}$ to the category of all quantum group representations.  Furthermore, by Theorem \ref{thm:braidedostrik}, and the fact that the induction functor $F:\rep(\mcl{V}ir_c)_{\aff}\to \rep(\mcl{W}_p)$ is a faithful ribbon tensor functor, we find that $\K$ is a map of ribbon tensor categories.  The diagram \eqref{eq:1192} commutes simply because the functor $\Theta:\rep(u_q(\mfk{sl}_2))\to \rep(\mcl{W}_p)$ is constructed by de-equivariantizing the universal map $\rep\SL(2)_q\to \rep(\mcl{W}_p)$ specified by the self-dual object $X^+_2$, which is (up to natural isomorphism) just the composite $F\circ \K$.
\end{proof}

\begin{center}
\emph{From this point on $\K:\rep\SL(2)_q\to \rep(\mcl{V}ir_c)_{\aff}$ denotes the braided tensor functor specified by the self-dual object $\mcl{L}_{1,2}$ in $\rep(\mcl{V}ir_c)_{\aff}$}.\vspace{2mm}
\end{center}

We recall that the collection of all simples in $\rep\SL(2)_q$ is precisely the collection of products $\{L(\mu)\ot V_s:\mu\in \mbb{Z}_{\geq 0}\frac{p\alpha}{2},1\leq s\leq p\}$.  The following describes the behavior of the functor $\operatorname{K}$ on simple objects.

\begin{lemma}\label{lem:grring}
Take $\mu_r=rp\alpha/2$ for any nonnegative integer $r$.  For the simple objects in $\rep\SL(2)_q$ we have $\K\left(L(\mu_r)\ot V_s\right)\cong \mcl{L}_{r+1,s}$.
\end{lemma}

\begin{proof}
By the descriptions of the fusion rules for $\rep\SL(2)_q$ and $\rep(\mcl{V}ir_c)_{\aff}$, given for example in \cite[\S 2.5]{ostrik08} and \cite[Theorem 4.11]{mcraeyang}, we observe that the lengths of the powers $\mbb{V}^{\ot n}$ and $\mcl{L}_{1,2}^{\ot n}$ are the same, and that each such power has precisely one simple composition factor which does not appear in a lower tensor power.  Specifically, the simples $L(\mu_r)\ot V_s$ and $\mcl{L}_{r+1,s}$ appear first in the powers $\mbb{V}^{\ot (pr+s)}$ and $\mcl{L}_{1,2}^{\ot(pr+s)}$.  So we observe by induction, and exactness of $\K$, that the simples in $\rep\SL(2)_q$ map to simples in $\rep(\mcl{V}ir_c)_{\aff}$, in the prescribed manner.
\end{proof}

The following corollary is immediate.  

\begin{corollary}\label{cor:grring}
The induced map on Grothendieck rings $[\operatorname{K}]:K(\rep\SL(2)_q)\to K(\rep(\mcl{V}ir_c)_{\aff})$ is an isomorphism of $\mbb{Z}_+$-rings.
\end{corollary}

\subsection{Projective representations}

\begin{proposition}\label{prop:1245}
The functor $\K$ sends indecomposable projectives to indecomposable projectives.
\end{proposition}

\begin{proof}
Since the induction functor $F:\rep(\mcl{V}ir_c)_{\aff}\to \rep(\mcl{W}_p)$ has an exact right adjoint, given by restriction along the algebra inclusion $\mcl{V}ir_c\to \mcl{W}_p$, it follows that induction preserves projective objects.  The object $\1=\mcl{L}_{1,1}$ in $\rep(\mcl{V}ir_c)_{\aff}$ has a projective cover $\mcl{Q}_1$ with simple socle and cosocle, both of which are just a copy of $\mcl{L}_{1,1}$.  We have the composition series
\[
[\mcl{Q}_{1}]=2[\mcl{L}_{1,1}]+[\mcl{L}_{2,p-1}]
\]
\cite[Theorem 5.7]{mcraeyang}.
\par

Since $F(\mcl{L}_{2,p-1})\cong (X^-_{p-1})^{\oplus 2}$ \cite[Proposition 7.4]{mcraeyang}, we see that the induction $F(\mcl{Q}_{1})$ is a length $4$ projective in $\rep(\mcl{W}_p)$ which comes equipped with a surjective map to the unit object.  Since the projective cover $\msc{P}^+_1$ of the unit in $\rep(\mcl{W}_p)$ is also of length $4$, we have $\msc{P}^+_1\cong F(\mcl{Q}_{1})$.
\par

Take $P_1$ the projective cover of the unit $\1$ in $\rep\SL(2)_q$.  This is just a lift of the corresponding projective cover in $\rep(u_q(\mfk{sl}_2))$, and so the image under composition
\[
F\circ\K(P_1)\cong \Theta(\operatorname{res}P_1)\cong \msc{P}^+_1
\]
is the projective cover of $X^+_1$.  Now, in $\rep(\mcl{V}ir_c)_{\aff}$ projectivity of $\mcl{Q}_1$ implies the existence of a map $\mcl{Q}_1\to \K(P_1)$ which lifts the surjection $\K(P_1)\to \K(\1)=\mcl{L}_{1,1}$.  After applying induction to $\rep(\mcl{W}_p)$, any such map is an isomorphism.  Since induction is exact and faithful, it follows that the lift $\mcl{Q}_1\to \K(P_1)$ is an isomorphism in $\rep(\mcl{V}ir_c)_{\aff}$.  Hence, $\K$ sends the projective cover $P_1$ of $\1$ to the projective cover $\mcl{Q}_1$ of the unit in $\rep(\mcl{V}ir_c)_{\aff}$.
\par

One similarly argues that the image $\K(P_s)$ of the projective cover of each simple $V_s$ in $\rep\SL(2)_q$ is the projective cover of the corresponding simple $\mcl{L}_{1,s}$ in $\rep(\mcl{V}ir_c)_{\aff}$, for $1\leq s<p$.  For the projective simples $V_p$ and $\mcl{L}_{1,p}$, we have already established that $\K(V_p)=\mcl{L}_{1,p}$ at Lemma \ref{lem:grring}.
\par

Now, for the remaining indecomposable projectives we have $\mcl{Q}_{r,s}=\mcl{L}_{r,1}\ot\mcl{Q}_{1,s}$ in $\rep(\mcl{V}ir_c)_{\aff}$ \cite[Theorem 5.9]{mcraeyang}, so that $\K(L(rp\alpha/2)\ot P_s)\cong \mcl{Q}_{r,s}$ by the above calculations and Lemma \ref{lem:grring}.  We note finally that each product $L(rp\alpha/2)\ot P_s$ in $\rep\SL(2)_q$ is the projective cover of the simple $L(rp\alpha/2)\ot V_s$ \cite[Lemma 4.1]{bfgt09} to complete the proof.
\end{proof}

As we saw in the proof, Proposition \ref{prop:1245} implies that $\K$ sends the projective cover of a given simple in $\rep\SL(2)_q$ to the projective cover of the corresponding simple in $\rep(\mcl{V}ir_c)_{\aff}$.

\subsection{The proof of Theorem \ref{thm:LMM}}

\begin{proof}[Proof of Theorem \ref{thm:LMM}]
By Proposition \ref{prop:1245} all projectives in $\rep(\mcl{V}ir_c)_{\aff}$ are in the image of $\K$.  By considering projective resolutions, we see that $\K$ is fully faithful, and also essentially surjective, if and only if its restriction to the additive subcategory of projectives in $\rep\SL(2)_q$ is fully faithful.  Rather, $\K$ is an \emph{equivalence} if and only if its restriction to the subcategory of projectives is fully faithful.
\par

As any tensor functor is faithful \cite[Proposition 1.19]{delignemilne82}, we already know that $\K$ is faithful.  So we need only establish an equality of dimensions
\begin{equation}\label{eq:1279}
\dim_\mbb{C}\Hom_{\SL(2)_q}(P,P')\overset{?}=\dim_\mbb{C}\Hom_{\mcl{V}ir_c}(\mcl{Q},\mcl{Q}'),
\end{equation}
for indecomposable projectives $P$ and $P'$ in $\rep\SL(2)_q$ with images $\mcl{Q}$ and $\mcl{Q}'$ in $\rep(\mcl{V}ir_c)_{\aff}$.  But now, for the corresponding simple $L$ so that $P$ is the projective cover of $L$, the dimension of the above Hom space is just the multiplicity of the simple $L$ in the composition series for $P'$,
\[
\dim_\mbb{C}\Hom_{\SL(2)_q}(P,P')=[L,P'].
\]
Similarly $\Hom_{\mcl{V}ir_c}(\mcl{Q},\mcl{Q}')=[\mcl{L},\mcl{Q}']$.  Since $\K$ induces an isomorphism at the level of Grothendieck rings, by Corollary \ref{cor:grring}, we observe that $[L,P]=[\mcl{L},\mcl{Q}]$, as desired.
\end{proof}

\section{Categories of modules for the singlet algebra}
\label{sect:repMp}

We describe the (affine) representation category for the singlet vertex operator algebra $\mcl{M}_p$. As with the Virasoro VOA, the singlet is not $C_2$-cofinite (hence not strongly-finite). The singlet algebra sits between the triplet algebra and the Virasoro $\mcl{V}ir_c\subset \mcl{M}_p\subset \mcl{W}_p$; in particular, $\mcl{M}_p$ is the invariant subalgebra (orbifold) in $\mcl{W}_p$ of a maximal torus $\mbb{C}^\times$ in  Aut$(\mcl{W}_p)\cong \PSL(2,\mbb{C})$ (a subgroup isomorphic to $S^1$ also works, but $\mbb{C}^\times$ is more natural from our perspective). Going in the other direction, $\mcl{W}_p$ is obtained from $\mcl{M}_p$ by extending by an infinite order simple current (invertible simple) which generates Rep$\,\mbb{C}^\times\cong\mbb{Z}$.
\par

Our presentation is based on work of Creutzig, McRae, and Yang \cite{creutzigmcraeyang}, but we provide further elaborations on the behaviors of induction, both from the Virasoro to the singlet, and from the singlet to the triplet.  We also provide a Tannakian (group theoretic) description of the M\"uger center of the affine representation category $\rep(\mcl{M}_p)_{\aff}$.  We discuss subsequent relationships between singlet modules and representations of quantum $\SL(2)$ in Section \ref{sect:singlet}.

\subsection{Categories of singlet modules}
\label{sect:Mp_rep}

The singlet algebra $\mcl{M}_p$ has simple modules $M_\lambda$ (sometimes called typicals) labeled by arbitrary functions $\lambda\in \mfk{h}^\ast=\mbb{C}\alpha$ \cite[\S 5]{adamovic03}.  Each $M_\lambda$ appears, in particular, as the lowest weight $\mcl{M}_p$-submodule in the associated Fock module $\mcl{F}_\lambda$.  We have furthermore the ``integral" (atypical) simples $M_{r,s}$ which are associated to the Fock modules of conformal weight $\lambda_{r,s}=\frac{1}{2}((1-r)-\frac{1}{p}(1-s))\sqrt{p}\alpha\in \mfk{h}$ \cite[\S 2]{creutzigridoutwood14}, where $r$ is an arbitrary integer and $1\leq s\leq p$.

We consider, first, the category $\rep(\mcl{M}_p)_{\mathrm{int}}$ of finite length grading-restricted $\mcl{M}_p$-modules with composition factors among the integral simples $M_{r,s}$.  It is shown in \cite{creutzigmcraeyang} that the category $\rep(\mcl{M}_p)_{\mathrm{int}}$ admits a rigid vertex tensor structure, and we have the tensor subcategory $\rep(\mcl{M}_p)_{\aff}$ in $\rep(\mcl{M}_p)_{\rm int}$ generated by the simples $\{M_{r,s}:r\in\mbb{Z},1\leq s\leq p\}$.

\begin{remark}
The categories $\rep(\mcl{M}_p)_{\mathrm{int}}$ and $\rep(\mcl{M}_p)_{\aff}$ above are the categories $\mcl{C}_{\mcl{M}(p)}$ and $\mcl{C}^0_{\mcl{M}(p)}$ of \cite{creutzigmcraeyang}, respectively.
\end{remark}

According to the fusion rules of \cite[Theorem 3.2.6]{creutzigmcraeyang}, the category $\rep(\mcl{M}_p)_{\aff}$ is alternatively generated by the distinguished simple $M_{1,2}$.  For yet another interpretation, the category $\rep(\mcl{M}_p)_{\aff}$ is the centralizer of the algebra object $\mcl{W}_p=\oplus_{n\in \mbb{Z}}M_{2n+1,1}$ in $\rep(\mcl{M}_p)_{\mathrm{int}}$ \cite[\S 1.1 vs.\ \S 5.1]{creutzigmcraeyang}.
\par

Each simple $M_{r,s}$ in $\rep(\mcl{M}_p)_{\aff}$ admits a decomposition as a product $M_{r,s}\cong M_{r,1}\ot M_{1,s}$, and the simple modules $M_{r,1}$ are all invertible with tensor inverse $M^\ast_{r,1}=M_{-r+2,1}$.  The simples $\{M_{r,1}:r\in \mbb{Z}\}$ provide a complete list of invertible objects in $\rep(\mcl{M}_p)_{\aff}$, and they satisfy the fusion rule $M_{r,1}\ot M_{r',1}\cong M_{r+r'-1,1}$ \cite[Theorem 5.2.1]{creutzigmcraeyang}.
\par

We consider the tensor subcategory $\langle M_{3,1}\rangle$ in $\rep(\mcl{M}_p)_{\aff}$ generated by the object $M_{3,1}$, in the sense of Section \ref{sect:tgen}.  By the above information, the simple objects in $\langle M_{3,1}\rangle$ are precisely those of the form $M_{r,1}$, with $r$ odd.

\begin{lemma}\label{lem:ext_Mr1}
There are no extensions $\Ext^1_{\mcl{M}_p}(M_{r,1},M_{r',1})=0$ when $r$ and $r'$ are odd.  Furthermore, when $p>2$, there are no extensions between $M_{r,1}$ and $M_{r',1}$ at arbitrary $r,r'\in\mbb{Z}$.
\end{lemma}

\begin{proof}
Suppose $p>2$.  It suffices to show that the extensions $\Ext^1_{\mcl{M}_p}(\1,M_{r,1})$ vanish, via duality.  We have the projective resolution of the unit $P^\bullet=\cdots\to P_{0,p-1}\oplus P_{2,p-1}\to P_1\to 0$, where the $P_{r,s}$ are the projective covers of the simples $M_{r,s}$, and $P_1$ is the projective cover of the unit \cite[Theorem 5.1.3]{creutzigmcraeyang}.  In particular, there are no nonzero maps from the degree $-1$ portion of the resolution $P^{-1}$ to any $M_{r,1}$.  Hence the subquotient $\Ext^1_{\mcl{M}_p}(\1,M_{r,1})$ of $\Hom_{\mcl{M}_p}(P^{-1},M_{r,1})$ vanishes.
\par

At $p=2$ the above projective resolution of the unit is still valid, but now appears as $\cdots\to P_{0,1}\oplus P_{2,1}\to P_1\to 0$.  So we see that there are no such extensions between $\1$ and $M_{r,1}$ when $r$ is odd.
\end{proof}

We note that at $p=2$ there \emph{are} in fact non-vanishing classes in $\Ext^1_{\mcl{M}_p}(\1,M_{0,1})$ and $\Ext^1_{\mcl{M}_p}(\1,M_{2,1})$, provided by the length $2$ quotients of the projective cover $P_1$ of the unit.  In any case, the above lemma implies the following.  

\begin{corollary}\label{cor:M31}
The tensor subcategory $\langle M_{3,1}\rangle$ generated by $M_{3,1}$ in $\rep(\mcl{M}_p)_{\aff}$ is semisimple, and has simple objects $\{M_{r,1}:r\ \mrm{odd}\}$.
\end{corollary}

\subsection{Induction from the Virasoro}

As a Virasoro module, the algebra $\mcl{M}_p$ decomposes as the sum  of simples $\mcl{M}_p=\oplus_{n\geq 0} \mcl{L}_{2n+1,1}$.  Hence $\mcl{M}_p$ lies in the affine representation category $\rep(\mcl{V}ir_c)_{\aff}$, or rather its Ind-category $\Rep(\mcl{V}ir_c)_{\aff}$ (see Section \ref{sect:ind_cat}).  We therefore have the induction functor
\begin{equation}\label{eq:1415}
I=I_{\mcl{V}ir_c}^{\mcl{M}_p}:\rep(\mcl{V}ir_c)_{\aff}\to \Rep(\mcl{M}_p),\ \mcl{L}\mapsto \mcl{M}_p\ot \mcl{L},
\end{equation}
which has image in the braided monoidal category of $\mcl{M}_p$-modules whose restriction to $\mcl{V}ir_c$ lies in $\Rep(\mcl{V}ir_c)_{\aff}$ \cite[Theorem 1.2, Theorem 2.67]{creutzigkanademcrae}.  Since $\rep(\mcl{V}ir_c)_{\aff}$ is rigid, the functor $I$ is exact (see Appendix \ref{sect:VW}).

\begin{lemma}\label{lem:1334}
There are isomorphisms $I(\mcl{L}_{3,1})\cong M_{3,1}\oplus M_{1,1}\oplus M_{-1,1}$ and $I(\mcl{L}_{1,2})\cong M_{1,2}$.
\end{lemma}

\begin{proof}
We have the adjunction
\[
\Hom_{\mcl{M}_p}(I(\mcl{L}_{3,1}),M_{r,1})\cong \Hom_{\mcl{V}ir_c}(\mcl{L}_{3,1},M_{r,1})
\]
and the Virasoro decompositions $M_{1,1}=\oplus_{n\geq 0}\mcl{L}_{2n+1,1}$,  $M_{3,1}=M_{-1,1}=\oplus_{n\geq 0}\mcl{L}_{2n+3,1}$.  In particular, there is a unique-up-to-scaling nonzero map $\mcl{L}_{3+2n}\to M_{r,1}$ over the Virasoro, for $r=3$, $1$, $-1$.  So, by the above adjunction, we have a nonzero map $I(\mcl{L}_{3,1})\to M_{r,1}$ for each such $r$, and this nonzero map must be a surjection since each $M_{r,1}$ is simple.  We therefore have the product map $I(\mcl{L}_{3,1})\to M_{3,1}\oplus M_{1,1}\oplus M_{-1,1}$, which must also be surjective since each simple in the codomain is distinct.
\par

Note that the above surjection of $\mcl{M}_p$-modules is also a surjection of modules over the Virasoro VOA.  Since the decompositions of $I(\mcl{L}_{3,1})$ and $M_{3,1}\oplus M_{1,1}\oplus M_{-1,1}$ into simple Virasoro modules agree
\[
I(\mcl{L}_{3,1})=\mcl{M}_p\ot \mcl{L}_{3,1}\cong (\oplus_{n\geq 0}\mcl{L}_{2n+1,1})\bigoplus(\oplus_{n\geq 0}2\cdot \mcl{L}_{2n+3,1})=M_{3,1}\oplus M_{1,1}\oplus M_{-1,1}
\]
this surjection must be an isomorphism.
\par

Similarly, we consider the decomposition $M_{1,2}=\oplus_{n\geq 0}\mcl{L}_{2n+1,2}$ to see that there is a surjection $I(\mcl{L}_{1,2})\to M_{1,2}$.  This surjection must, again, be an isomorphism since these two objects have the same decompositions into simples over the Virasoro
\[
I(\mcl{L}_{1,2})=\mcl{M}_p\ot \mcl{L}_{1,2}\cong \oplus_{n\geq 0}\mcl{L}_{2n+1,2}=M_{1,2}.
\]
\end{proof}

Of course, the significance of the object $\mcl{L}_{1,2}$ is that it is the distinguished tensor generator for the affine representation category $\rep(\mcl{V}ir_c)_{\aff}$.  Similarly, the object $\mcl{L}_{3,1}$ generates the M\"uger center in $\rep(\mcl{V}ir_c)_{\aff}$.  This just follows from the fact that it is the image $\mcl{L}_{3,1}=\K\left(L(p\alpha)\right)$ of the central generator $L(p\alpha)$ in $\rep\PSL(2)\subset \rep\SL(2)_q$, under the equivalence of Theorem \ref{thm:LMM}.  One can see in particular Lemma \ref{lem:grring}.  The previous lemma and the aforementioned generating property for $\mcl{L}_{1,2}$ provide the following.

\begin{corollary}
The induction functor $I$ has image in the affine subcategory $\rep(\mcl{M}_p)_{\aff}$, and restricts to a surjective, ribbon, tensor functor
\begin{equation}\label{eq:1517}
I:\rep(\mcl{V}ir_c)_{\aff}\to \rep(\mcl{M}_p)_{\aff}.
\end{equation}
\end{corollary}

\begin{proof}
Let $\Rep(\mcl{M}_p)$ denote the braided monoidal category of $\mcl{M}_p$-modules which restrict to objects in $\Rep(\mcl{V}ir_c)_{\aff}$, along the inclusion $\mcl{V}ir_c\to \mcl{M}_p$, by an abuse of notation.  We note that all of the simples $M_{r,s}$ in $\rep(\mcl{M}_p)_{\aff}$ lie in this category $\Rep(\mcl{M}_p)$ \cite[\S 2.3]{creutzigetal}, and that the induction functor $I$ is an \emph{exact} braided monoidal functor, via rigidity of $\rep(\mcl{V}ir_c)_{\aff}$.
\par

We have the (rigid) tensor subcategory $\rep(\mcl{M}_p)_{\aff}$ in $\Rep(\mcl{M}_p)$ which is closed under taking subquotients, and is tensor generated by the simple $M_{1,2}$.  Since $\rep(\mcl{V}ir_c)_{\aff}$ is tensor generated by the simple $\mcl{L}_{1,2}$, and $I(\mcl{L}_{1,2})=\mcl{M}_{1,2}$ by Lemma \ref{lem:1334}, we see that $I$ restricts to a surjective, braided, tensor functor as in \eqref{eq:1517}.  Compatibility of $I$ with the twist follows by the same arguments given in the proof of Lemma \ref{lem:Iprime}.
\end{proof}

\subsection{The M\"uger center in $\rep(\mcl{M}_p)_{\aff}$}

As a module over the singlet algebra, we have $\mcl{W}_p=\oplus_{n\in \mbb{Z}}M_{2n+1,1}$.  In particular, $\mcl{W}_p$ lies in the Ind-category of affine representations for the singlet $\mcl{M}_p$, and we have the braided monoidal functor $I':\rep(\mcl{M}_p)\to \Rep(\mcl{W}_p)$ provided by induction.

\begin{lemma}\label{lem:1550}
Induction restricts to a surjective, braided tensor functor
\[
I':\rep(\mcl{M}_p)_{\aff}\to \rep(\mcl{W}_p),
\]
and furthermore satisfies $I'(M_{r,s})=X^{\epsilon(r)}_s$ with $\epsilon(r)=(-1)^{r+1}$.
\end{lemma}

\begin{proof}
Take $I=I_{\mcl{V}ir_c}^{\mcl{M}_p}$, $I'=I_{\mcl{M}_p}^{\mcl{W}_p}$, and $F=I_{\mcl{V}ir_c}^{\mcl{W}_p}$.  For induction $F:\rep(\mcl{V}ir_c)_{\aff}\to \rep(\mcl{W}_p)$ from the Virasoro, we have $F=I'\circ I$, by general shenanigans with adjunctions.  Hence
\[
I'(M_{1,2})=I'\circ I(\mcl{L}_{1,2})=F(\mcl{L}_{1,2})=X^+_{2}.
\]
Since $X^+_2$ generates the representation category of the triplet, and $M_{1,2}$ generates the representation category of the singlet, it follows that the braided tensor functor $I'$ has image in the category $\rep(\mcl{W}_p)$ of finite length $\mcl{W}_p$-modules, and is surjective.  The calculation of $I'(M_{r,s})$ follows by a direct analysis of the fusion rules for $\mcl{M}_p$ and $\mcl{W}_p$ \cite{mcraeyang,tsuchiyawood13}.
\end{proof}

\begin{lemma}\label{lem:1345}
The M\"uger center of $\rep(\mcl{M}_p)_{\aff}$ is precisely the tensor subcategory $\langle M_{3,1}\rangle$ generated by the invertible simple $M_{3,1}$.
\end{lemma}

\begin{proof}
Since the induction functor $I:\rep(\mcl{V}ir_c)_{\aff}\to \rep(\mcl{M}_p)_{\aff}$ is surjective and braided, it must send the M\"uger center in $\rep(Vir)_{\aff}$ into the M\"uger center in $\rep(\mcl{M}_p)_{\aff}$.  By Lemma \ref{lem:1334} we have that $M_{3,1}$ is in the surjective image of the M\"uger center of $\rep(\mcl{V}ir_c)_{\aff}$, and hence the M\"uger center contains $\langle M_{3,1}\rangle$.  By Corollary \ref{cor:M31}, the category $\langle M_{3,1}\rangle$ is the subcategory of all objects in $\rep(\mcl{M}_p)_{\aff}$ with composition factors (only) among the $M_{r,1}$, with $r$ odd.  To see that no other simple in $\rep(\mcl{M}_p)_{\aff}$ are central, we simply apply the braided monoidal functor $I':\rep(\mcl{M}_p)_{\aff}\to \rep(\mcl{W}_p)$ and note that $I'(M_{r,s})=X^{\epsilon(r)}_s$ centralizes $I'(M_{1,2})=X^+_2$ if and only if $r$ is odd and $s=1$, by Theorem \ref{thm:modular}.
\end{proof}

We provide a further analysis of the M\"uger center in $\rep(\mcl{M}_p)_{\aff}$, and its behaviors under induction both from the Virasoro and to the triplet algebra.

\begin{lemma}\label{lem:ZMp}
The induction functor $I:\rep(\mcl{V}ir_c)_{\aff}\to \rep(\mcl{M}_p)_{\aff}$ restricts to a surjective, symmetric tensor functor
\[
Z_{\text{\rm M\"ug}}(\rep(\mcl{V}ir_c)_{\aff})\to Z_{\text{\rm M\"ug}}\left(\rep(\mcl{M}_p)_{\aff}\right).
\]
\end{lemma}

\begin{proof}
Surjectivity of the induction functor $I$ implies that $I$ sends the M\"uger center in $\rep(\mcl{V}ir_c)_{\aff}$ to the M\"uger center in $\rep(\mcl{M}_p)_{\aff}$.  So the result follows from the computation $I(\mcl{L}_{3,1})=M_{3,1}\oplus M_{1,1}\oplus M_{-1,1}$ of Lemma \ref{lem:1334}, the fact that $M_{3,1}$ generates the M\"uger center in $\rep(\mcl{M}_p)_{\aff}$, and the fact that $\mcl{L}_{3,1}=\K(L(p\alpha))$ is M\"uger central in $\rep(\mcl{V}ir_c)_{\aff}$.
\end{proof}

Via Theorem \ref{thm:LMM}, and Section \ref{sect:Fr}, we understand that the M\"uger center in $\rep(\mcl{V}ir_c)_{\aff}$ is equivalent to the representation category of $\PSL(2)$.  We have a corresponding group theoretic interpretation of the center in $\rep(\mcl{M}_p)_{\aff}$.

\begin{proposition}\label{prop:Gm}
The induction functor $I':\rep(\mcl{M}_p)_{\aff}\to \rep(\mcl{W}_p)$ restricts to a symmetric fiber functor $Z_{\text{\rm M\"ug}}\left(\rep(\mcl{M}_p)_{\aff}\right)\to Vect$.  Furthermore, there is a symmetric tensor equivalence
\[
\operatorname{Fr}':\rep\mbb{G}_m\overset{\sim}\to Z_{\text{\rm M\"ug}}\left(\rep(\mcl{M}_p)_{\aff}\right).
\]
\end{proposition}

For the unfamiliar reader, $\mbb{G}_m$ denotes the multiplicative group $\mbb{C}^\times$, considered as an affine algebraic group, and the category $\rep\mbb{G}_m$ is identified with the symmetric tensor category of $\mbb{Z}$-graded vector spaces. This is the group $\mbb{Z}$ of simple currents, with which we extend the VOA $\mcl{M}_p$ to obtain $\mcl{W}_p$.

\begin{proof}
Take $\msc{Z}=Z_{\text{\rm M\"ug}}\left(\rep(\mcl{M}_p)_{\aff}\right)$.  Since $I'$ is a braided, surjective, tensor functor, it sends the M\"uger center in $\rep(\mcl{M}_p)_{\aff}$ to the M\"uger center in $\rep(\mcl{W}_p)$.  But the M\"uger center in $\rep(\mcl{W}_p)$ is trivial, by Theorem \ref{thm:modular}, so that $I'$ restricts to a symmetric fiber functor from the M\"uger center in $\rep(\mcl{M}_p)_{\aff}$, as claimed.  It follows now, by Tannakian reconstruction \cite[Theorem 2.11]{delignemilne82}, that there is an algebraic group $\mbb{G}$ which admits a symmetric tensor equivalence $\rep\mbb{G}\overset{\sim}\to \msc{Z}$.
\par

By our understanding of $\msc{Z}$ provided by Corollary \ref{cor:M31} and Lemma \ref{lem:1345}, $\mbb{G}$ has an invertible representation $\mbb{C}_\chi$ which generates $\rep\mbb{G}$ and admits no extensions from its various tensor powers.  This representation therefore specifies a surjective map of algebraic groups
\[
\mbb{G}\to \operatorname{GL}(\mbb{C}_{\chi})=\mbb{G}_m
\]
for which the restriction functor $\rep\mbb{G}_m\to \rep\mbb{G}\cong \msc{Z}$ is fully faithful and essentially surjective.  This functor is therefore an equivalence, so that $\msc{Z}\cong\rep\mbb{G}_m$.
\end{proof}

\section{The singlet algebra and torus extended $u_q(\mfk{sl}_2)$}
\label{sect:singlet}

In this section we establish a quantum group equivalence for the affine representation category of the singlet vertex operator algebra $\mcl{M}_p$.  We compare the singlet to the ``torus extended" small quantum group $\dotu_q(\mfk{sl}_2)$ of \cite[\S 36.2.1]{lusztig93}, at the given parameter $q=\exp(\pi i/p)$.  We prove the following analog of Theorems \ref{thm:triplet} and \ref{thm:LMM} for the singlet (see Section \ref{sect:Mp_rep}).  

\begin{theorem}\label{thm:singlet}
There is an equivalence $\Psi:\rep(\dotu_q(\mfk{sl}_2))^{\rm rev}\overset{\sim}\to \rep(\mcl{M}_p)_{\aff}$ of ribbon tensor categories which fits into a (2-)diagram
\[
\xymatrix{
\rep(\dotu_q(\mfk{sl}_2))^{\rm rev}\ar[d]_{\res^\omega}\ar[rr]^{\Psi}& & \rep(\mcl{M}_p)_{\aff}\ar[d]^{\mcl{W}_p\ot-}\\
\rep(u_q(\mfk{sl}_2))^{\rm rev}\ar[rr]^\Theta & & \rep(\mcl{W}_p).
}
\]
\end{theorem}

The proof of Theorem \ref{thm:singlet} relies in essential ways on notions of equivariantization and de-equivariantization for tensor categories, relative to a given algebraic group action.  We recall the relevant constructions in Appendix \ref{sect:appendix}.

\begin{remark}\label{rem:CM}
As remarked in the introduction, a quantum group equivalence for the singlet was conjectured in \cite{creutzigmilas14,costantinoetal15}.  The works \cite{creutzigmilas14,costantinoetal15} conjecture, specifically, an equivalence between representations of the so-called unrolled quantum group $\mbf{u}^H_q(\mfk{sl}_2)$ and a certain extension $\rep_{\langle s\rangle}(\mcl{M}_p)$ of $\rep(\mcl{M}_p)_{\aff}$ by the category of $\mbb{C}/\mbb{Z}$-graded vector spaces.  The category of $\mbf{u}^H_q(\mfk{sl}_2)$-representations is just the category of $\mfk{h}^\ast$-graded representations of $u_q(\mfk{sl}_2)$, as opposed to those graded by the character lattice $\Lambda\subset \mfk{h}^\ast$.  So the above theorem differs, to some degree, from the precise conjectures of \cite{creutzigmilas14,costantinoetal15}. 
\end{remark}

\subsection{The category of $\dotu_q(\mfk{sl}_2)$-representations}
\label{sect:udot}

We understand the algebra $\dotu_q(\mfk{sl}_2)$ directly through its representations.  A representation of $\dotu_q(\mfk{sl}_2)$ is a finite-dimensional $\Lambda=\frac{1}{2}\mbb{Z}\alpha$-graded vector space $V$ which comes equipped with $p$-nilpotent linear endomorphisms $E,F:V\to V$ which (a) shift the degree by $\alpha$ and $-\alpha$, respectively, and (b) satisfy the quantum group relations of \cite{lusztig90}.  So, the construction of $\rep(\dotu_q(\mfk{sl}_2))$ is completely analogous to the construction of $\rep\SL(2)_q$ given in Section \ref{sect:SL2q}, where we simply forget about the additional operators $E^{(p)}$ and $F^{(p)}$.
\par

The expected coproduct $\Delta(E)=E\ot 1+K\ot E$, and $\Delta(F)=F\ot K^{-1}+1\ot F$, provides the category $\rep(\dotu_q(\mfk{sl}_2))$ with a rigid tensor structure.  Furthermore, the $R$-matrix and twist from \ref{sect:SL2q} define a (unique) ribbon structure on the category $\rep(\dotu_q(\mfk{sl}_2))$ so that the forgetful functor $\rep\SL(2)_q\to \rep(\dotu_q(\mfk{sl}_2))$ is a ribbon tensor functor.  This is the \emph{standard} ribbon structure on the category of $\dotu_q(\mfk{sl}_2)$-representations, and in the statement of Theorem \ref{thm:singlet} we consider $\rep(\dotu_q(\mfk{sl}_2))$ with its reversed braiding and inverted twist, relative to this standard structure.
\par

As in the previous sections, we omit the superscript $(-)^{\rm rev}$ from our analysis, and take for granted that we are considering the category of $\dotu_q(\mfk{sl}_2)$-representations with the reversed braiding and inverted twist.

\subsection{Triplet modules via (de-)equivariantization}
\label{sect:1611}

We let $\Rep(\mcl{M}_p)_{\aff}$ denote the Ind-category of $\rep(\mcl{M}_p)_{\aff}$, which we define as in Section \ref{sect:ind_cat}.

The equivalence $\operatorname{Fr}'$ from $\rep\mbb{G}_m$ to the M\"uger center of $\rep(\mcl{M}_p)_{\aff}$ sends the algebra object $\O(\mbb{G}_m)$ to the algebra object
\[
A=\oplus_{n\in \mbb{Z}}M_{2n+1,1}\cong \mcl{W}_p
\]
in the M\"uger center of $\rep(\mcl{M}_p)_{\aff}$.  (Note that the object $A$ has a unique commutative algebra structure in $\rep(\mcl{M}_p)_{\aff}$, up to isomorphism, under which it is a simple module over itself.)  The following essentially rephrases Lemma \ref{lem:1550}.

\begin{lemma}[{cf.\ \cite[Theorem 3.65]{creutzigkanademcrae}}]
Restriction $\rep(\mcl{W}_p)\to \Rep(\mcl{M}_p)_{\aff}$, and any choice of algebra isomorphism $\mcl{W}_p\cong A$, identifies $\rep(\mcl{W}_p)$ with the (ribbon tensor) category of finitely generated $A$-modules in $\Rep(\mcl{M}_p)_{\aff}$.  Furthermore, any finitely generated $A$-module is finitely presented.
\end{lemma}

To be clear, by a finitely generated $A$-module $X$ we mean one which admits an $A$-module surjection $A\ot V\to X$ from a free module, with $V$ in $\rep(\mcl{M}_p)_{\aff}$.  By a finitely presented module we mean one which admits an exact sequence $A\ot W\to A\ot V\to X\to 0$, where $W$ and $V$ are in $\rep(\mcl{M}_p)_{\aff}$.

\begin{proof}
Lemma \ref{lem:1550} implies that any free module $I'(V)=\mcl{W}_p\ot V$ is, in particular, a finite length $\mcl{W}_p$-module in $\Rep(\mcl{M}_p)_{\aff}$.  Hence any quotient $\mcl{W}_p\ot V\to X$ of a free module is a finite length $\mcl{W}_p$-module.  So we see that any finitely generated $\mcl{W}_p$-module is of finite length, and so lies in $\rep(\mcl{W}_p)$.  Conversely, since the category $\rep(\mcl{M}_p)_{\aff}$ has enough projectives, surjectivity of the induction functor $I'$ (Lemma \ref{lem:1550}) is equivalent to the claim that any finite length $\mcl{W}_p$-module $X$ admits a surjection from a free module $\mcl{W}_p\ot V\to X$, and so is finitely generated.  Finally, since the kernel of such a surjection $\mcl{W}_p\ot V\to X$ is also of finite length, we observe that any finitely generated module is also finitely presented.  This shows that $\rep(\mcl{W}_p)$ is identified with the category of finitely presented $\mcl{W}_p$-modules in $\Rep(\mcl{M}_p)_{\aff}$ under restriction.  Restricting along any algebra isomorphism $A\to \mcl{W}_p$ now provides the claimed result.
\end{proof}

In the language of Section \ref{sect:dE}, and Appendix \ref{sect:A2}, we have an identification of ribbon tensor categories
\[
\rep(\mcl{W}_p)=(\rep(\mcl{M}_p)_{\aff})_{\mbb{G}_m}:=\left\{
\begin{array}{c}
\text{finitely presented $A=\operatorname{Fr}'\O(\mbb{G}_m)$}\\
\text{modules in }\Rep(\mcl{M}_p)_{\aff}
\end{array}\right\}
\]
between $\mcl{W}_p$-modules and the de-equivariantization of $\rep(\mcl{M}_p)$ along the central functor $\rep\mbb{G}_m\to \rep(\mcl{M}_p)$.  By a general result \cite{arkhipovgaitsgory03,negron}, it follows that there is a categorical action of $\mbb{G}_m$ on $\rep(\mcl{W}_p)$ for which we have an equivalence
\begin{equation}\label{eq:equiv}
\rep(\mcl{W}_p)^{\mbb{G}_m}\cong \rep(\mcl{M}_p)_{\aff}
\end{equation}
between the category of $\mbb{G}_m$-equivariant objects in $\rep(\mcl{W}_p)$ and the affine representation category of $\mcl{M}_p$ \cite[Proposition A.2]{negron}.

Let us explain the equivalence \eqref{eq:equiv} in more tangible terms.  As explained above, the restriction functor $\rep(\mcl{W}_p)\to \Rep(\mcl{M}_p)_{\aff}$ identifies finite length modules over the triplet algebra with finitely generated $A$-modules in $\Rep(\mcl{M}_p)_{\aff}$.  The algebra $A$ can be written as $\mcl{W}_p\cong A=\mcl{M}_p[t,t^{-1}]$ where $t$ is the invertible object $M_{3,1}$ and $t^{-1}=M_{3,1}^\ast=M_{-1,1}$.  Furthermore, the translation action of $\mbb{G}_m$ on $A=\Fr'\O(\mbb{G}_m)$ corresponds precisely to the $\mbb{Z}$-grading on $A$--and also $\mcl{W}_p$--specified by taking $\deg(t)=1$, $\deg(t^{-1})=-1$.
\par

With the above framing in mind, a $\mbb{G}_m$-equivariant object in $\rep(\mcl{W}_p)$ is just a $\mcl{W}_p$-module $X$ in $\Rep(\mcl{M}_p)_{\aff}$ with a compatible $\mbb{Z}$-grading so that $t^{\pm 1}\cdot X_k=X_{k\pm 1}$.  We are claiming at \eqref{eq:equiv} that the map
\begin{equation}\label{eq:1648}
\rep(\mcl{W}_p)^{\mbb{G}_m}=\mcl{W}_p\text{-mod}^\mbb{Z}\to \rep(\mcl{M}_p)_{\aff},\ \ X\mapsto X_0
\end{equation}
is an equivalence of categories, and has inverse provided by induction.  One can see directly (or abstractly as above) that this is indeed the case.  The functor \eqref{eq:1648} is furthermore an equivalence of ribbon tensor categories, as its inverse (induction) is compatible with the ribbon tensor structure.

\begin{remark}
For an alternate take on the above information, phrased explicitly in the language of vertex operator algebras and Lie group actions, one can see \cite[Section 7.2]{mcraeyang}.
\end{remark}

\subsection{Principles for Theorem \ref{thm:singlet}}

As recalled in Appendix \ref{sect:appendix}, the equivalence $(\rep\SL(2)_q)_{\PSL(2)}\cong \rep(u_q(\mfk{sl}_2))$ of Theorem \ref{thm:N} induces a $\PSL(2)$-action on the category of $u_q(\mfk{sl}_2)$-representations.  We have the standard torus $\mbb{G}_m\cong T\subset \PSL(2)$ of diagonal matrices, and we can restrict the action of $\PSL(2)$ to an action of the torus on $\rep(u_q(\mfk{sl}_2))$.
\par

Given this torus action on $\rep(u_q(\mfk{sl}_2))$, we can consider the non-full tensor subcategory $\rep(u_q(\mfk{sl}_2))^T\subset \rep(u_q(\mfk{sl}_2))$ of $T$-equivariant representations.  One should view objects in $\rep(u_q(\mfk{sl}_2))^T$ as $u_q(\mfk{sl}_2)$-representations equipped with a compatible rational $T$-action, although this is an oversimplification.  This subcategory is identified with the category of $\dotu_q(\mfk{sl}_2)$-representations via restriction.

\begin{proposition}[{\cite[Proposition 9.1]{negron} \cite{arkhipovgaitsgory03}}]
There is an equivalence $\rep(u_q(\mfk{sl}_2))^T\cong \rep(\dotu_q(\mfk{sl}_2))$ of ribbon tensor categories.
\end{proposition}

We consider the above theorem in parallel with the equivalence \eqref{eq:1648} for the singlet.  We claim at this point that (after some error correction if necessary) the equivalence $\Theta: \rep(u_q(\mfk{sl}_2))\overset{\sim}\to \rep(\mcl{W}_p)$ of Theorem \ref{thm:triplet} commutes with the $\mbb{G}_m$-actions on these categories, where we act on $\rep(u_q(\mfk{sl}_2))$ via the torus $\mbb{G}_m\cong T$ and we act on $\rep(\mcl{W}_p)$ in the manner prescribed in Section \ref{sect:1611}.  One then obtains an induced equivalence on the subcategories of $\mbb{G}_m$-equivariant objects
\[
\xymatrix{
\rep(\dotu_q(\mfk{sl}_2))\cong \rep(u_q(\mfk{sl}_2))^T\ar[rr]^\sim_(.53){\Psi:=\Theta^{\mbb{G}_m}} & & \rep(\mcl{W}_p)^{\mbb{G}_m}\cong \rep(\mcl{M}_p),
}
\]
which provides the claimed result.

\subsection{The proof of Theorem \ref{thm:singlet}}

We freely use the calculus of equivariantization and de-equivariantization in the proof.  One should see Appendix \ref{sect:appendix}, or the original texts \cite{arkhipovgaitsgory03,davydovetingofnikshych18,negron}.

\begin{proof}[Proof of Theorem \ref{thm:singlet}]
We have the ($2$-)diagram of braided tensor functors
\[
\xymatrix{
\rep\SL(2)_q\cong \rep(\mcl{V}ir_c)_{\aff}\ar[rr]^(.6)I & & \rep(\mcl{M}_p)_{\aff}\ar[rr]^{I'} & & \rep(\mcl{W}_p)\\
\rep\PSL(2)\ar[u]^{\rm Fr}\ar[rr]_{\rm surject}^{\rm Lem\ \ref{lem:ZMp}} & & \rep\mbb{G}_m\ar[u]^{\rm Fr'}\ar[rr]^{\rm Prop\ \ref{prop:Gm}}_{\rm fiber} & & Vect,\ar[u]^{\rm unit}
}
\]
where $I$ and $I'$ are the appropriate induction functors.  By Tannakian reconstruction, and surjectivity of the map
\[
I|_{\rep\PSL(2)}:\rep\PSL(2)\to \rep\mbb{G}_m,
\]
we have that $I|_{\rep\PSL(2)}$ is isomorphic to restriction $\res_\alpha:\rep\PSL(2)\to \rep\mbb{G}_m$ along a group embedding $\alpha:\mbb{G}_m\to \PSL(2)$ \cite[Theorem 2.3.2]{rivano06} \cite[Proposition 2.21]{delignemilne82}.  So the above diagram can be replaced with the ($2$-)diagram
\[
\xymatrix{
\rep(\mcl{V}ir_c)_{\aff}\ar[rr]^I & & \rep(\mcl{M}_p)_{\aff}\ar[r]^{I'} & \rep(\mcl{W}_p)\\
\rep\PSL(2)\ar[u]^{\rm Fr}\ar[rr]^{\res_\alpha} & & \rep\mbb{G}_m\ar[u]^{\rm Fr'}\ar[r]^{\rm fiber} & Vect.\ar[u]^{\rm unit}
}
\]
Hence we have that the $\mbb{G}_m$-action on $\rep(\mcl{W}_p)$ induced by the identification $\rep(\mcl{W}_p)=\rep(\mcl{M}_p)_{\mbb{G}_m}$ is isomorphic to the restriction
\[
\mbb{G}_m\overset{\alpha}\to \PSL(2)\to \underline{\Aut}_{\ot}(\rep(\mcl{W}_p))
\]
of the $\PSL(2)$-action on $\rep(\mcl{W}_p)$ induced by the equivalence $\rep(\mcl{W}_p)\cong (\rep(\mcl{V}ir_c)_{\aff})_{\PSL(2)}$ of Theorem \ref{thm:triplet}, and Theorem \ref{thm:LMM}.  Let us argue this point more directly.
\par

The map $\alpha:\mbb{G}_m\to \PSL(2)$ provides, dually, a map $\alpha^\ast:\O(\PSL(2))\to \O(\mbb{G}_m)$ of Hopf algebras (in $Vect$).  From the map $\alpha^\ast$ we also realize $\O(\PSL(2))$ as a $\O(\mbb{G}_m)$-comodule algebra, and $\alpha^\ast$ itself is a map of $\O(\mbb{G}_m)$-comodule algebras.  It follows that the equivalence
\begin{equation}\label{eq:1568}
\O(\mbb{G}_m)\ot_{\O(\PSL(2))}I(-):(\rep(\mcl{V}ir_c)_{\aff})_{\PSL(2)}\overset{\sim}\to (\rep(\mcl{M}_p)_{\aff})_{\mbb{G}_m}
\end{equation}
is $\mbb{G}_m$-equivariant, since the $\mbb{G}_m$-actions on these categories are induced by the $\O(\mbb{G}_m)$-coactions on the associated algebras \cite[Section A.1]{negron}.  (Note that we have abused notation in the expression \eqref{eq:1568}, as $\O(\mbb{G}_m)$ should really be $\Fr'\O(\mbb{G}_m)$, for example.)
\par

Now, the $\PSL(2)$-action on $\rep(\mcl{W}_p)$ induced by the equivalence
\[
(\rep\SL(2)_q)_{\PSL(2)}\cong (\rep(\mcl{V}ir_c)_{\aff})_{\PSL(2)}\cong \rep(\mcl{W}_p)
\]
is such that the equivalence $\Theta:\rep(u_q(\mfk{sl}_2))\to \rep(\mcl{W}_p)$ of Theorem \ref{thm:triplet} is $\PSL(2)$-equivariant.  We therefore have an equivalence
\begin{equation}\label{eq:1832}
\Theta^{\mbb{G}_m}:\rep(u_q(\mfk{sl}_2))^{\mbb{G}_m}\overset{\sim}\to \rep(\mcl{W}_p)^{\mbb{G}_m},
\end{equation}
where we act via the map $\alpha:\mbb{G}_m\to \PSL(2)$.  We note that the map $\alpha:\mbb{G}_m\to \PSL(2)$ identifies $\mbb{G}_m$ with a maximal torus in $\PSL(2)$, and all maximal tori are conjugate.  Therefore we can find $x\in \PSL(2)$ so that
\[
\operatorname{Ad}_x:\PSL(2)\to \PSL(2)
\]
sends $\alpha(\mbb{G}_m)$ to the standard torus $T=\{\operatorname{diag}\{a,a^{-1}\}:a\in \mbb{C}^\times\}\subset \PSL(2)$.  It follows that the action of $x$ on $\rep(u_q(\mfk{sl}_2))$ provides an equivalence
\[
x\cdot-:\rep(u_q(\mfk{sl}_2))^T\overset{\sim}\to \rep(u_q(\mfk{sl}_2))^{\mbb{G}_m},
\]
and we find from \eqref{eq:1832} an equivalence $\rep(u_q(\mfk{sl}_2))^T\cong \rep(\mcl{W}_p)^{\mbb{G}_m}$.  We recall that the equivariantization of $\rep(u_q(\mfk{sl}_2))$ by the action of the torus of diagonal matrices is the representation category $\rep(\dotu_q(\mfk{sl}_2))$ \cite[Proposition 9.1]{negron} to observe finally
\[
\rep(\dotu_q(\mfk{sl}_2))\cong \rep(u_q(\mfk{sl}_2))^T\cong \rep(\mcl{W}_p)^{\mbb{G}_m}\cong \rep(\mcl{M}_p)_{\aff}.
\]
Compatibility with the ribbon structure follows from the diagram
\[
\xymatrix{
\rep(\dotu_q(\mfk{sl}_2))\ar[r]^\sim \ar[dr]_{\rm res}& \rep(u_q(\mfk{sl}_2))^T\ar[rr]^{\rm induced}\ar[d]_{\rm forget} & & \rep(\mcl{W}_p)^{\mbb{G}_m}\ar[d]^{\rm forget}& \rep(\mcl{M}_p)\ar[l]_\sim\ar[dl]^{I'}\\
& \rep(u_q(\mfk{sl}_2))\ar[rr]^{\Theta\circ x} & & \rep(\mcl{W}_p)
}
\]
and the fact that all of the functors present, save for possibly the induced map in question, are (faithful) ribbon tensor functors.
\end{proof}

\appendix

\section{Induction for VOA extensions}
\label{sect:VW}

We recall some information regarding extensions of vertex operator algebras and induction.  The original references for the following materials are \cite{kirillovostrik02,huangkirillovlepowski15,creutzigkanademcrae}.

\subsection{Vertex algebra extensions and induction}

Consider $\mcl{V}$ a vertex operator algebra with $\rep(\mcl{V})_{\rm dist}$ a full subcategory of distinguished objects in the category of finite length $\mcl{V}$-modules.  We suppose additionally that the subcategory $\rep(\mcl{V})_{\rm dist}$ is closed under taking subquotients, and that this subcategory admits a vertex tensor structure as described in Section \ref{sect:voacat}.  We let $\Rep(\mcl{V})_{\rm dist}$ denote the associated Ind-category of distinguished modules, which one can describe as in Section \ref{sect:ind_cat}.  In this setting, the category $\Rep(\mcl{V})_{\rm dist}$ inherits a unique braided monoidal structure for which the product $\ot$ commutes with colimits.  Furthermore, this monoidal structure is specified by intertwining operators, as one might expect \cite[\S 6]{creutzigmcraeyangII}.

Consider $\rep(\mcl{V})_{\rm dist}$ as above, and suppose we have a vertex operator algebra extension $\mcl{V}\to \mcl{W}$ with $\mcl{W}$ lying in $\Rep(\mcl{V})_{\rm dist}$, as a $\mcl{V}$-module.  Then the vertex algebra structure on $\mcl{W}$ gives it the structure of a commutative algebra object in $\Rep(\mcl{V})_{\rm dist}$ \cite[Theorem 3.13]{creutzigkanadelinshaw20}.  In such a setting we let $\rep(\mcl{W})_{\rm dist}$ denote the category of finite length $\mcl{W}$-modules which restrict to $\mcl{V}$-modules in $\Rep(\mcl{V})_{\rm dist}$, along the extension $\mcl{V}\to \mcl{W}$.
\par

Consider now a general commutative algebra object $A$ in $\Rep(\mcl{V})_{\rm dist}$.  Recall that an $A$-module $M$ is called local if the action map $act:A\ot M\to M$ is such that $act\circ c^2_{A,M}=act$, where $c_{A,M}$ is the braiding for $\Rep(\mcl{V})_{\rm dist}$.  In this way any local $A$-module admits an unambiguous $A$-bimodule structure, and the category of local $A$-modules inherits a braided monoidal structure under the product $\ot_A$ \cite[Theorem 1.10]{kirillovostrik02}.  We let $\rep^0(A)$ denote the category of finite length, local, $A$-modules in $\Rep(\mcl{V})_{\rm dist}$, and let $\Rep^0(A)$ denote its Ind-category.  So, $\Rep^0(A)$ is the category of $A$-modules in $\Rep(\mcl{V})_{\rm dist}$ which are the unions of their finite length submodules.
\par

We have the following essential results of Creutzig, Kanade, McRae, and Yang, and also Huang, Kirillov, and Lepowsky.

\begin{theorem}[{\cite[Theorem 7.7]{creutzigmcraeyangII}, \cite[Theorem 3.65]{creutzigkanademcrae}, \cite{huangkirillovlepowski15}}]\label{thm:A}
Consider $\rep(\mcl{V})_{\rm dist}$ as above, and an extension of vertex operator algebras $\iota:\mcl{V}\to \mcl{W}$ with $\mcl{W}$ lying in $\Rep(\mcl{V})_{\rm dist}$.  Let $A$ denote $\mcl{W}$, considered as an algebra object in $\Rep(\mcl{V})_{\rm dist}$.  Then the category $\rep(\mcl{W})_{\rm dist}$ admits a vertex tensor structure, and restriction along $\iota$ provides an identification of braided monoidal categories $\rep(\mcl{W})_{\rm dist}=\rep^0(A)$.
\end{theorem}

Having fixed a commutative algebra object $A$ in $\Rep(\mcl{V})_{\rm dist}$, we let $\Rep^0(\mcl{V})_{\rm dist}$ denote the M\"uger centralizer of $A$ in $\Rep(\mcl{V})_{\rm dist}$.  We then have the free module functor $A\ot-:\Rep^0(\mcl{V})_{\rm dist}\to \Rep^0(A)$.  This functor is braided monoidal, and is left adjoint to the restriction functor $\Rep^0(A)\to \Rep(\mcl{V})_{\rm dist}$ \cite[Theorem 1.6]{kirillovostrik02}.  Taking this fact, and Theorem \ref{thm:A} into account, we observe a braided monoidal functor
\begin{equation}\label{eq:1626}
I_{\mcl{V}}^{\mcl{W}}=\mcl{W}\ot-:\Rep^0(\mcl{V})_{\rm dist}\to \Rep(\mcl{W})_{\rm dist}
\end{equation}
to the category of $\mcl{W}$-modules, for any VOA extension $\mcl{V}\to \mcl{W}$ with $\mcl{W}$ in $\Rep(\mcl{V})_{\rm dist}$.
\par

Note that when the category $\rep(\mcl{V})_{\rm dist}$ is rigid, the tensor product $\ot$ is necessarily biexact and commutes with colimits \cite[Proposition 2.1.8]{bakalovkirillov01}.  It follows that the induction functor $I_{\mcl{V}}^{\mcl{W}}$ is exact whenever $\rep(\mcl{V})_{\rm dist}$ is rigid.

\begin{lemma}\label{lem:ind}
Consider $\mcl{V}$ and $\mcl{W}$ as above.  Suppose that $\mcl{W}$ lies in the M\"uger center of $\rep(\mcl{V})_{\rm dist}$, that the category $\rep(\mcl{V})_{\rm dist}$ is rigid, and that the free modules $\mcl{W}\ot L$ are of finite length over $\mcl{W}$ for each simple module $L$ in $\rep(\mcl{V})_{\rm dist}$.  Then induction provides an exact braided monoidal functor
\begin{equation}\label{eq:1633}
I_{\mcl{V}}^{\mcl{W}}:\rep(\mcl{V})_{\rm dist}\to \rep(\mcl{W})_{\rm dist}
\end{equation}
which is left adjoint to restriction.
\end{lemma}

\begin{proof}
Given that $I_{\mcl{V}}^{\mcl{W}}$ is exact, and that $I_\mcl{V}^{\mcl{W}}(L)$ is of finite length for each simple $L$, we have by induction on the lengths of objects that $I_\mcl{V}^{\mcl{W}}$ restricts to a functor \eqref{eq:1633} between the subcategories of finite length objects.
\end{proof}

This should be contrasted with the more familiar case where both $\mcl{V}$ and $\mcl{W}$ are strongly-finite. In this case, $\mcl{W}$ won't lie in the M\"uger center (unless $\mcl{W}=\mcl{V}$), and induction lands in Rep$(A)$ rather than its subcategory Rep$^0(A)$.

\section{Rational (de-)equivariantization, again}
\label{sect:appendix}

We elaborate on the presentation of \cite[\S 8, Appendix]{negron}, which itself is an elaboration on the presentations of \cite{arkhipovgaitsgory03,davydovetingofnikshych18}, in order to clarify some of the mechanics in the proof of Theorem \ref{thm:singlet}.  Our aim is to clarify certain points about de-equivariantization and equivariantization along isomorphic group actions and isomorphic central embeddings. For the VOA theorist, equivariantization and de-equivariantization are the two steps which when combined capture tensor-categorically the orbifold construction of VOAs.

\subsection{Equivariantization and equivariant tensor functors}
\label{sect:A1}

Consider $\mbb{G}$ an affine algebraic group and write $R=\O(\mbb{G})$ for the algebra of global functions on $\mbb{G}$, viewed as a Hopf algebra in $Vect$ with comultiplication $\Delta$ and counit $\epsilon$.  We recall, from \cite{arkhipovgaitsgory03,davydovetingofnikshych18,negron}, that a rational action of an affine algebraic group $\mbb{G}$ on a tensor category $\msc{C}$, or more generally $\mbb{C}$-linear monoidal category, is a triple $\phi=(\phi,\nabla,\varepsilon)$ of an exact monoidal functor $\phi:\msc{C}\to \msc{C}_R$ to the base change \cite[\S 8.2]{negron} along with $R$-linear natural transformations $\nabla:\phi\to \phi^2$ and $\varepsilon:\phi\to id_\msc{C}$ whose adjoint maps provide isomorphisms $\Delta_\ast\phi\overset{\cong}\to \phi^2$ and $\epsilon_\ast\phi\overset{\cong}\to id_\msc{C}$ \cite[\S 8.1]{negron}.  So, an action of $\mbb{G}$ on $\msc{C}$ is a choice of a sufficiently structured comonad for $\msc{C}$ (or rather, the Ind-category $\operatorname{Ind}\msc{C}$).  We write
\[
\phi:\mbb{G}\to \underline{\Aut}_\ot(\msc{C})
\]
for a particular choice of $\mbb{G}$-action on $\msc{C}$.
\par

An isomorphism between two actions $\phi,\phi':\mbb{G}\rightrightarrows \underline{\Aut}_\ot(\msc{C})$ is a natural, $R$-linear isomorphism of monoidal functors $\eta:\phi\overset{\cong}\to \phi'$ which forms the appropriate diagrams with the structure maps $\nabla,\ \nabla',\ \varepsilon$, and $\varepsilon'$.  Also, given two categories $\msc{C}$ and $\msc{D}$ equipped with actions $\phi:\mbb{G}\to \underline{\Aut}_\ot(\msc{C})$ and $\psi:\mbb{G}\to \underline{\Aut}_\ot(\msc{C})$, a $\mbb{G}$-equivariant structure on a tensor functor $F:\msc{C}\to \msc{D}$ is a choice of $R$-linear isomorphism of monoidal functors $\tau:F\phi\to \psi F$ which again forms the appropriate diagrams with the structure maps for $\phi$ and $\psi$.
\par

Recall, finally, that the equivariantization $\msc{C}^\mbb{G}=\msc{C}^{\mbb{G},\phi}$ of a tensor category, relative to some $\mbb{G}$-action $\phi$, is the non-full subcategory of objects $V$ in $\msc{C}$ which are equipped with coassociative, counital, coaction $\rho:V\to \phi(V)$.  Rather, this is the category of comodules over the comonad $\phi$.  We have the following basic lemma.

\begin{lemma}\label{lem:A1}
(a) If two actions $\phi,\phi':\mbb{G}\rightrightarrows \underline{\Aut}_\ot(\msc{C})$ on a tensor category $\msc{C}$ are isomorphic, via some isomorphism $\eta:\phi\overset{\cong}\to \phi'$, then $\eta$ induces an equivalence of categories $\eta_\ast:\msc{C}^{\mbb{G},\phi}\to \msc{C}^{\mbb{G},\phi'}$.
\par

(b) If $F:\msc{C}\to \msc{D}$ is a $\mbb{G}$-equivariant tensor functor, with comparison transformation $\tau:F\phi\overset{\cong}\to \psi F$, then there is an induced monoidal functor between equivariantizations
\[
F^\mbb{G}:\msc{C}^\mbb{G}\to \msc{D}^\mbb{G},\ \ (V,\rho)\mapsto \left(FV,\tau F\rho\right).
\]
Furthermore, if $F$ is an equivalence then $F^\mbb{G}$ is an equivalence.
\end{lemma}

\begin{proof}
(a) The equivalence $\eta_\ast$ sends a $\phi$-comodule $(V,\rho_V)$ to the $\phi'$-comodule $(V,\eta_V\rho_V)$.  One similarly constructs the inverse $\eta_\ast^{-1}$ via $\eta^{-1}$.  (b) We only speak to the second point, and so assume $F$ is an equivalence.  It is clear that fully faithfulness of $F$ implies fully faithfulness of $F^\mbb{G}$.  Similarly, essential surjectivity of $F$ implies essential surjectivity of $F^\mbb{G}$.  Indeed, if an object $W$ in $\msc{D}$ is isomorphic to some $F(V)$, and $W$ admits a $\psi$-comodule structure $\rho_W$, then $F(V)$ admits an isomorphic $\psi$-comodule structure, which then induces a corresponding $\phi$-comodule structure on $V$ via $\tau$.  So we see that $(W,\rho_W)$ is in the essential image of $F^\mbb{G}$.
\end{proof}

Given a group map $\mbb{T}\to \mbb{G}$, with dual Hopf algebra map $w:\O(\mbb{G})\to \O(\mbb{T})$, and an action $(\phi,\nabla,\varepsilon):\mbb{G}\to \underline{\Aut}_\ot(\msc{C})$ we pull back to get an action $\mbb{T}\to \underline{\Aut}_\ot(\msc{C})$ with structure maps
\[
(\phi,\nabla,\varepsilon)|_\mbb{T}:=(w_\ast\phi:\msc{C}\to \msc{C}_S,\ w_\ast\phi \to (w_\ast \phi)^2,\ w_\ast\phi\to id).
\]
Here $R=\O(\mbb{G})$ and $S=\O(\mbb{T})$, considered as Hopf algebras in $Vect$, and $w_\ast:\msc{C}_R\to \msc{C}_S$ is the base change functor.
%In constructing the above structure map $w_\ast\phi\to (w_\ast\phi)^2$ one notes that the unit of the restriction base-change adjunction provides a map $\phi^2\to (w_\ast\phi)^2$, and that the composite $\phi\to (w_\ast\phi)^2$ has a uniquely defined adjoint map $w_\ast\phi\to (w_\ast\phi)^2$.

\subsection{Actions by ribbon and braided automorphisms}

Let $\msc{C}$ be a ribbon tensor category.  Then the braiding and twist on $\msc{C}$ induces a unique braiding and twist on the base change $\msc{C}_R$ so that the base change functor $R\ot_\mbb{C}-:\msc{C}\to \msc{C}_R$ is braided monoidal functor which commutes with the twist.  We say an action $\mbb{G}\to \underline{\Aut}_{\ot}(\msc{C})$ is an action by ribbon automorphisms if the associated monoidal functor $\phi:\msc{C}\to \msc{C}_R$ is braided and commutes with the twist.  Note that preservation of the ribbon structure, for a given action, is a property not an additional structure.
\par

One can check that, when $\mbb{G}$ acts by ribbon automorphisms the de-equivariantization $\msc{C}^\mbb{G}$ admits a unique ribbon structure so that the forgetful functor $\msc{C}^\mbb{G}\to \msc{C}$ is a map of ribbon tensor categories.  One also observes that the class of ribbon actions $\mbb{G}\to \underline{\Aut}_{\ot}(\msc{C})$ is closed under isomorphism.
\par

Of course, we can omit the twist to obtain the appropriate notion of a rational group action by braided automorphisms.  If a group acts $\mbb{G}\to \underline{\Aut}_{\ot}(\msc{C})$ by braided automorphisms then the equivariantization $\msc{C}^\mbb{G}$ inherits a unique braiding so that the forgetful functor is a braided tensor functor.

\subsection{De-equivariantization}
\label{sect:A2}

We again follow \cite{arkhipovgaitsgory03,davydovetingofnikshych18,negron}.  Consider a braided tensor category $\msc{C}$ with a M\"uger central tensor functor $i:\rep\mbb{G}\to \msc{C}$, where $\mbb{G}$ is an affine algebraic group.  Let $\O$ denote the regular (co)representation $\O(\mbb{G})$ in $\Rep\mbb{G}=\operatorname{Corep}\O(\mbb{G})$.  The de-equivariantization $\msc{C}_\mbb{G}=\msc{C}_{\mbb{G},i}$ is defined as the category of finitely presented $i\msc{O}$-modules in $\msc{C}$. 
\par

The de-equivariantization $\msc{C}_\mbb{G}$ is an exact $\mbb{C}$-linear monoidal category, with product $\ot_{i\O}$, and exact structure induced by the faithful inclusion $\msc{C}_\mbb{G}\to \operatorname{Ind}\msc{C}$.  In most cases, the de-equivariantization $\msc{C}_\mbb{G}$ will actually be an \emph{abelian} subcategory in $\operatorname{Ind}\msc{C}$.  Rigidity of the monoidal structure, however, should only hold when $i$ is an embedding, i.e.\ when the image of $\rep\mbb{G}$ in $\msc{C}$ is closed under taking subquotients (see \cite[\S 5.1, \S 5.2]{negron}).
\par

For any M\"uger central functor $\rep\mbb{G}\to\msc{C}$ as above we have the de-equivariantization functor, or free module functor $dE:\msc{C}\to \msc{C}_\mbb{G}$, $V\mapsto i\O\ot V$.  The category $\msc{C}_\mbb{G}$ admits a unique braiding so that the de-equivariantization map is a map of braided monoidal categories.  When $\msc{C}$ is additionally ribbon, and the twist acts trivially on the image of $\rep\mbb{G}$, then $\msc{C}_\mbb{G}$ also inherits a twist.  This twist gives $\msc{C}_\mbb{G}$ a ribbon structure, in the event that the de-equivariantization is indeed rigid.
\par

In the statement of the following lemma, by ``the" fiber functor for $\rep\mbb{G}$ we mean the forgetful functor $\rep\mbb{G}\to Vect$.  We note that there is only one fiber functor for $\rep\mbb{G}$ up to isomorphism, in any case \cite{delignemilne82}.

\begin{lemma}
Given any M\"uger central tensor functor $\rep\mbb{G}\to \msc{C}$, the composite $\rep\mbb{G}\to \msc{C}\to \msc{C}_\mbb{G}$ factors through the fiber functor $\rep\mbb{G}\to Vect\to \msc{C}_\mbb{G}$.
\end{lemma}

By factoring through the fiber functor, we mean that the two maps to $\msc{C}_\mbb{G}$ are isomorphic as monoidal functors.

\begin{proof}
The unit $Vect\to \msc{C}_\mbb{G}$ sends a vector space $V$ to the free module $\O\ot V$.  The composite $\rep\mbb{G}\to \msc{C}_\mbb{G}$ factors through the de-equivariantization map $\rep\mbb{G}\to (\rep\mbb{G})_\mbb{G}$ for $\rep\mbb{G}$ itself.  By the fundamental theorem of Hopf modules \cite[Theorem 1.9.4]{montgomery93} the de-equivariantization map for $\rep\mbb{G}$ is isomorphic to the composite $\rep\mbb{G}\to Vect \to (\rep\mbb{G})_\mbb{G}$ of the forgetful functor and the unit morphism for $(\rep\mbb{G})_\mbb{G}$.
\end{proof}

\begin{remark}
One could more generally consider de-equivariantization along central tensor functors from $\rep\mbb{G}$ to an arbitrary tensor category $\msc{C}$, in the sense of \cite[Definition 4.15]{dgno10}.  We stick to the braided setting for simplicity.
\end{remark}

\subsection{The rational $\mbb{G}$-action on $\msc{C}_\mbb{G}$, and $\rep\mbb{G}$-linear functors}
\label{sect:C_G}

As in the previous section, we consider a M\"uger central tensor functor $i:\rep\mbb{G}\to \msc{C}$ into a braided tensor category $\msc{C}$, and fix $\O$ the regular (co)representation in $\Rep\mbb{G}$.
\par

Take $R$ to be the image of $\O$ in $Vect$ under the forgetful functor $\rep\mbb{G}\to Vect$.  Of course, we have the unit $Vect\to \rep\mbb{G}$ so that $R$ is also an algebra object in the completion $\Rep\mbb{G}$.  One sees that the comultiplication on $\O(\mbb{G})$ provides $\O$ with a canonical $R$-comodule structure $\Delta:\msc{O}\to R\ot\msc{O}$ in $\Rep\msc{O}$.
\par

Now, the group $\mbb{G}$ acts naturally on the de-equivariantization $\msc{C}_\mbb{G}$ by braided automorphisms, and ribbon automorphisms when in the ribbon context.  This action is specified by the functor $\phi(M):=iR\ot M$, where $M$ is an $\O$-module in $\operatorname{Ind}\msc{C}$ and $\O$ acts on $iR\ot M$ via the coaction $\Delta:i\O\to iR\ot i\O$.  Coassociativity of the $R$-coaction on $\O$, and the counit, provide the necessary transformations $\phi\to \phi^2$ and $\phi\to id$.  The following lemma is straightforward.

\begin{lemma}\label{lem:A2}
If the M\"uger central embeddings $i,i':\rep\mbb{G}\to \msc{C}$ are isomorphic, with chosen monoidal isomorphism $\eta:i\overset{\cong}\to i'$, then the de-equivariantizations admit a $\mbb{G}$-equivariant braided equivalence
\[
\msc{C}_{\mbb{G},i'}\to \msc{C}_{\mbb{G},i},\ \ M\mapsto M|_{i\O},
\]
where we restrict along the algebra map $\eta:i\O\to i'\O$, and the isomorphism between the associated $\mbb{G}$-actions is given by $\eta\ot id:\phi(M)=iR\ot M\to i'R\ot M=\phi'(M)$. 
\end{lemma}

Let us call a braided category $\msc{C}$, with a fixed central tensor functor $i:\rep\mbb{G}\to \msc{C}$, a tensor category \emph{over} $\rep\mbb{G}$.  By a map of tensor categories over $\rep\mbb{G}$ we mean a tensor functor $F:\msc{C}\to \msc{D}$ and a specific choice of natural isomorphism  of tensor functors $l:j\overset{\cong}\to Fi$ which makes the diagram
\[
\xymatrix{
 & \rep\mbb{G}\ar[dl]_i\ar[dr]^j\\
\msc{C}\ar[rr]^F & & \msc{D}
}
\]
commute. In this case we have immediately a $\mbb{G}$-equivariant functor between the de-equivariantizations
\[
\msc{C}_{\mbb{G}}\to \msc{D}_{\mbb{G},Fi},\ \ M\mapsto FM,
\]
and Lemma \ref{lem:A2} provides a $\mbb{G}$-equivariant equivalence $\msc{D}_{\mbb{G},Fi}\overset{\cong}\to \msc{D}_{\mbb{G},j}$.  Thus we have

\begin{lemma}\label{lem:A4}
Any map $F=(F,l):\msc{C}\to \msc{D}$ of tensor categories over $\rep\mbb{G}$ induces a $\mbb{G}$-equivariant braided monoidal functor
\[
F_\mbb{G}:\msc{C}_\mbb{G}\to \msc{D}_\mbb{G},\ \ M\mapsto FM|_{j\O}.
\]
\end{lemma}

We have the following general version of Proposition \ref{prop:basechange}.

\begin{lemma}\label{lem:A5}
Suppose a tensor category $\msc{D}$ is equipped with the trivial $\rep\mbb{G}$ structure, $\rep\mbb{G}\to Vect\to \msc{D}$, and $i:\rep\mbb{G}\to \msc{C}$ is an arbitrary tensor category over $\msc{C}$.  Suppose also that the de-equivariantization $\msc{C}_\mbb{G}$ is a (rigid) tensor category.  Then any map $F:\msc{C}\to \msc{D}$ of tensor categories over $\rep\mbb{G}$ admits a factorization $f:\msc{C}_\mbb{G}\to \msc{D}$ defined by $f=\1\ot_{F\O}F(-)$.
\end{lemma}

\begin{proof}
Same as the proof of Proposition \ref{prop:basechange}.
\end{proof}

We note, in closing, that the implicit isomorphism $l:Fi\cong forget$ specifies an algebra isomorphism $F\O\cong R$, where $R=forget\O$.  This isomorphism specifies a particular point (algebra map) $F\O\cong R\overset{\epsilon}\to \1$, and is in fact \emph{specified by} its associated point $F\O\to \1$.  It is precisely at this point for $F\O$ at which we take the fiber to obtain the factorizing map $f=\1\ot_{F\O}F(-)$.

%\subsection{Maps between Tannakian categories}

%\begin{lemma}
%Let $G$ and $G'$ be affine algebraic groups (over an algebraically closed field).  Any symmetric tensor functor $F:\rep(G)\to \rep(G')$ is isomorphic to the restriction functor $\res_\alpha:\rep(G)\to \rep(G')$ along a group map $\alpha:G'\to G$. 
%\end{lemma}

%\begin{proof}
%Let $w:\rep(G)\to Vect$ be the forgetful functor and $w':\rep(G)\to Vect$ be the composite of $F$ with the forgetful functor functor $\rep(G')\to Vect$.  We claim that the identity $id:\rep(G)\to \rep(G)$ is isomorphic to a tensor equivalence $\eta:\rep(G)\to \rep(G)$ with $w=w'\eta$.
%\par

%First note that the fiber functor $w'$ lifts to a symmetric tensor isomorphism $W':\rep(G)\to \rep(A)$, $V\mapsto w'(V)$, where $A=\Aut^{\ot}(w')$, and suffices to show that this map is isomorphic to restriction along some $\alpha:A\to G$.  Take $\sigma:w\to w'$ any isomorphism (which exists by uniqueness of the fiber functor) and define $\alpha=\operatorname{Ad}_\sigma:G=\Aut^{\ot}(w)\overset{\cong}\to A$.  Then we have the composite
%\[
%\eta:=\res_\alpha\circ W':\rep(G)\to \rep(A)\to \rep(G)
%\]
%with $w\eta=w'$ and natural maps $\sigma:V\to \res_\alpha(w'(V))$ which provide a natural isomorphism $id_{\rep(G)}\cong \eta$.
%\end{proof}

\bibliographystyle{abbrv}
%\bibliography{wp}

\end{document}